\documentclass{amsart}
\usepackage{bbm, amssymb}
\newtheorem{theorem}{Theorem}[section]
\newtheorem{lemma}[theorem]{Lemma}
\newtheorem{corollary}[theorem]{Corollary}
\newtheorem{proposition}[theorem]{Proposition}

\theoremstyle{definition}
\newtheorem{remark}[theorem]{Remark}
\newtheorem{example}[theorem]{Example}

\newcommand{\I}{{\mathbbm{1}}}
\newcommand{\M}{\mathcal{M}}
\newcommand{\A}{\mathcal{A}}
\newcommand{\R}{\mathcal{R}}
\newcommand{\B}{\mathcal{B}}
\newcommand{\D}{\mathcal{D}}
\newcommand{\N}{\mathcal{N}}
\newcommand{\Rdb}{{\mathbb R}}
\newcommand{\Edb}{{\mathbb E}}

\newcommand{\Cdb}{{\mathbb C}}
\newcommand{\Tdb}{{\mathbb T}}
\newcommand{\Pdb}{{\mathbb P}}

\newcommand{\Zdb}{{\mathbb Z}}

\begin{document}

\title[Conditional expectations and generalized representing measures]{Von Neumann algebra conditional expectations with applications to generalized representing measures for noncommutative function algebras}
\author{David P. Blecher}
\address{Department of Mathematics, University of Houston, Houston, TX
77204-3008}
\email[David P. Blecher]{dpbleche@central.uh.edu}
\author{Louis E. Labuschagne}
\address{DSI-NRF CoE in Math. and Stat. Sci, Pure and Applied Analytics,\\ Internal Box 209, School of Math. \& Stat. Sci.,
North-West University, PVT. BAG X6001, 2520 Potchefstroom, South Africa}
\email{louis.labuschagne@nwu.ac.za}

\thanks{DB is supported by a Simons Foundation Collaboration Grant (527078). LL was supported by the National Research Foundation KIC Grant 171014265824. Any opinion, findings and conclusions or recommendations expressed in this material, are those of the author, and therefore the NRF do not accept any liability in regard thereto. LL was also the recipient of a UCDP grant for the period 2019-2020 which indirectly helped to fund a visit by the first author where some of this work was done.}


\begin{abstract}  We establish several deep existence criteria for conditional expectations on von Neumann algebras, and then apply this theory to develop a noncommutative theory of representing measures of characters of a function algebra. Our main cycle of results describes what may be understood as a `noncommutative Hoffman-Rossi theorem' giving the existence of weak* continuous  `noncommutative representing measures' for so-called $\D$-characters. These results may also be viewed as `module' Hahn-Banach extension theorems for weak* continuous `characters' into possibly noninjective von Neumann algebras. In closing we introduce the notion of `noncommutative Jensen measures', and show that as in the classical case representing measures of 
logmodular algebras are Jensen measures. The proofs of the two main cycles of results rely on the delicate interplay of Tomita-Takesaki theory, noncommutative Radon-Nikodym derivatives, Connes cocycles, Haagerup noncommutative $L^p$-spaces, Haagerup's reduction theorem, etc. 
\end{abstract}

\maketitle

\section{Introduction}  Recall that if $\A \subset C(K)$ is a function algebra or uniform algebra on a compact space $K$,
then the primary associated object is the set $M_{\A}$ of (scalar valued)  {\em characters} on $\A$.  
These are the nonzero scalar valued homomorphisms on $\A$.   The basic texts
on function algebras spend considerable time on the theory and applications of these characters, which are fundamental 
in this branch of function theory (see e.g.\ \cite{Gam, Stout}).   In \cite{BLv}, inspired by Arveson's
seminal paper \cite{Arv},  we proposed a new generalization of characters
applicable to operator algebras.   By this last expression we mean 
a subalgebra $\A$ of a $C^*$-algebra or von Neumann algebra $\M$.   
In translating the classical 
theory of characters $\chi : \A \to \Cdb$ we may identify the range of the character with $\D = \Cdb 1_{\A} = \Cdb 1_{\M}$, 
in which case the homomorphism becomes an idempotent map 
on $\A$, and is a $\D$-bimodule map.  Thus  in the 
operator algebra case we suppose that we have a $C^*$- (indeed in this paper usually von Neumann) subalgebra $\D$ of $\A$ and of $\M$, and 
we define a $\D$-{\em character} to be a unital 
contractive homomorphism  $\Phi : \A \to \D$.    We also assume that
 $\Phi$  is a $\D$-bimodule map (or equivalently, is the identity map on $\D$).  
The basic theory of $\D$-characters is developed 
 in \cite{BLv}.   One motivation for our study of such maps comes from \cite{Arv}, where Arveson gives many important 
examples of $\D$-characters in his approach to noncommutative analyticity/generalized analytic functions/noncommutative
Hardy spaces.   In this
`generalized analytic function theory', and in certain other operator algebraic 
situations,  it is very important that the `noncommutative character' $\Phi$ is $\D$-valued rather than merely 
$B(H)$-valued (and the same comment holds for the `noncommutative representing measures'  
 for a $\D$-character discussed below).

In the present paper we  
simultaneously generalize important aspects of
the theory of conditional expectations of von Neumann algebras, and the theory of representing measures of 
characters of  function algebras.  To describe the first topic we begin by recalling some fundamental 
 probability theory.  Let   $(K,\A,\mu)$ be a probability 
 measure space.    
`Essentially order-closed' sub-$\sigma$-algebras $\B$ of  $\A$ correspond (up to null sets) to von Neumann subalgebras 
 $\D$ of $\M = L^\infty(K,\A,\mu)$ 
 (see e.g.\ Theorem 7.27 (iii) and (iv) in \cite{BGL}).  
  For any such subalgebra 
 $\D$ there exists a unique conditional expectation $E : \M \to \D$ such that $\int \, E (f)  \, d \mu = \int \, f   \, d \mu$ for 
 $f \in \M$.   Also $E$ is 
 weak* continuous (or equivalently, in von Neumann algebra language, {\em normal}), and {\em faithful} (that is, 
 $E(f) = 0$ implies $f = 0$ if $f \in \M_+$). 
In the von Neumann algebra context we will by a conditional expectation mean a 
(usually normal) unital positive (or equivalently, by p.\ 132--133 in \cite{Bla}, contractive)  idempotent $\D$-module map $E : \M \to \D$
 from a $C^*$-algebra onto its $C^*$-subalgebra $\D$.  
 The unique conditional expectation for $\mu$ above will be called the 
 {\em weight-preserving conditional expectation}.   A better name might be 
  `probabilistic conditional expectation', but much of our paper is in the setting of 
   `weights' as opposed to `states'.    If we 
 write it as $\Edb_\mu$ then  any other normal conditional expectation $E : \M \to \D$ is given by $E(x) = \Edb_{\mu}(hx)$ for a 
 density function $h \in L^1(K,\A,\mu)_+$ with $\Edb_{\mu}(h) = 1$. 
 Such $h$ is sometimes called a {\em weight function}.  See 
 e.g.\ \cite[Section 6]{Bposx} for references. 
 We shall show that a similar structure pertains in the noncommutative context.
 
 Turning to von Neumann algebras and noncommutative integration theory, 
recall that here von Neumann algebras are regarded as noncommutative $L^\infty$-spaces. So by fixing a  faithful 
normal semifinite (often abbreviated as {\em fns}) weight $\omega$ on $\M$, we may view the pair $(\M,\omega)$ as 
a structure encoding a noncommutative measure space. (We will define the notions of faithful, normal, 
semifinite, and weight in Section \ref{Sec1} 
below. For now recall that every von Neumann algebra has a fns weight \cite[Theorem VII.2.7]{tak2}.) For an 
inclusion $\D \subset \M$ of von Neumann algebras,  the existence of a normal conditional expectation 
$E : \M \to \D$ is a rich subject, and is often a very tricky matter.   The best result is as follows 
\cite[Theorem IX.4.2]{tak2}: If  $\omega$ above restricts to a semifinite  weight on $\D$ then there 
exists a (necessarily unique)  $\omega$-preserving  normal conditional expectation $E : \M \to \D$ if and only if 
the modular group $(\sigma^{\omega}_t)_{t \in \Rdb}$ of  $\omega$ (see  \cite[Chapter  VIII]{tak2}) leaves 
$\D$ invariant. By $\omega$-preserving 
we here mean that $\omega \circ E = \omega$ on $\M$.
 We will refer 
  to this unique conditional expectation as the {\em weight-preserving conditional expectation}, and often write it as $\Edb_\omega$ 
 or as $E_{\D}$.  
  This is of course  the noncommutative analog of the object with the same name mentioned in the last paragraph.
 As a special case, if $\omega$ restricts to a semifinite trace on $\D$ then there 
 exists an $\omega$-preserving  normal conditional expectation $E : \M \to \D$ if and only if  the modular group of 
 $\omega$ acts as the identity on $\D$, or equivalently if $\D$ is contained in the {\em centralizer} (defined below) 
 of $\omega$.   If in addition $\omega$ is a normal state on $\M$, then this centralizer condition is saying that
  $$\omega(dx) = \omega(xd) , \qquad x \in \M, d \in \D .$$   We make a few contributions to the general 
   theory of conditional expectations of von Neumann algebras in 
the early sections of our paper.   See e.g.\ Theorems \ref{wcent},  
\ref{locis}, \ref{excodsfex},
  \ref{wcentnonf}, and their corollaries.  Also in Section \ref{Gcex} we develop the noncommutative 
version of the `weight function' mentioned in the last paragraph for a 
classical   conditional expectation.   See the last two sections of \cite{Bposx} for some more  facts and background about 
  the above view of classical and von Neumann algebraic conditional expectations (and for a brief survey of some aspects of the present paper).

 We now turn to  representing measures for characters  of a function algebra $\A \subset C(K)$.
 A positive measure $\mu$  on  $K$ is called a \emph{representing measure} for a character $\Phi$  of   $\A$  if
$\Phi(f)=\int_K \, f \,d\mu$ for all $f\in A$. The functional $\widetilde{\Phi}(g) = \int_K \, g \, d\mu$ on $C(K)$ is a state on $C(K)$,
and indeed representing measures for $\Phi$ are in a bijective correspondence with the extensions of $\Phi$ 
to positive functionals on $C(K)$.  That is, representing measures for  a character are just the Hahn-Banach extensions to $C(K)$ of the 
character.   The classical {\em Hoffman-Rossi theorem} proves the existence of `normal' or `absolutely continuous'  representing measures': 
weak* continuous characters  of a function algebra $\A \subset \M = L^\infty(X,\mu)$ have 
weak* continuous  positive Hahn-Banach extensions to $\M$.   Such extensions correspond to measures on $X$ which are absolutely continuous
with respect to $\mu$.    The original Hoffman-Rossi theorem assumes that $\mu$ is a probability measure, but actually 
any positive measure will do (see  \cite{BFZ} for a swift proof of this result).  
In several sections of the present paper we prove noncommutative versions of the Hoffman-Rossi 
theorem; that is, we establish the existence of `noncommutative representing measures'  $\Psi : \M \to \D$ 
for weak* continuous $\D$-characters on a weak* closed subalgebra $\A$ of $\M$, where by a noncommutative 
representing measure we mean the extension of a $\D$-character to a positive $\D$-valued map on all of $\M$ 
satisfying certain regularity conditions.  
 These may be viewed as `module' Hahn-Banach extension theorems for weak* 
continuous $\D$-characters 
into possibly noninjective von Neumann algebras.   It is not at all clear that `Hahn-Banach extensions' into $\D$ exist, 
let alone weak* continuous
ones, in fact this is quite surprising. 
In particular it is important that $\A$ is an algebra and $\Phi$ is a homomorphism for such positive weak* continuous extensions to exist. 
 In \cite{BFZ} the first author and coauthors proved such a theorem 
which works for all von Neumann algebras $\M$,
but requires the very strong condition that the subalgebra $\D$ 
be purely atomic. 
At the end of that paper it was mentioned that we were still pursuing 
 the special case of the noncommutative Hoffman-Rossi theorem in the case 
that $\M$ has a faithful normal tracial state 
and $\D$ is any von Neumann subalgebra of $\A$.    In 
 the present paper we supply this  noncommutative generalization of the Hoffman-Rossi theorem, indeed we 
 generalize to subalgebras of much larger classes of von Neumann algebras.    

To be more specific, in Sections  \ref{HR1}--\ref{repm} we consider inclusions $\D \subset \A \subset \M$ for a
weak* closed subalgebra $\A$, and a von Neumann subalgebra  
 $\D$ of our von Neumann algebra $\M$, together with a weak* continuous $\D$-character $\Phi : A \to \D$.   
(As we said, it is usually crucial in our results here that $\A$ is an algebra, as is the case in Hoffman and Rossi's original 
theorem, but in a few  results 
 $\A$ can be allowed to be a $\D$-submodule of $\M$ containing $\D$.) 
We as before view $\M$ as a structure encoding a noncommutative  measure space by fixing a faithful normal 
semifinite weight $\nu$ on $\M$.  We then seek to  find a  (usually normal) 
extension of $\Phi$ to a positive  map $\Psi : \M \to \D$.   Such extensions $\Psi$ will be our `noncommutative 
representing measures'. Note that  $\Psi$ is necessarily a 
conditional expectation since it is the identity map on $\D$ and hence is idempotent and is a 
$\D$-module map (see p.\ 132--133 in \cite{Bla} for the main facts about conditional expectations and their
relation to bimodule maps and projections maps of norm $1$). 
Thus our setting in Sections  \ref{HR1}--\ref{repm}  simultaneously generalizes the classical setting above (the case that 
$\M$ is commutative and $\D = \Cdb 1$) and the setting for  von Neumann algebraic conditional expectations  (the case $\A = \D$).
We will describe our specific results in more detail below.   It also generalizes, as we said above,  Arveson's profound noncommutative abstraction 
from \cite{Arv,BLsurv} of Hardy space theory on the disk.   In   particular it is very important in Arveson's theory that the `noncommutative
representing measures' take values in $\D$, and {\em not merely} in the ambient $B(H)$. 
We remark that it should prove fruitful to relate our theory  to the noncommutative Choquet theory of e.g.\ \cite{DKen}.

We now describe the structure of our paper.   In Section \ref{Sec1}  we state and prove some  results in 
noncommutative integration theory, 
some of which may be folklore.
We also state some other general facts about von Neumann algebras, noncommutative $L^p$ spaces, and extensions
of conditional expectations to the latter.   In Section \ref{ncec} we develop some aspects (which we 
were not able to find in the von Neumann algebra literature) of 
the relation between normal conditional expectations and centralizers of normal states and weights.  For example,
we show how from any given state on $\M$, one may use an averaging technique to construct a related normal state for which $\D$ is in the centralizer of that state.   In Section \ref{Gcex}, which 
is also purely von Neuman algebraic, 
we develop the noncommutative theory of the `weight function' mentioned early in our Introduction,
and give  existence criteria and characterizations
of  normal conditional expectations in terms of such `weight functions' $h$.  
This uses the delicate interplay of Tomita-Takesaki theory, noncommutative Radon-Nikodym derivatives, Connes cocycles, 
Haagerup $L^p$-spaces, etc.

In Sections  \ref{HR1}--\ref{repm} we
prove the noncommutative Hoffman-Rossi theorems described above, that is the existence of `noncommutative representing measures'.
The proof of such theorems involve three stages. In the first stage we show that $\omega \circ \Phi$ extends to a normal state on $\M$.
 In the second stage we show that we may assume that 
 this  normal state has $\D$ in its centralizer, by  averaging over the unitary group of $\D$.   Finally in the third stage we use some of the technology 
 of conditional expectations  developed earlier to prove the main result.
Indeed as we hinted above, such states with $\D$ in its centralizer possess associated
 $\D$-valued normal conditional expectations.
 
 More specifically in Section  \ref{HR1} we develop our main 
ideas for proving the noncommutative Hoffman-Rossi theorem, and as an example apply these ideas to prove this 
result for all inclusions $\D \subset \A \subset \M$ where $\M$ is commutative.  This 
generalizes the classical  Hoffman-Rossi theorem (in which $\D = \Cdb 1$).  In Section  \ref{HR2} we 
prove the theorem in the case that $\M$ is finite  or $\sigma$-finite (defined below), and $\D$ is contained in the centralizer. 
The  $\sigma$-finite case becomes quite technical, involving the technology of Haagerup $L^p$-spaces.  In Section  
\ref{repm} we provide several results giving the existence of `representing measures' for general von Neumann 
algebras.    The most general result one can hope for is that for inclusions $\D \subset \A \subset \M$ for which 
we know that there {\em does exist} a normal conditional expectation $E_\omega : \M \to \D$, 
every weak* continuous $\D$-character $\Phi$ on $\A$  extends to some normal conditional expectation 
$\Psi : \M \to \D$. We obtain several special cases of this.  In our most general settings in Section \ref{repm}, 
using the famous Haagerup reduction theorem \cite{HJX}  we are able to obtain a  conditional expectation 
$\Psi : \M \to \D$ extending $\Phi$, that is, a `noncommutative representing measure'. 
However at present we are in this generality not always able to show that $\Psi$ may be chosen to be normal. 

Finally in Section \ref{Jens}, we consider a class of `noncommutative representing measures' generalizing 
the Jensen measures and Arens-Singer measures in classical function theory \cite{Gam,Stout}. We prove a 
noncommutative variant of the classical result that if $\A$ is {\em logmodular} then every representing measure 
is a Jensen measure. The theory achieved therefore provides a well-rounded foundation for the development of a 
new noncommutative theory of function spaces.   

Turning to notation: We write $\mathbb{P}(\M)$ for the projection lattice of a von Neumann algebra $\M$, and $\M_+$ for the positive cone of 
$\M$, namely  $\{ x \in \M : x \geq 0 \}$.

In this paper an operator algebra
is a unital algebra $\A$ of operators on a Hilbert space, or more abstractly a Banach algebra isometrically isomorphic to  such an algebra of Hilbert space operators.   
Many authors consider 
an operator space structure on an operator algebra, and replace the word `isometrically' by 
`completely isometrically' in the last sentence--see e.g.\
\cite{BLM} for definitions--but this will not be important in the present paper.  Indeed in the introduction to 
\cite{BLv} we explain why $\D$-characters are automatically completely contractive. 
By 4.2.9 in \cite{BLM}, which is an earlier result of the present authors, 
any unital completely contractive projection of a unital operator algebra onto a subalgebra
$\D$ is a  $\D$-bimodule map.   Thus one may equivalently define the $\D$-characters above to be
the unital completely contractive projections of $\A$ onto $\D$ which are also homomorphisms.  
In this paper $\A$ is usually weak* closed, and $\D$-characters are usually weak*  continuous. 

We will sometimes abusively call a von Neumann algebra {\em finite} if it has a faithful normal tracial 
state (this is abuse, since in nonseparable situations it differs from the common usage meaning that the 
identity element is a finite projection).   We recall that a von Neumann algebra is {\em $\sigma$-finite} if it has a faithful normal 
state (this includes all  von Neumann algebras on separable Hilbert spaces), and {\em semifinite} if it has a faithful normal tracial weight.

\section{Preliminaries on weights on von Neumann algebras} \label{Sec1} A {\em weight} on a
 von Neumann algebra $\M$ is a $[0,\infty]$-valued map on $\M_+$ which is
additive, and satisfies $\omega(tx) = t \omega(x)$ for $t \in [0,\infty]$ (interpreting $0\cdot \infty = 0$). 
We use the standard notation $${\mathfrak p}_\omega  = \{ x \in \M_+ : \omega(x) < \infty \}, \; \; \; \; \; \; \; \; 
{\mathfrak n}_\omega = \{ x \in \M : \omega(x^* x) < \infty\},$$ and 
$${\mathfrak m}_\omega = {\rm Span} \, {\mathfrak p}_\omega =  {\mathfrak n}_\omega^* {\mathfrak n}_\omega .$$ 
If here we wish to emphasize that we are only considering elements of ${\mathfrak m}_\omega$ taken from a 
specific subset ${\mathcal E}$ of $\M$, we will then write\ ${\mathfrak m}({\mathcal E})_\omega$.  
The latter may be regarded as `the definition domain' of $\omega$: this is the part of $\M$ on which $\omega$ 
is naturally defined as a  linear functional (see around VII.1.3 in \cite{tak2}, 
or see e.g.\ 
\cite{Comb} for some of the early development of the theory of weights).  We say that $\omega$ is 
{\em normal} if it preserves suprema of bounded increasing nets in $\M_+$, {\em faithful} if  $\omega(x) = 0$ 
implies $x = 0$ for $x \in \M_+$, and {\em tracial} (or is a {\em trace}) if $\omega(x^* x) = \omega(x x^*)$ 
for $x \in \M$.  Any normal weight is the supremum of an increasing family of positive 
normal functionals (see \cite[Theorem VII.1.11]{tak2}). There are several equivalent definitions of semifiniteness 
for  weights: that ${\mathfrak n}_\omega$ is weak* dense in $\M$, that  ${\mathfrak m}_\omega$ is weak* dense in $\M$, 
that there exists an increasing net in ${\mathfrak p}$ with supremum 1, and that  ${\mathfrak p} = 
\{ x \in \M_+ : \omega(x) < \infty \}$ generates $\M$ as a von Neumann algebra. A weight is {\em strictly semifinite} if it is  the sum of a 
family of positive functionals with  mutually orthogonal supports, or equivalently that its  restriction to the 
centralizer (defined below) is semifinite. 
 
For a normal trace $\tau$ on a von Neumann algebra the $\tau$-finite projections are a directed set by \cite[Proposition V.1.6]{tak1}, 
and this gives a net with weak* limit $1$ if  $\tau$ is also semifinite. 

\begin{lemma} \label{istrD} Let $\omega$ be a normal semifinite weight  on 
a von Neumann algebra $\M$,  with $\omega(d_0d_1) =  \omega(d_1d_0)$ for $d_0, d_1 \in {\mathfrak m}_\omega$.  
Then $\omega$ is a trace on $\M$. \end{lemma}

\begin{proof}   
 Let $d \in \M$ be given with $\omega(d^*d)<\infty$.  That is $d\in {\mathfrak n}_\omega$.  By semifiniteness there must exist 
  $(d_t)\subset  {\mathfrak m}_\omega$ 
such that $d_t \nearrow 1$ in $\D$ with $\omega(d_t) < \infty$. Then since $d_t^{\frac{1}{2}}d  \in {\mathfrak m}_\omega$ and 
$\omega(d^*d) = \lim_t \, \omega(d^* d_t d)$, we have by 
hypothesis that $$\omega(d^*d) = \lim_t \, \omega(d_t^{\frac{1}{2}}d d^*d_t^{\frac{1}{2}})  = 
\lim_t \, \omega(d_t^{\frac{1}{2}} |d^*|^2 d_t^{\frac{1}{2}}) = \lim_t \, \omega(|d^*|\,d_t\,|d^*|)= \omega(dd^*), $$ 
which is finite. On replacing 
$d$ with $d^*$ the above also shows that if  $\omega(dd^*)<\infty$, then $\omega(d^*d)=\omega(dd^*)$. So  $\omega(d^*d)<\infty$ 
if   and only if   $\omega(dd^*)<\infty$, in which case they are equal. It follows that $\omega$ is a trace.   
\end{proof} 
    
The important data for a weight  $\omega$ is obviously ${\mathfrak p}_\omega$ and the values of the weight $\omega$ there. 
Any self-adjoint element $x\in \M$ of course generates a minimal abelian von Neumann subalgebra of $\M$, which we shall 
denote by $W^*(x)$. Given $x \in {\mathfrak p}_\omega$,  we may identify $W^*(x)$ with $L^\infty(X,\A,\mu)$ for a 
strictly localizable measure $\mu$, and $x$ with 
some $f \in L^\infty(X,\mu)$ \cite[Theorem 1.18.1]{Sakai}.   Any $f \in L^\infty(X,\mu)$ is an increasing limit of simple 
functions corresponding to projections in ${\mathfrak p}_\omega$. 
This carries over to the noncommutative case. Hence the important data for a normal weight $\omega$ is  the set of 
projections in ${\mathfrak p}_\omega$ and the values of $\omega$ on such projections. To see this note that for any 
$f\in\mathfrak{p}_\omega$ and any $\epsilon>0$, we will have that $\omega(e_\epsilon)\leq \epsilon^{-1}
\omega(f)<\infty$ where $e_\epsilon =\chi_{(\epsilon,\infty)}(f)$. The element $f$ may of course be written as an increasing limit of positive `Riemann sums' consisting 
of linear combinations of sub-projections of projections like $e_\epsilon$. By normality the value of 
$\omega$ on $f$ is determined by its value on those Riemann sums, thereby proving the claim.

The support projection $e = s(\omega)$ 
 of a normal weight  $\omega$ is the complement of the largest projection $q$ in $\M$ with $\omega(q) = 0$.
For any projection $q$ in $\M$ with $\omega(q) = 0$ we have that
 $q \in {\mathfrak m}_\omega$.  Also for any $x \in {\mathfrak p}_\omega$ we have $\omega(qx) = \omega(xq) =  \omega(qxq) = 0$ by 
the Cauchy-Schwarz inequality, hence $\omega(x) = \omega(exe) < \infty$.    
For any 
normal  weight $\omega$ on $\M$ there exist canonical projections $f \leq e = s(\omega)$ in $\M$ such that $\omega$ is faithful on 
$e \M e$, semifinite  on $f^\perp \M f^\perp$, and semifinite and faithful on
$(e-f) \M (e-f)$.    For this reason  authors often assume that normal  weights are faithful and semifinite.    We shall not usually do this though.

For a faithful normal 
semifinite weight $\omega$  on  a von Neumann algebra $\M$, the {\em centralizer} $\M_\omega$ of $\omega$ is the von Neumann subalgebra consisting of $a\in \M$ with $\sigma^\omega_t(a)=a$ for all $t\in \mathbb{R}$.   By \cite[Theorem VIII.2.6]{tak2}, if $\omega$ is a state then $$\M_\omega=\{a\in \M: \omega(ax)=\omega(xa), x\in \M\}.$$ 
For a faithful normal 
semifinite weight the centralizer is known to equal the set of $x \in \M$ with $x \mathfrak{m}_\omega \subset 
\mathfrak{m}_\omega$ and $\mathfrak{m}_\omega  \, x\subset \mathfrak{m}_\omega$,
and $\omega(xy) = \omega(yx)$ for all $y \in \mathfrak{m}_\omega$. 
If $\omega$ is semifinite but not necessarily faithful we will  say that an element $x \in \M$ is  $\omega$-{\em central} 
if  $x$ commutes with the support projection $e$ of $\omega$ and  satisfies the conditions
$$x \mathfrak{m}_\omega e \subset 
\mathfrak{m}_\omega \; , \; \; \; \;  \; \; \; \;  e \mathfrak{m}_\omega  \, x\subset \mathfrak{m}_\omega,$$
and $\omega(xy) = \omega(yx)$ for all $y \in \mathfrak{m}_\omega$.  Since this definition 
differs somewhat from competing definitions of $\omega$-centrality for non-faithful 
weights, we shall in this case write 
$\M^\omega$ for the $\omega$-central elements in $\M$.  
 In the applications later in the paper
the elements $x$ will also be in $\mathfrak{m}_\omega$, as is the case for example if $\omega$ is a state, 
so that one may drop the $e$'s from the last centered equation defining `$\omega$-central'.
See particularly Lemma \ref{tuse} below. 

If $\omega(q) = 0$ for a projection $q$ then $q$ is dominated by, 
and hence commutes with, the complement of the support projection $e$ of  $\omega$.  Since
$\mathfrak{m}_\omega$ is an algebra it follows that $\mathfrak{m}_\omega$ is left and right 
invariant under multiplication by $q$, and hence by $e$.  It follows 
that $q$ and $e$ are $\omega$-central, and are 
in $\M^\omega$.   
Indeed
 $e$  lies  in the center of $\M^\omega$ by the definition of the latter.  
It is easy to see that $\M^\omega$ is a unital $*$-subalgebra of $\M$.   In fact it is a von Neumann subalgebra 
(we remark that we were not able to see this without  including in the definition 
of $\M^\omega$  the commutation with  the support projection of $\omega$, which plays a key role in the proof below).

\begin{lemma}  \label{centvN}  Let $\omega$ be a normal semifinite weight on  a von Neumann algebra $\M$  with
 support projection $e$.   Then $\M^\omega$ is a von Neumann subalgebra of $\M$. 
 If $\N = e \M e$ then 
 $\N_\omega = \N \cap \M^\omega = e \M^\omega e$ .
  \end{lemma}
  
\begin{proof} We first show that $\N_\omega \subset \M^\omega$.  
Given $x \in \N_\omega, y \in  \mathfrak{m}_\omega$, then
$$xy e = ex e y e \in  
 {\mathfrak m}(\N)_\omega \subset  {\mathfrak m}(\M)_\omega.$$
 Similarly $e yx \in {\mathfrak m}(\M)_\omega.$ Moreover, $$\omega(xy) = \omega(ex e y  e) =  \omega(ey e x  e) =  \omega(yx) .$$
 Thus $\N_\omega \subset \M^\omega$.  
 
If  $d \in \M^\omega$ then  $d \mathfrak{m}(\M)_\omega e \subset \mathfrak{m}(\M)_\omega$
and $\omega(d x) = \omega(xd)$ for $x \in \mathfrak{m}(\M)_\omega$.
Then $\omega(ed xe) = \omega(exde)$.   If  $x \in \mathfrak{p}(\N)_\omega \subset \mathfrak{m}(\M)_\omega$,
then $edexe \in e \mathfrak{m}(\M)_\omega e \subset \mathfrak{m}(\N)_\omega$ 
and $\omega (edex) = \omega(dx) = \omega(xd) = \omega(x ede)$.  So $ede \in \N_\omega$.  

We said that $\M^\omega$ is a unital $*$-subalgebra of $\M$ with $e$ in its center.
Suppose that  $d_t \in \M^\omega$ with $d_t \to d$ weak* in $\M$.   Then $e d_t e \to ede$ weak*, and 
$e d_t e$ is in the von Neumann algebra $\N_\omega$, so that 
$e d e \in \N_\omega \subset \M^\omega$.   With $q = e^\perp$ we have  $q d_t e = 0$ so that $qde = 0$.
Similarly $e d q = 0$.  
Also  $q d q \in \mathfrak{m}(\M)_\omega$ with $\omega(q d q x) = 0 = \omega(x q d q)$ for $x \in \mathfrak{m}(\M)_\omega$, so   that 
$qdq \in \M^\omega$.   Thus $d \in \M^\omega$.   Hence  $\M^\omega$ is a von Neumann algebra.   The rest is clear.
 \end{proof}  

\begin{lemma}  \label{tuse} Let $\omega$ be a normal weight on  a von Neumann algebra $\M$ which is 
semifinite on a  von Neumann subalgebra $\D$, with $\omega(xd) = \omega(dx)$ for all
$d \in \mathfrak{m}(\D)_\omega$ and $x \in \mathfrak{m}(\M)_\omega$.   
  Then $\D \subset \M^\omega$. 
   If $e$ is the support projection of  $\omega$ 
  then $e \in \mathcal{D}'$ and $\mathcal{D} e$ is a von Neumann subalgebra
  of $e \M e$.    Also,  
  $\omega$  is a normal semifinite trace on $\M^\omega$, and also on each of $\D$ and $\D e$. 
  \end{lemma}

\begin{proof}  Since $1-e \in \mathfrak{p}(\M)_\omega$ as we said above, if $d \in \mathfrak{m}(\D)$ we see that 
$(1-e) d^* \in \mathfrak{m}(\M)_\omega$, and 
$$0 = \omega( (1-e) d^* d) =   \omega(d (1-e) d^*) =   
\omega(ed (1-e) d^*e).$$ 
 Hence  $ed (1-e) d^*e = 0$ 
since $\omega$ is faithful on $e \M e$.   We therefore have  $e d (1-e)  = 0$, so that $ed = ede$, and this must equal $d e$ by symmetry.    Thus $e$ commutes with 
the weak* closure of $\mathfrak{m}(\D)_\omega$ and its support projection.    Since
$\mathfrak{m}(\M)_\omega$ is an algebra invariant under multiplication by $e$, our hypothesis 
implies that 
$\mathfrak{m}(\D)_\omega \subset \M^\omega$.   Because $\omega$ is semifinite on $\D$, we see that $\D \subset \M^\omega$ and 
 $e \in \D'$.   The map $d \mapsto de$ is a normal 
$*$-homomorphism on $\D$.   Hence its range $\D e$ is weak*-closed, a von Neumann subalgebra of $e \M e$.

By Lemma \ref{istrD}, $\omega$  is a normal semifinite trace on $\D$.   
Since $\omega = \omega(e \, \cdot \, e)$, it is easy to see that 
it is also a normal semifinite trace on $\D e$. 
Setting $\D = \M_{\omega}$ in the argument above shows that 
$\omega$  is a normal semifinite trace on $\M_{\omega}$. 
 \end{proof}

 We say that a densely defined  closed operator $S$ in $H$ is
{\em affiliated} with a von Neumann algebra $M$ in $B(H)$ if $S$ commutes with all
unitaries $U \in M'$.  We write $S \, \eta \, \M$ to denote this.

We recall the Radon-Nikodym theorem of Pedersen and Takesaki \cite{PT}. Fix a faithful normal weight $\omega$  on  a von Neumann algebra $\M$. We use notation from \cite{PT}, except that given a positive operator $h$ affiliated to $\M_\omega$, we will as in \cite{tak2} 
write $\omega_h$ for the weight  \begin{equation}\label{PTRN}
\lim_{\epsilon\searrow 0}\omega(h_\epsilon^{1/2} x h_\epsilon^{1/2}),
\end{equation} 
where $h_\epsilon=h(\I+\epsilon h)^{-1}$.  
 If $\psi$ is a normal semifinite weight on $\M$ then we say that $\psi$  {\em commutes with} $\omega$  if $\psi \circ \sigma^\omega_t = \psi$ 
for all $t$.   The Radon-Nikodym theorem states that $\psi$  commutes with $\omega$  if and only if $\psi = \omega_h$ for a
(necessarily unique) positive selfadjoint $h$ affiliated with $\M_\omega$.    We call $h$ the {\em  Radon-Nikodym derivative} and write it as 
$\frac{d \psi}{d \omega}$.  See \cite{PT} for some theory of the Radon-Nikodym derivative, and also \cite{tak2}, e.g.\  Corollary VIII.3.6 there and the 
connection to the Connes cocycle derivative, although 
there all weights are faithful.   We also sometimes call the Radon-Nikodym derivative a {\em density}.
More loosely, by a density we mean  a positive affiliated operator which can be used to derive a new weight from the given reference weight.

In a few places we will assume that the reader is familiar with the extended positive part $\widehat{\M}_+$ of a von Neumann algebra $\M$
and its connection to weights and operator valued weights, 
as may be found in \cite[Section IX.4]{tak2} and \cite{haag-OV1}. 

We shall need some knowledge of noncommutative $L^p$-spaces for the ensuing analysis. In the case where $\M$ is equipped with a faithful normal semifinite trace $\tau$, this is achieved by passing to the so-called $\tau$-measurable operators $\widetilde{\M}$, which are defined to be the set of all those closed densely defined operators $f$ affiliated to $\M$, for which we have that $\tau(\chi_{(\epsilon,\infty)}(|f|))<\infty$ for some $\epsilon>0$. Under strong sum and product, this turns out to be a complete topological *-algebra when equipped with the so-called topology of convergence in measure. The $L^p(\M,\tau)$ spaces $(0<p<\infty)$, are then simply defined to be $L^p(\M,\tau)=\{f\in \widetilde{\M}\colon \tau(|f|^p)<\infty\}$. These spaces respect the same basic structural and duality properties of their commutative cousins,
 with $\|\cdot\|_p=\tau(|\cdot|)^{1/p}$ being the canonical norm on $L^p(\M,\tau)$ when $1\leq p<\infty$. The Pedersen-Takesaki Radon-Nikodym theorem ensures that $L^1(\M,\tau)$ is a copy of $\M_*$, 
realised as the predual of $\M=L^\infty(\M,\tau)$ via $\tau$-duality. 
(See e.g.\ \cite{GLnotes} for further details.)

To construct $L^p$ spaces for general von Neumann algebras $\M$ equipped with a faithful normal semifinite weight $\nu$, we need to pass through the crossed product with the modular group. In order to understand the analysis hereafter, we shall need some fairly detailed knowledge of key aspects of this construction. Since we will 
refer to the following facts frequently, we label them as remarks. 
These facts are scattered in various papers in the literature. However in view of their importance, we choose to summarise what we need locally
 (see e.g.\ \cite{GLnotes} for more details and background):  

\begin{remark} \label{crossprod}  Let $\M$ be a 
von Neumann algebra  acting on a Hilbert space $H$, and let $\nu$ be a fns weight on $\M$. Then $\M$ may be represented as an algebra acting as a von Neumann algebra
 on $L^2(\mathbb{R},H)$ by means of the $*$-isomorphism $a\to\pi_\nu(a)$ where $(\pi_\nu(a)\xi)(s)= \sigma_{-s}^\nu(a)\xi(s)$ for every $s\in \mathbb{R}$. 
We will usually write $\pi_\nu$ as $\pi$.
 One may also define left shift operators by $(\lambda_t(\xi))(s)= 
\xi(s-t)$. The dual action $\theta_s$ of $\mathbb{R}$ on $B(L^2(\mathbb{R},H))$ is implemented by the unitary group $w(t)$, where $(w(t)\xi)(s)=e^{-its}\xi(s)$.
We have   $\theta_s(\pi(a))=\pi(a)$ for all $a\in\M$, and $\theta_s(\lambda_t)=e^{-its}\lambda_t$. The crossed product $\mathfrak{M}=\M\rtimes_\nu \mathbb{R}$ is the 
von Neumann algebra  generated by $\pi(\M)$ and the $\lambda_t$'s. 
The operator valued weight $T_\nu$ from 
the extended positive part 
$\widehat{\mathfrak{M}}_+\equiv\widehat{\pi_\nu(\mathfrak{M})}^+$ to $\widehat{\M}_+$ is given by $T_\nu(a)=\int_{-\infty}^\infty\theta_s(a)\,ds$. 
The dual weight on $\mathfrak{M}$ of a normal semifinite weight $\rho$ on $\M$ 
(computed using $\nu$ as a reference weight)
is given by $\widetilde{\rho}=\rho\circ T_\nu$. It is well 
 known that $\sigma_t^{\widetilde{\nu}}$ is implemented by $\lambda_t$, and that $\sigma_t^{\widetilde{\nu}}(\pi(a))=\pi(\sigma_t^\nu(a))$ for any $a\in\M$. (See Theorem 4.7 and Lemma 5.2 and its proof in  \cite{haag-OV1, haag-OV2} for these facts.) By Stone's theorem there exists 
$k_\nu \, \eta \, \mathfrak{M}$ so that $k_\nu^{it}=\lambda_t$ 
 for each $t$. By \cite[Theorem VIII.3.14]{tak2}, $\mathfrak{M}$ is semifinite. It follows from the proof of \cite[Theorem VIII.3.14]{tak2}, that $\tau_{\mathfrak{M}}=\nu(k_\nu^{-1}\cdot)$ will then be a faithful normal semifinite trace on $\mathfrak{M}$ for which we have that $k_\nu=\frac{d\widetilde{\nu}}{d\tau}$ 
(writing $\tau_{\mathfrak{M}}$ as $\tau$). 
It is precisely this trace that plays such a crucial role in the construction of Haagerup $L^p$-spaces, as one may begin to appreciate 
already in the proof of \cite[Lemma 5.2]{haag-OV2}. 
For any normal semifinite weight $\rho$ on $\M$, the dual weight $\widetilde{\rho}=\rho\circ T_\nu$ above 
is characterized by the fact that the densities $k_\rho=\frac{d\widetilde{\rho}}{d\tau}$ 
satisfy $\theta_s(k_\rho)=e^{-s}k_\rho$ for all $s$.

In fact there is bijective correspondence between the classes of normal semifinite weights on $\M$, dual weights of such normal semifinite weights, and densities $k\,\eta\, \mathfrak{M}$ which satisfy $\theta_s(k)=e^{-s}k$ for all $s$. 
(See \cite[Proposition II.4]{terp} or \cite[Proposition 6.67]{GLnotes}.) In view of this correspondence, we shall denote this class of densities by ${}^\eta L^1_+(\M)$. 
The densities in ${}^\eta L^1_+(\M)$ which correspond to positive normal functionals are precisely those which belong to the algebra of $\tau$-measurable operators affiliated to $\mathfrak{M}$ (see \cite[Corollary 6.71]{GLnotes}). Using this fact one may show that the space $L^1(\M)=\{f\in \widetilde{\mathfrak M}\colon\theta_s(f)=e^{-s}f\mbox{ for all }s\in \mathbb{R}\}$ is a copy of $\M_*$. More generally one may define the $L^p(\M)$ $(0<p<\infty)$ spaces by $L^p(\M)=\{f\in \widetilde{\mathfrak M}\colon\theta_s(f)=e^{-s/p}f\mbox{ for all }s\in \mathbb{R}\}$. The space $L^1(\M)$ admits a trace-functional $tr_\M$ (or simply just $tr$ if $\M$ is understood) defined by $tr(f)=\omega_f(\I)$, where $\omega_f$ is the normal functional in $\M_*$ corresponding to $f\in L^1(\M)$. As was the case in the semifinite setting, these spaces observe all the expected duality rules, with the dual action of $L^p(\M)$ on $L^q(\M)$ (where $1=\frac{1}{p}+\frac{1}{q}$) 
being implemented via the $tr$-functional.
See \cite{GLnotes} for more on these topics.

Observe that if the same construction is done for any other 
von Neumann algebra  $\N$ with fns weight $\varphi$ acting on $H$, then the crossed product $\mathfrak{N}=\N\rtimes_\varphi\mathbb{R}$ will share the same left shift operators, 
which will here too implement the modular automorphism group for the dual weight $\tilde{\varphi}$ on $\mathfrak{N}$. Hence remarkable as it may seem, 
 the same argument as above will then yield the conclusion that the operator $k$ for which we have that $\lambda_t=k^{it}$ for all $t$, here satisfies $k=\frac{d\widetilde{\varphi}}{d\tau_{\mathfrak{N}}}$ and hence that 
$\frac{d\widetilde{\varphi}}{d\tau_{\mathfrak{N}}}=\frac{d\widetilde{\nu}} {d\tau_{\mathfrak{M}}}$ 
as operators affiliated to $B(L^2(\mathbb{R},H))$.  See \cite[Theorem 6.62]{GLnotes}. 
However we need to sound a word of warning here. Given two faithful normal semifinite weights $\nu$ and $\rho$ on $\M$, one should carefully differentiate between the dual weight of $\rho$ computed using the crossed product $\M\rtimes_\nu\mathbb{R}$, and that using the crossed product $\M\rtimes_\rho\mathbb{R}$. Let us write $\widetilde{\rho}^{(\nu)}$ and $\widetilde{\rho}^{(\rho)}$ for these two dual weights. The weight $\widetilde{\rho}^{(\nu)}$ is an altogether different 
 object to $\widetilde{\rho}^{(\rho)}$, and if $\nu$ does not agree with $\rho$, one should most certainly not expect that $\frac{d\widetilde{\nu}}{d\tau_{\mathfrak{M}}}$ agrees with $\frac{d\widetilde{\rho}^{(\nu)}}{d\tau_{\mathfrak{M}}}$.  
\end{remark}

\begin{remark}\label{crossprod2} Let $\M$ be a von Neumann algebra  acting on a Hilbert space $H$, and let $\nu$ be a fns weight on 
$\M$. Let $\D$ be a von Neumann subalgebra of $\M$ and $E_\D$ a faithful normal conditional expectation from $\M$ onto $\D$ for which we have that $\nu\circ E_\D=\nu$. Then of course $\sigma_t^\nu(\D)=\D$ and $E_\D \circ\sigma_t^\nu=\sigma_t^\nu\circ E_\D$ for all $t$. The first claim is proved in \cite[IX.4.2]{tak2}. For the second note that for any $a\in \mathfrak{m}_\nu$,  using $\nu\circ E_\D=\nu$ and $\nu\circ\sigma_t^\nu = \nu$, it is an exercise to see that $\nu(|E_\D(\sigma_t(a))-\sigma_t(E_\D(a))|^2)=0$ for all $t$, and hence that $E_\D \circ\sigma_t^\nu=\sigma_t^\nu\circ E_\D$ on $\mathfrak{m}_\nu$. 
By normality and the fact that $\mathfrak{m}_\nu$ is weak* dense in this case, $E_\D \circ\sigma_t^\nu=\sigma_t^\nu\circ E_\D$.   
See also \cite[Proposition 4.9]{haag-OV1}.

If now we compute the crossed product of $\D$ with respect to the modular automorphism group determined by $\nu_{|\D}$, it is clear from the definitions given in the preceding remark that in this case $\pi_{\nu_{|\D}}(d)=\pi_\nu(d)$ for all $d\in D$. Since $\D\rtimes_{\nu_{|\D}} \mathbb{R}$ is generated by  $\pi_{\nu_{|\D}}(\D)$ and the left shift operators, it follows that $\D\rtimes_{\nu_{|\D}} \mathbb{R}$ is in the 
naive sense a linear subspace of $\M\rtimes_{\nu} \mathbb{R}$. So in this particular case, the discussion in the preceding remark does indeed show that $\frac{d\widetilde{\nu}}{d\tau_{\mathfrak{M}}}=\frac{d\widetilde{\nu_{|\D}}}{d\tau_{\mathfrak{D}}}$.

By \cite[Theorem 4.1]{HJX}, $E_\D$ extends to a 
faithful normal conditional expectation $\overline{E_\D}$ from $\M\rtimes_\nu\mathbb{R}$ to $\D\rtimes_\nu\mathbb{R}$ which satisfies 
$\widetilde{\nu}\circ \overline{E_\D}= \widetilde{\nu}$ and $\theta_s\circ \overline{E_\D}=\overline{E_\D}\circ\theta_s$ for all $s$. In particular $\widetilde{\nu}$ is semifinite on $\D\rtimes_\nu\mathbb{R}$. The prescription given by \cite{HJX} for this extension is exactly the one given by 
\cite[Proposition  4.3]{Gol}. In fact Section 4.1 of \cite{Gol} is a precursor to \cite[Theorem 4.1]{HJX}. So \cite[Theorem 4.1]{HJX} asserts that \cite[Proposition  4.3, Lemma 4.4, Proposition 4.6, Theorem 4.7]{Gol} holds in the non-$\sigma$-finite setting. By exactly the same proof as the one used by Goldstein, \cite[Lemma 4.5 \& Lemma 4.8]{Gol} hold in the non-$\sigma$-finite setting. Now observe that the proof of \cite[Proposition 4.5]{LM} shows that we also have that $\tau \circ \overline{E_{\D}} =\tau$ where $\tau$ is the trace on $\M\rtimes_\nu\mathbb{R}$. 

By \cite[Remark 5.6]{HJX} and the reasoning in  Example 5.8 there, there is an induced expectation $E^{(1)}_\D$ on $L^1(\M)$,
given by the continuous extension to all $L^1(\M)$, of the map $k^{1/2}ak^{1/2}\to k^{1/2}E_\D(a)k^{1/2}$ where $a\in \mathfrak{m}_\nu$, and where $k=\frac{d\widetilde{\nu}}{d\tau}=\frac{d\widetilde{\nu_{|\D}}}{d\tau_{\mathfrak{D}}}$.
For any $a\in \mathfrak{m}_\nu$ it follows from \cite[Proposition 2.13(a)]{GL2} that $$tr_\D(E^{(1)}_\D(k^{1/2}ak^{1/2}))=tr_\D(k^{1/2}E_\D(a)k^{1/2})=\nu(E_\D(a))=\nu(a)=tr(k^{1/2}ak^{1/2}).$$ Now recall that 
$k^{1/2}\mathfrak{m}_\nu k^{1/2}$ is dense in $L^1(\M)$ \cite[Proposition 2.11]{GL2}. So by continuity $tr_\D\circ E^{(1)}_\D =tr$. 
\end{remark}

\section{Normal conditional expectations and centralizers of normal states and weights} \label{ncec}

\begin{lemma} \label{aves} Let $\omega$ be a normal state on a finite von Neumann algebra $\M$ which is 
tracial on a von Neumann subalgebra $\mathcal{D}$.   Let ${\mathcal U}$ be the unitary group of $\D$, and let $K_\omega$
be the 
norm closed convex hull of 
$$\{ u^* \omega u : u \in {\mathcal U} \}$$ in $\M_*$.
   Then $K_\omega$ is 
weakly compact, and contains a  normal state  $\psi$ extending $\omega_{\vert \D}$ such that 
$$\psi(dx)=\psi(xd), \qquad d \in \mathcal{D}, x\in \M.$$   \end{lemma}

\begin{proof}   We follow the idea in the proof that (i) implies (ii) in  \cite[Theorem V.2.4]{tak2}.  By that proof
 $K_\omega$ is weakly compact, and the group of isometries $\varphi \mapsto u^* \varphi u$, for 
$u \in {\mathcal U}$, has a fixed point $\psi$  in $K_\omega$.  
Since $\psi = u^* \psi u$, and since $ {\mathcal U}$ spans $\D$, we have $\psi(dx)=\psi(xd)$ for $d \in \mathcal{D}, x\in \M.$
 It is easy to see that any element of  $K_\omega$ is a normal state 
   on $\M$ extending $\omega_{|\D}$.
\end{proof}

\begin{remark}  If we drop the requirement that $\omega$ is tracial on $\D$ one obtains the same conclusion except that 
$\psi$ need not extend $\omega_{|D}$.   

 We thank Roger Smith and Mehrdad Kalantar for conversations around this last result. 
  Indeed Roger Smith pointed us to  \cite[Lemma 3.6.5]{SS},
 and the following result about masas, and its proof, which was an important initial step in the present investigation: 
  Let $\M$ be a von Neumann algebra with a faithful normal tracial state $\tau$, 
and let $\D$ be a masa in  $M$. Let $\Phi : \A \to \D$ be a unital weak* continuous
map on a weak* closed subalgebra  $\A$ of $\M$ containing $\D$, which is a $\D$-bimodule map.   Then
$\Phi$ is the restriction to $\A$ of the
unique $\tau$-preserving  normal conditional expectation $\Edb_{\tau} : M \to D$. 
To see this, since $\D$ is a masa, by
Lemma 3.6.5 and the
remark after it in \cite{SS}, for $x \in \A$ there is a net $x_t$
with terms which are convex combinations
of $u^* x u$ for various unitaries $u \in \D$
such that $x_t \to \Edb_{\tau}(x)$ weak*.  Since $\Phi(u^* x u) = u^* \Phi(x) u = \Phi(x)$,
we deduce that $\Phi(x) = \Edb_{\tau}(x)$ for $x \in \A$.

On the other hand, after we had proved the last lemma and used it to obtain  a later result, 
Mehrdad Kalantar showed us an alternative proof that 
gives  $\psi$ above as $\omega \circ E_{\D' \cap \M}$.
A generalization of his argument plays a role
 in the next result and the point in its proof where his contribution is mentioned in more detail.  
\end{remark}

\begin{example} \label{abave} One may ask if there is a variant of the last lemma for 
$\sigma$-finite von Neumann algebras.   
Consider the case that  $\D$ is the copy of $L^\infty([0,1])$  in the $\sigma$-finite semifinite von Neumann algebra $\M = B(L^2([0,1]))$.   Let  $\omega$ be the normal state on $\M$ defined by 
$\omega(x) = \frac{1}{3} \, \sum_{k \in \Zdb} \, 2^{-k} \,  \langle x e_k , e_k \rangle$, for the standard orthonormal basis $e_n = e^{2 \pi n i t}$  of $L^2([0,1])$.   If $M_f$ is multiplication by $f \in L^\infty([0,1])$ then 
 $\omega(M_f) = \frac{1}{3} \, \sum_k \, 2^{-k} \,  \int_0^1 \, f \, dx =  \int_0^1 \, f \, dx$, so $\omega$ restricts to a faithful normal tracial state
on $\D$.    However there is no  normal state  $\psi$ extending $\omega_{|\D}$ such that 
$\psi(dx)=\psi(xd),$ 
for $d \in \mathcal{D}, x\in \M.$   Indeed if there were then by e.g.\ Theorem  \ref{wcent} there would be a normal conditional
expectation onto $\D$, which is well known to be false.   Note that $\M_\omega$ in this example is also an abelian von Neumann algebra,
namely the operators that are diagonal with respect to this orthonormal basis. 
\end{example}

The situation is better if we fix a faithful 
 normal state  $\omega$ on a $\sigma$-finite von Neumann algebra $\M$  with $\D$ a
von Neumann subalgebra of the centralizer $\M_\omega$:

\begin{lemma} \label{aves2} Let $\omega$ be a faithful normal state on  a von Neumann algebra $\M$, and let $\D$ be a
von Neumann subalgebra of $\M_\omega$.   Let ${\mathcal U}$ be the unitary group of $\D$, $\psi$ a
normal state extending $\omega_{\vert \D}$, and let $K_\psi$
be the norm closed convex hull of 
$$\{ u^* \psi u : u \in {\mathcal U} \}$$ in $\M_*$.
   Then $K_\psi$  contains a  normal state  $\rho$ extending $\omega_{\vert \D}$ such that 
$$\rho(dx)=\rho(xd), \qquad d \in \mathcal{D}, x\in \M. $$     
Indeed we may take $\rho$ to be $\psi \circ E_{\D' \cap \M}$ 
where $E_{\D' \cap \M}$ is the unique $\omega$-preserving 
conditional expectation onto 
$\D' \cap \M$.  
Finally, if  $\psi$ commutes with $\omega$ then so does $\rho$.
\end{lemma}

\begin{proof}   First we claim that there exists an  $\omega$-preserving normal conditional expectation $E_{\D' \cap \M}$ onto  
$\D' \cap \M$.   We follow the proof strategy of  \cite[Lemma 3.6.5]{SS}.
Set $k = k_\omega$, the density $\frac{d\widetilde{\omega}}{d\tau}$ of the dual weight $\widetilde{\omega}$ on $\M\rtimes_\nu\mathbb{R}$ 
(as in \ref{crossprod}; here $\nu$ is a reference weight 
and $\tau$ is the canonical fns trace on the crossed product).
  It is well known that  $k$ strongly commutes with elements of the centralizer of $\omega$ (see e.g.\  
Proposition \ref{commhk} below).        The 
embedding $x \mapsto x k^{\frac{1}{2}}$  of $\M$ into  $L^2(\M)$  is weak* continuous, since it is easy to see
that for any $\eta \in L^2(\M)$ the functional 
$\langle \, \cdot \, k^{\frac{1}{2}} , \eta \rangle$ is weak* continuous (since the representation of $M$ on $L^2(\M)$ is normal, and
$k^{\frac{1}{2}} \in L^2(\M)$).   Thus the image of any norm closed ball in $\M$ is 
weakly compact.    Let $K_{\D}(x)$ be the weak* closed convex hull  in $\M$ of $\{ u x u^* : u \in {\mathcal U} \}$. 
Hence $K_{\D}(x) k^{\frac{1}{2}}$  is a weakly closed, hence $\| \cdot \|_2$-norm closed by Mazur's theorem,
 convex set in $L^2(\M)$ (and in $\M k^{\frac{1}{2}}$).  
 Let $\phi(x)$ be the unique element of smallest $\| \cdot \|_2$-norm (in $L^2(\M)$) in $K_{\D}(x) k^{\frac{1}{2}} \subset L^2(\M)$. 
   We have $\phi(x) \in M k^{\frac{1}{2}}$, and $\phi(x)  = E(x) k^{\frac{1}{2}}$ for  an 
element $E(x) \in K_{\D}(x) \subset \M$.   Clearly  $\| E(x ) \| \leq \| x \|$.   We also have $\omega(E(x)) = \omega(x)$, since $\D \subset
\M_\omega$.    
For a fixed unitary $v \in {\mathcal U}$ the map $x k^{\frac{1}{2}} \to v x k^{\frac{1}{2}} v^*$ is a $\| \cdot \|_2$-norm isometry, and so as in  \cite[Lemma 3.6.5]{SS}  
we have that $E(x) \in \D' \cap \M$, and for any unitaries $v,w \in \D' \cap \M$ 
we have $K_{\D}(wxv) = w K_{\D}(x) v$ and $E(wxv) = w E(x) v$.
 It follows that $$\omega(E(x + y) v) = \omega((x+y) v)= \omega( 
E(xv)+E(yv) ) = \omega((E(x)+E(y))v), \qquad x, y \in \M. $$
Hence this is true with $v$ replaced by any element in $\D' \cap \M$.   Hence 
$E(x + y) = E(x)+E(y)$.   For a scalar $t > 0$ we have $K_{\D}(tx) = t K_{\D}(x)$, and $\phi(tx) = t \phi(x)$.
It follows that $E$ is linear, and also  $E$ is a $\D' \cap \M$-module map.    
Of course $E(1) = 1$ and so $E$ is an $\omega$-preserving 
conditional expectation onto 
$\D' \cap \M$.       The uniqueness of $\omega$-preserving conditional expectations is standard (see e.g.\ the proof of 
Theorem \ref{wcent}).

We have $$\sigma^\omega_t(d') d =  \sigma^\omega_t(d' d) = \sigma^\omega_t(dd' )= 
d \sigma^\omega_t(d'), \qquad d' \in \D' \cap \M, \; d \in \D.$$  Thus $\sigma^\omega_t(\D' \cap \M) \subset 
 \D' \cap \M$, and so by \cite[Theorem IX.4.2]{tak2} there does exist a normal 
 $\omega$-preserving conditional expectation onto $\D' \cap \M$.     Thus $E$ is normal, by the uniqueness above.  
   Note  that $K_{\D}(v x v^*) = K_{\D}(x)$ for $v  \in {\mathcal U}$,
and so $E(v x v^*) = E(x)$.   Hence $\rho(u^* x u) = \rho(x)$ for all $u \in {\mathcal U}$
and $x \in \M$, so that $\D \subset \M_\rho$.     

We learned the following argument from M. Kalantar. 
We claim the normal state $\rho=\psi \circ E$ is in $K_\psi$.  By way of  contradiction, assume $\rho \notin K_\psi$.  By the geometric Hahn-Banach theorem, there exists $x \in \M$ and real $t$ such that $${\rm Re}(\rho(x)) < t < {\rm Re}(\alpha(x)) , \qquad \alpha \in K_\psi . $$ 
Since $E(x)  \in K_{\D}(x)$ 
we see that the real part of $\rho(x) = \psi(E(x))$ is in the closed convex hull of 
$$\{ {\rm Re} \, \psi(ux 
u^*) : u\in {\mathcal U} \} \subset  \{  {\rm Re} \, \alpha(x) : \alpha \in K_\psi \} \subset [t, \infty) .$$
 This contradicts the inequality above.  So $\rho \in K_\psi$.
 
Since $\rho \in K_\psi$ it extends $\omega_{\vert \D}$, since if $d \in \D$ then  $\rho(d)$  is in the norm 
closed convex hull of terms of form 
$\psi(u^* du) = \omega(d)$. 

Finally, if  $\psi$ commutes with $\omega$ then $$\psi(u^* \sigma^\omega_t(x) u) = \psi( \sigma^\omega_t( u^* x u)) 
= \psi ( u^* x u) , \qquad x \in \M , u \in \D.$$ This will persist for the convex hull of 
$\{ u^* \psi u : u \in {\mathcal U} \}$, hence for its norm closure $K_\psi$.  So $\rho$ commutes with $\omega$. 
\end{proof}

\begin{theorem}  \label{wcent} Let $\omega$ be a normal state on  a von Neumann algebra $\M$ which is faithful and tracial 
on a von Neumann subalgebra $\mathcal{D}$.  Then $$\omega(dx)=\omega(xd), \qquad d \in \mathcal{D}, x\in \M,$$ if and only if there is a normal conditional expectation $\mathbb{E}_\omega$ from $\M$ onto $\mathcal{D}$ which preserves $\omega$.  
 Moreover such an $\omega$-preserving normal conditional expectation onto $\D$ is unique, and  is  faithful if $\omega$ is faithful on $\M$.
\end{theorem} 

\begin{proof} For the `only if' part, the hypothesis about  $\omega(dx)=\omega(xd)$ ensures that 
$\mathcal{D}\subset \M_\omega$.   By \cite[Theorem VIII.2.6]{tak2}, if $\omega$ is faithful  (so $\M$ is $\sigma$-finite) then  there exists a unique $\omega$-preserving normal conditional expectation 
$E_\D$ from $\M$ onto $\D$. (Alternatively, by 
\cite[Theorem IX.4.2]{tak2} there exists a normal conditional expectation $\mathbb{E}_1$ from $\M$ onto $\M_\omega$ which 
preserves $\omega$.  From the the latter it follows that $\mathbb{E}_1$ is  faithful.  By 
\cite[Proposition V.2.36]{tak1}, there exists a faithful normal conditional expectation $\mathbb{E}_2$ from $\M_\omega$ onto $\mathcal{D}$. The conditional expectation we seek is $\mathbb{E}=\mathbb{E}_2\circ\mathbb{E}_1$, as the reader can easily check.)

Now consider the case where $\omega$ is not faithful. Let $z$ be the support projection of the state $\omega$. 
By Lemma \ref{tuse} and its proof we may make the following assertions: 
$z \in \mathcal{D}'$, $\mathcal{D}z$ is in fact a $W^*$-algebra, a von Neumann subalgebra of $z \M z$, and the map $d \mapsto dz$ is a normal 
$*$-homomorphism from  $\D$ onto $\D z$.   These do not require 
$\omega$ to be faithful on $\D$.  
Also, 
if $\omega$ is faithful on $\D$  then the  $*$-homomorphism $d \mapsto dz$ is also faithful, hence is a $*$-isomorphism from  $\D$ onto $\D z$, 
 and  $\omega$ is a trace on $\mathcal{D}z$.

We next show that as a subalgebra of $z \M z$, $z\mathcal{D}z$ satisfies the hypothesis of the proposition. For any $b\in \M$ and $d\in \mathcal{D}$, we have that $$\omega((zbz)(dz))=\omega(bzd)=\omega(dbz)=\omega(db).$$ Similarly $\omega((dz)(zbz))=\omega(bd)$. Since $\omega(bd)=\omega(db)$, the restriction of $\omega$ to $z\M z$ is a faithful state which as claimed, satisfies the hypothesis with respect to the subalgebra $\mathcal{D}z$. Therefore by the first part of the proof, there exists a faithful normal conditional expectation $\mathbb{E}_0:z\M z\to \mathcal{D}z$ which preserves the restriction of $\omega$ to $z\M z$.

Since the map $\mathcal{D}\to \mathcal{D}z:d\mapsto dz$ is an injective  normal 
$*$-homomorphism, its inverse 
$\iota_{\mathcal{D}}: \mathcal{D}z\to \mathcal{D}:zdz\mapsto d$ is an  injective  normal 
$*$-homomorphism. With $Q_\omega$ denoting the compression $\M\to z\M z$, we will show that $\mathbb{E}_\omega=\iota_{\mathcal{D}}\circ\mathbb{E}_0\circ Q_\omega$ is a normal conditional expectation from $\M$ onto $\mathcal{D}$. 

The map $\mathbb{E}_\omega$ is clearly a unital map onto $\mathcal{D}$. The map $\mathbb{E}_0\circ Q_\omega$ is clearly positive, and $\iota_{\mathcal{D}}$ is positive.
So $\mathbb{E}_\omega$ is a unital positive map, and hence is contractive. 

Given $d\in \mathcal{D}$, we have that $$\mathbb{E}_\omega(d)=\iota_{\mathcal{D}}\circ\mathbb{E}_0\circ Q_\omega(d)= \iota_{\mathcal{D}}\circ\mathbb{E}_0(dz)= \iota_{\mathcal{D}}(dz)=d.$$ Thus $\mathbb{E}_\omega \circ \mathbb{E}_\omega = \mathbb{E}_\omega$. The map $\mathbb{E}_\omega$ is therefore a conditional 
expectation from $\M$ onto $\mathcal{D}$. It is an exercise to see that $\omega \circ i_{\D} = \omega$, and hence that 
$\omega\circ\mathbb{E}_\omega=\omega$.
That  $\mathbb{E}_\omega$ is normal follows since the maps $\mathbb{E}_0, Q_\omega,$ and $i_{\D}$ are all normal.

 Given  another such $\omega$-preserving normal conditional expectation $\rho$ onto $\D$ we have
$$\omega((\rho(x) - \mathbb{E}_\omega(x)) d) = \omega(xd) - \omega(xd) = 0 , \qquad d \in \D .$$
Since $\omega$ is faithful we deduce that  $\rho = \mathbb{E}_\omega$.

For the  ``if'' part, suppose that there exists a normal conditional expectation $\mathbb{E}_\omega$ from $\M$ onto $\mathcal{D}$ which preserves $\omega$.  Then 
$$\omega(ax)=\omega(\mathbb{E}_\omega(ax))=\omega(a\mathbb{E}_\omega(x))= \omega(\mathbb{E}_\omega(x)a)=\omega(\mathbb{E}_\omega(xa))=\omega(xa),$$ for  $a\in \mathcal{D}$ and $x\in \M$. 
\end{proof} 

\begin{remark} 
Putting the last result together with Lemma \ref{aves}, if $\omega$ is a normal state on  a 
finite von Neumann algebra $\M$, which is faithful and tracial 
on a von Neumann subalgebra $\mathcal{D}$, then there is a normal conditional expectation 
from $\M$ onto $\mathcal{D}$ which preserves $\omega \circ E_{\D' \cap \M}$.   The latter state is `an averaging of $\omega$ over the unitary group' of $\D$.  
\end{remark}

Let $\omega$ be a normal weight on a von Neumann algebra $\M$. 
We say that a von Neumann subalgebra $\D$ is  {\em locally} $\omega$-{\em central} if $\omega(pxpdp) = \omega(pdpxp)$ 
for all $p\in\mathbb{P}(\D)$ with $\omega(p)<\infty$ and $x \in \M, d \in \D$.   Note that $pxp$ and $pdp$ above are in the `definition domain' 
${\mathfrak m}(\M)_\omega$ of $\omega$ (see \cite[Definition VII.1.3]{tak2}). 
We say that a normal conditional expectation $E$ from $\M$ onto
  $\mathcal{D}$ is {\em locally} $\omega$-{\em preserving} if we have that $\omega(E(p x p)) = \omega(pxp)$ for every $x \in \M_+$ and every projection $p \in \D$ such that $\omega(p) < \infty$.

We next show that `local = global' in these two definitions, under a reasonable assumption 
that always holds for example for faithful weights.  
Knowing that these two properties coincide then provides us with a very powerful technique for lifting results known for states to results pertaining to normal semifinite weights.

\begin{theorem}  \label{locis} Let $\omega$ be a normal weight on  a von Neumann algebra $\M$ which is faithful, semifinite, and tracial 
on a von Neumann subalgebra $\mathcal{D}$.  Then $\D$ is   $\omega$-central if and only if $\D$ is  locally $\omega$-central
and  the support projection of $\omega$ commutes with $\D$.  
If these hold then we have 
$$\sup_p \, \omega(pxp) =
\omega(x), \qquad x \in \M_+,$$ where the supremum is taken over all projections in $\D$ with $\omega(p)<\infty$; also 
if $E : \M \to \D$ is a locally $\omega$-preserving normal conditional expectation  then $E$ is $\omega$-preserving.
\end{theorem} 

\begin{proof} Clearly if $\D$ is  $\omega$-central then it is  locally $\omega$-central: indeed 
if $\omega(xd) = \omega(dx)$ for all 
$d \in \mathfrak{m}(\D)_\omega$ and $x \in \mathfrak{m}(\M)_\omega$, then $\omega(pxpdp) = \omega(pdpxp)$ 
for all $p\in\mathbb{P}(\D)$ with $\omega(p)<\infty$ and $x \in \M, d \in \D$.  
The support projection of $\omega$ commutes with $\D$ by Lemma  \ref{tuse}.

 For the converse, suppose 
  that $\D$ is locally $\omega$-central. 
 First let $\omega$ be a fns weight which is a semifinite trace on $\D$.   
Let ${\mathfrak A}_\omega$  be the full left Hilbert algebra associated with $\omega$.   
 This contains the image of $p \M p$ for $\omega$-finite $p \in \Pdb(\D)$. 
Since $\omega$ is tracial on $\D$, the set $\Pdb(\D)_\omega$  of $\omega$-finite projections in $\D$ is upwardly directed, 
 and may be viewed as a net with weak* limit $1$; that is $p \nearrow 1$ along this directed set $\Pdb(\D)_\omega$.  Since the left and right representations
 of $\M$ on $L^2(M)$ 
  are  weak* continuous, we see that  $p_t \to 1$ strongly on $L^2(M)$. 
Hence  
\begin{equation}\label{eq:loc=glob} \| x - pxp \|_2 \leq \| x - px \|_2 + \| p (x - xp)  \|_2 \leq \| x - px \|_2 + \| x - xp  \|_2 \to 0\end{equation}  
as $p \nearrow 1$ as above. 
  We now use the language and notation of \cite[Lemma VI.1.5]{tak2} 
 and its proof.  Thus $p \M p \subset {\mathfrak D}^{\sharp}$, and $S_0(pxp) = px^* p$ for $x \in \M$.   
 Indeed ${\mathfrak D}^{\sharp}$ contains the set $K$ of those $x \in \M$ which equal $p x p$ for some  $p \in \Pdb(\D)_\omega$. 
 Note that $K$ is a linear subspace, since if $r \in \Pdb(\D)_\omega$ dominates $p$ and $q$ in $\Pdb(\D)_\omega$ 
 then $y = pxp + qzq$ equals $r yr$.  
 Also, by equation (\ref{eq:loc=glob}) above $K$ is $2$-norm dense in $D^{\sharp}$, and hence also in $H = L^2(M)$.  Let $R$ be the restriction of $S_0$ to  $K$, 
and let $R_1$ be the linear operator $R_1 = C R$ into $\bar{H}$, where as in the proof of \cite[Lemma VI.1.5]{tak2}, $C$ is the canonical anti-linear isometry from $H$ to $\bar{H}$.  
 We claim that the closure of $R_1$ is the closed operator $C S$. 
To see this note that, as was shown in the proof of \cite[Lemma VI.1.5]{tak2}, $C S$ is the closure of $S_1 = C S_0$.   Moreover if $(x, S_1 x)$ is in the graph of 
 $S_1$ then $x \in {\mathfrak D}^{\sharp}$, and $p_t x p_t \to x$ in $2$-norm as above.  
 Since also $p_t x^* p_t \to x^*$ in $2$-norm, it is now clear that $C S_0(p_t x p_t) \to C S_0(x)$.   It follows that 
(the graph of) $C S$ is the closure of (the graph of) $R_1$, or that $K$ is a core for $C S$.

Note that if $q \leq p$ in $\Pdb(\D)_\omega$, then 
 $$\omega(pxp q) = \omega(pxp q p) = \omega(pqp xp) = \omega(q pxp) = \langle pxp , q \rangle , \qquad x \in \M.$$ 
It follows from the above that for any $\xi \in K$ and any $q\in \Pdb(\D)_\omega$, we have 
$\langle S_0 \xi , q \rangle = \langle q , \xi \rangle$, or equivalently 
$\langle q, S_0 \xi \rangle = \langle \xi  , q \rangle$. To see this recall that $\xi$ is of the form $\xi=pxp$ for some $p\in \Pdb(\D)_\omega$. Since $\Pdb(\D)_\omega$ is  upwardly directed, we may assume that $p\geq q$. It is then an exercise to conclude from the previously centred equation that $\langle q, S_0 \xi \rangle = \omega((px^*p)^*q) =\omega(qpxp) =\langle \xi, q\rangle$ as claimed. This ensures that the functional $\xi \mapsto \langle q , S_0 \xi \rangle$ where $\xi \in K$, will for any $q\in \Pdb(\D)_\omega$, be bounded and linear on $K \subset {\mathfrak A}_\omega$.
 Equivalently, the functional  $\xi \mapsto \langle R_1 (\xi), Cq \rangle$ is bounded and linear on  $K$, and 
 for $\xi \in K$  we have $\langle R_1 (\xi), Cq \rangle =  \langle \xi  , q \rangle$.   
 Thus $Cq \in {\rm Dom} (R_1^*)$ and $R_1^*(Cq) = q$.   
 Since  $K$ is a core for $C S$, $(R_1)^* = (CS)^*$ by an elementary lemma in the  theory of unbounded operators. 
Thus $(CS)^*(Cq) =(CS)^*(CSq)  = q$, whence $\Delta(q) = q$.   Since $q \in {\mathfrak A}_\omega$ 
it follows by the criterion in \cite[Theorem VI.2.2 (i)]{tak2}  that $q$ is in the Tomita algebra ${\mathfrak A}_0$.  
Hence by \cite[10.21 Corollary]{Strat-Zsido}, $q$ is in the Tomita algebra as described on p.\ 28 in \cite{Stratila}.
 
By \cite[Proposition 2.13]{Stratila}, the functionals $a\to \omega(ap)$ and $a\to \omega(pa)$ on 
$\mathfrak{m}(\M)_\omega$ will for any $p\in\Pdb(\D)_\omega$ extend to
weak* continuous functionals on $\M$. Since  $\Pdb(\D)_\omega$ is  upwardly directed
we may construct a net $(q_t)$  in $\Pdb(\D)_\omega$ increasing to 1, with each $q_t$ majorising $p$. Given $a\in\mathfrak{m}(\M)_\omega$ we  then have $$\omega (q_tap)= \omega((q_taq_t)(q_tpq_t))= \omega((q_tpq_t)(q_taq_t))= \omega(paq_t).$$ Since $q_t \to 1$ in the $\sigma$-strong* topology, the continuity noted earlier ensures that 
$\omega(ap)=\lim_t \omega(q_tap)=\lim_t \omega(paq_t)=\omega(pa)$. Thus we  have $p\in \M_\omega$. Since $\omega$ is tracial on $\D$, the span of $\Pdb(\D)_\omega$ is weak* dense in $\D$. In addition $\M_\omega$ is a von Neumann algebra. 
Combining these two facts with what we have just proven then shows that $\D \subset \M_\omega$.  
That is, $\D$ is $\omega$-central.

Now for the non-faithful case.  
Here $\omega$ is a normal weight on  a von Neumann algebra $\M$ which is faithful, semifinite, and tracial 
on  $\mathcal{D}$, and  $\D$ is  locally $\omega$-central.  
  Let $e$ be  the support projection of  $\omega$.  As in the proof of Lemma \ref{tuse} and Theorem \ref{wcent}  
 $\mathcal{D} e$ is a  von Neumann subalgebra of $e \M e$, and the map $d \mapsto de$ is a faithful normal 
$*$-homomorphism from  $\D$ onto $\D e$.      It follows that any projection in $\D e$ is of the form $pe$ for $p \in \Pdb(\D)$, 
and $\omega(p) = \omega(pe)$.  If the latter is finite then 
$$\omega(pe d pe x e pe) = \omega(pdp exep) = \omega(pexep dp ) = \omega(pe x e pe d pe), \qquad d \in \D , exe \in e \M e .$$
Thus  $e \mathcal{D}$  is locally $\omega$-central in $e \M e$.   By the first part of the proof, $ep$ is in the centralizer of $\omega$ 
 in $e \M e$,
and $e \mathcal{D}$  is  $\omega$-central in $e \M e$.   Finally, if $d \in {\mathfrak m}(\D)$ then $d \in \M^\omega$ since $ed = de$ and 
$$\omega(d x) = \omega(ed xe) =\omega(ed exe) =\omega(exe de) =\omega(xd), \qquad x \in {\mathfrak m}(\M)_\omega .$$
Since $\M^\omega$ is a von Neumann algebra we have $\D \subset \M^\omega$.

For the final assertions we will apply \cite[Proposition 4.2]{PT} 
to the restriction of $\omega$ to 
 $e \M e$. 
For any $x\in\M_+$ we have that 
$$\sup_p \, \omega(pxp)= \sup_p \, \omega(epxpe) = \sup_p \, \omega(epexepe) = \omega(exe) = 
 \omega(x),$$ where the supremum is taken over $p \in \Pdb(\D)_\omega$.  
We have used 
that $e \in \D'$, that $pe \nearrow e$, and the fact from earlier in the proof
that $ep$ is in the centralizer of $\omega$ on $e \M e$.
The  $\omega$-preserving assertion clearly follows from the second last  centered  equality above, 
 since $$\omega(E(x)) =  \sup_p \, \omega(pE(x)p) =\sup_p \, \omega(E(pxp)) = \sup_p \,
 \omega(pxp) = \omega(x),$$  for $x \in \M_+$.
\end{proof}

\begin{remark} 
Let $\omega$ be a normal weight on  a von Neumann algebra $\M$ with support 
projection $e$, which is faithful, semifinite, and tracial 
on a von Neumann subalgebra $\mathcal{D}$, and suppose that $\mathcal{D}$ is locally  
$\omega$-central.
One can show that the following are equivalent:  
(1)\  
 $e \in \mathcal{D}'$, 
 (2)\  $\mathcal{D}$ is   $\omega$-central,
(3)\   $\sup_{p  \in \Pdb(\D)_\omega} \, \omega(pxp)= \omega(x)$  for $x \in \M_+$,
(4)\ there exists a $\omega$-preserving normal conditional expectation, and 
(5)\   there exists a locally  $\omega$-preserving normal conditional expectation $E$ with $E(e) = 1$
(or equivalently, with supp$(E) \leq e$, where supp$(E)$ is as defined above Lemma \ref{suE}). 
\end{remark}

\begin{theorem}  \label{wcentw} Let $\omega$ be a normal weight on a von Neumann algebra $\M$ which is a 
faithful semifinite trace when restricted to a von Neumann subalgebra $\mathcal{D}$.  Then there is a 
locally $\omega$-preserving normal conditional expectation $\mathbb{E}_\omega$ from $\M$ onto $\mathcal{D}$ if and only
$\D$ is locally $\omega$-central.
 Moreover such a locally $\omega$-preserving normal conditional expectation onto $\D$ is unique, and is  faithful if $\omega$ is faithful on $\M$. 
\end{theorem}

\begin{proof}
Suppose that there exists a locally $\omega$-preserving normal conditional expectation $E$ onto $\D$. Let $p\in \mathcal{D}$ be a projection with 
$\omega(p)<\infty$. The conditional expectation $E$ then restricts to a normal conditional expectation from $p \M p$ to $p \D p$.  
Also the restriction of $\omega$ to $p \M_+ p$ is  finite.  Thus $p \M p \subset {\mathfrak m}_\omega$,
and $\omega$ extends to a scalar multiple of a normal state $\omega_p$ on $p \M p$, with $\omega_p$  on  $p \D p$. 
 We have $\omega_p \circ E = \omega_p$ on $p \M p$. So by Theorem \ref{wcent} we 
have $\omega_p(dx) = \omega_p(xd)$ for $x \in p \M p, d \in p \mathcal{D} p$. Thus 
$\omega(p dp xp)=\omega(p xp dp)$ for $d \in \mathcal{D}, x\in \M$. 

Conversely, suppose that that for every projection $p\in \mathcal{D}$ with $\omega(p)<\infty$, we have that 
$$\omega(p dp xp)=\omega(p xp dp), \qquad d \in \mathcal{D}, x\in \M.$$ The collection of all projections in $\D$ with finite 
$\omega$-trace, is of course a net increasing to $\I$ when ordered with the natural ordering. By Theorem \ref{wcent}, there will for 
such projection $p$ exist a  normal conditional expectation $E_p: p \M p\to p \D p$, preserving the restriction of $\omega$ to 
$p \M p$. We claim that for $p_1 \geq p_0$, $(E_1)_{|p_0 \M p_0}=E_0$. 
This is a fairly straightforward consequence of the 
uniqueness assertion in Theorem \ref{wcent}, and the fact 
that $p_1\geq p_0$ ensures that $p_0 xp_0\in p_1 \M p_1$ and $p_0 dp_0\in p_1 \D p_1$, with $E_1(p_0 xp_0) = 
E_1(p_0 p_1 xp_1 p_0) = p_0 E_1(p_1 xp_1) p_0$.

We know that $\{p\in\mathbb{P}(\D):\omega(p)<\infty\}$ is $\sigma$-strong* convergent to $\I$. So for any $x\in \M$, $\{pxp: p\in\mathbb{P}(\D), \omega(p)<\infty\}$ is $\sigma$-strong* convergent to $x$. Now let $x\in \M$ be fixed. We show that similarly $\{E_p(p xp): p\in\mathbb{P}(\D), \omega(p)<\infty\}$ is $\sigma$-strong* convergent to some $d_x\in \D$. Observe that the set $\{E_p(p xp): p\in\mathbb{P}(\D), \omega(p)<\infty\}$ is norm-bounded, and hence relatively weak*-compact. Let $d_0$ be any weak* cluster point of the net $\{E_p(p xp): p\in\mathbb{P}(\D), \omega(p)<\infty\}$. By the relative compactness, there must exist some subnet $\{E_{p_\alpha}(p_{\alpha} xp_{\alpha})\}$ converging weak* to $p_0$. For any fixed $q\in\mathbb{P}(\D)$ with $\omega(q)<\infty$, the fact that $q E_{p}(pxp)q=E_q(q xq)$ for each $p\geq q$, ensures that $\{qE_p(p xp)q: p\in\mathbb{P}(\D), \omega(p)<\infty\}$ is uniformly convergent to  $E_q(q xq)$. But this net also has weak* limit point $q d_0 q$. Thus clearly $E_q(q xq)= q d_0 q$. Since $\{q d_0 q: q\in\mathbb{P}(\D), \omega(q)<\infty\}$ is $\sigma$-strong* convergent to $d_0$, the claim follows. 

We now define the map $E:\M\to \D$ by setting $E(x)=\lim_p \, E_p(p xp)$.  Clearly  
$$E(d) =\lim_p \, E_p(p dp)
= \lim_p \, pdp = d , \qquad d \in \mathcal{D}.$$  It easily follows that $E$ is a contractive unital positive projection.  Hence by Tomiyama's result $E$ is a conditional expectation.

We next show that $E$ is normal. Let $\{x_\alpha\}\subset \M_+$ be given with $x=\sup_\alpha x_\alpha\in \M_+$. By positivity $E(x)\geq \sup_\alpha E(x_\alpha)$. Suppose that $E(x)- \sup_\alpha E(x_\alpha)\neq 0$. We now equip $\D$ with the weight induced by the restriction of $\omega$ to $\D$. The semifiniteness of $(\D,\omega)$ then ensures that there must exist a non-zero subprojection $p$ of the support projection of $E(x)- \sup_\alpha E(x_\alpha)\geq 0$, for which $\omega(p)<\infty$, and $$E(pxp)- \sup_\alpha E(px_\alpha p)= p(E(x)- \sup_\alpha E(x_\alpha))p\neq 0.$$ But by construction $E_{|p \M p}=E_p$. Since $E_p$ is known to be normal, and 
$\sup_\alpha px_\alpha p=pxp$, it follows that $$E(pxp)- \sup_\alpha E(px_\alpha p)=E_p(pxp)- \sup_\alpha E_p(px_\alpha p)=0.$$ This is a contradiction. Hence $E$ must be normal. 

The fact that $E_{|p \M p}=E_p$, ensures that $E$ is locally $\omega$-preserving. The claims regarding uniqueness and faithfulness similarly follow from this fact. For example if it was indeed possible to find two distinct conditional expectations $E$ and $E_0$ fulfilling the criteria of the hypothesis, we would be able to find some $x\in \M_+$ for which $E(x)-E_0(x)\neq 0$. Thus there must then exist some non-zero subprojection $p$ of the support projection of $E(x)-E_0(x)$ for which we still have that $p(E(x)-E_0(x))p\neq 0$. But we then also have that $$p(E(x)-E_0(x))p = E(pxp)-E_0(pxp)=E_p(pxp)-E_p(pxp)=0,$$ a contradiction. Thus $E$ must be unique. 

If $\omega$ is faithful then  by  Theorem \ref{wcent} 
each $E_p$ is also faithful. Let $0\neq x\in\M_+$ be given. We know that $\{pxp: p\in\mathbb{P}(\D), \omega(p)<\infty\}$ is $\sigma$-strong* convergent to $x$. So for some $p\in\mathbb{P}(\D)$ with $\omega(p)<\infty$, we must have that $pxp\neq 0$. The faithfulness of $E_p$ then ensures that $0\neq E_p(pxp)=pE(x)p$, which in turn ensures that $E(x)\neq 0$. Hence $E$ is faithful.
\end{proof}

\begin{theorem} \label{excodsfex} 
Suppose that  $\D$ is a von Neumann subalgebra of  a von Neumann algebra $\M$.  The following are equivalent:  \begin{itemize}
\item [(1)]  $\D$ is semifinite (or equivalently, has a faithful normal semifinite trace $\tau$) and there exists a normal conditional expectation (resp.\ 
faithful normal conditional expectation) $\Psi : \M \to \mathcal{D}$.
\item [(2)]  $\M$ has a  normal (resp.\ 
faithful normal) weight $\omega$ 
with ${\mathfrak m}(\D)_\omega$ $\omega$-central  such that $\omega$ is semifinite  and faithful on $\D$.  
\item [(3)]  $\M$ has a  normal (resp.\  faithful normal) 
 weight $\omega$ 
with $\D$ locally $\omega$-central, such that $\omega$ is a semifinite  faithful trace 
 on $\D$.  
\end{itemize}
Moreover if  {\rm (2)} holds then 
we can choose $E$ in {\rm  (1)}
 to be $\omega$-preserving.  Such an $\omega$-preserving normal conditional expectation onto $\D$ is unique.  Also   $\Psi \mapsto \tau \circ \Psi$ is a bijective correspondence between the items in {\rm  (1)} and  {\rm  (2)}.  
 This correspondence  is also a bijection with the $\omega$ in {\rm  (3)} 
whose support projections commute with $\D$.
\end{theorem} 

\begin{proof}  (1) $\Rightarrow$ (2)\ Indeed if there exists a (faithful) normal conditional expectation $\Psi : \M \to \mathcal{D}$,
and if $\D$ has a fns trace $\tau$, then the composition 
$\omega=\tau\circ\Psi$ is a (faithful)  normal  semifinite weight  with ${\mathfrak m}(\D)_\omega$ $\omega$-central:
 $$\omega(xd) = \tau(\Psi(x) d)) =  \tau(d \Psi(x))) = \omega(dx), \qquad d \in {\mathfrak m}(\D)_\omega, x \in {\mathfrak p}(\M)_\omega .$$

(2)  $\Rightarrow$ (1)\ If $\M$ has a  
faithful normal weight $\omega$ 
with $\D \subset \M_\omega$, such that $\omega$ is semifinite on $\D$, then there exists a $\omega$-preserving 
faithful normal conditional expectation by \cite[Theorem IX.4.2]{tak2}.     
By Lemma \ref{istrD},  $\omega$ is a trace on $\D$.    

 Now suppose that $\M$ has a  
 normal weight $\omega$ such that $\omega$ is semifinite and faithful on $\D$, 
 with $\omega(xd) =  \omega(dx)$ for $d \in {\mathfrak m}(\D)_\omega, x \in {\mathfrak m}(\M)_\omega$.   
  Let $e$ be  the support projection of  $\omega$, then 
Lemma \ref{tuse} shows that  $\D \subset \M^\omega$,
 that $e \in \D'$ and $\D e$ is a von Neumann subalgebra of  $e \M e$, and
 the restriction of $\omega$  to  $\D$ (resp.\ $\D e$) is 
 a  normal   (resp.\ faithful normal) semifinite  trace.  
 
  By the earlier case there exists a (faithful) normal conditional expectation $\Psi : e \M e \to \mathcal{D} e$, 
 so that $\Psi(e \cdot e)$ is a normal conditional expectation from $\M$ onto $\D e$ as in the proof of Theorem  \ref{wcent}.     The canonical map $\D \mapsto \D e$ is faithful, since 
 $\omega(e d^* d e) = \omega(d^* d)$, and $\omega$ is faithful on $\D$. 
 As in the proof of Theorem  \ref{wcent},  the associated map $E : \M \to \D$ is a normal conditional expectation.
  Moreover $$\omega(E(x)) = \omega(L_e^{-1} \Psi(exe)) =  \omega(\Psi(exe)) = \omega(exe) = \omega(x), \qquad x \in \M_+.$$
   
 (2)  $\Rightarrow$  (3)\ We saw above that (2) implies that $\omega$ is a trace on $\D$.    Since $p \M p \subset {\mathfrak m}_\omega$
 (and similarly for $\D$), (3) is now clear.
 
 (3) $\Rightarrow$ (1)\ This follows from Theorem \ref{wcentw}.  
(Note that the `faithful case' follows from the `faithful case' of Theorem \ref{wcentw}.)

 The uniqueness of an $\omega$-preserving conditional expectation is just as in e.g.\ the proof of 
Theorem \ref{wcent}.    This easily gives the bijectivity of the map  $\Psi \mapsto \tau \circ \Psi$  between the items in {\rm  (1)} and  {\rm  (2)}.  
Finally if $e \in \D'$ then $\D$ is  locally $\omega$-central if and only if $\D$ is   $\omega$-central by Theorem  \ref{locis}. 
 \end{proof}

 As noted in an early paragraph in the introduction we may call $E$ in (1) in the last result {\em the (generalized) 
weight-preserving conditional expectation} from $\M$ to $\D$ associated with the 
normal weight
$\omega$ in (2) 
(by the bijective correspondence in the theorem).     We wrote it as $\mathbb{E}_\omega$ before (in the normal state case).

Putting the last two results and their proofs together  shows the  
weight-preserving   conditional expectation $E$
 associated with a weight $\omega$ in this way,  may be approximated by a net $(E_p)$ of
expectations on the  `finite' von Neumann algebras $p \M p$, for $\omega$-finite  $p \in \Pdb(\D)$. 
In Theorem \ref{locis} we also saw that $\omega$ can be approximated by a matching net of normal 
functionals: $\omega(x) = \sup_p \, \omega(pxp)  = \lim_p \, \omega(pxp)$ for $x \in \M_+.$
 
Another upshot of the last theorem is that for semifinite von Neumann subalgebras, the existence of conditional expectations necessitates 
 the $\omega$-centrality condition. 
   Thus in our later  generalization of the Hoffman-Rossi theorem to $\D$-characters for  a semifinite $\D$,
 it is natural  to assume that $\D \subset 
 \M_\omega$. 
   Of course for a normal state  the latter condition is just saying that $\omega(xd) = \omega(dx)$ for $x \in \M , d \in \D$. 
 In one of our `best Hoffman-Rossi theorems' (Theorem \ref{HRsfcent}), $\D$ is a subalgebra of a $\sigma$-finite von Neumann algebra $\M$ satisfying
 the latter condition.

For a normal positive map $E : \M \to \N$ we define the support projection ${\rm supp} \, (E)$ of $E$ to be 
the complement of the largest projection $p \in \M$ with $E(p) = 0$. All of the following result is no doubt 
known, even in this generality.  We furnish a proof 
amounting to a very  slight generalization of the ideas in \cite[Lemma 1.2]{ES}.

\begin{lemma} \label{suE} For a normal positive map $E : \M \to \N$ between von Neumann algebras
the support projection $z$ above of $E$ exists, and $E(x) =  E(zxz)$ for all $x \in \M$.
We have  $z = 1$ if and only if 
$E$ is faithful.   If $x \in \M_+$ then $E(x) = 0$ if and only if  $zxz = 0$.  

If in addition,  $\N \subset \M$, $E \circ E = E$, and $E$ is contractive, 
then $z E(x) = E(x) z$ for $x \in \M$.    \end{lemma} 

\begin{proof} Let $L_E = \{ x : E(x^* x) = 0 \}$, the left kernel 
of $E$.  For $x \in L_E$ we have  $E(y^* x) = 0$ for all $y \in \M$, 
since for all normal states $\varphi$ on $\N$ we have 
$|\varphi(E(y^* x))| = 0$ by the Cauchy-Schwarz inequality for $\varphi\circ E$.    Then $L_E$ is a weak* closed left ideal in $\M$. The support projection 
$f$ of $L_E$ is the ``largest projection'' mentioned in the statement preceding the Lemma.  
Indeed clearly $E(f) = 0$, and for a projection $p \in \M$ with $E(p) = 0$ 
we have $p \in L_E$ so that $p \leq f$.
 It follows that if $z = f^\perp$  and $x \in \M$ then 
$E(x) = E(xz)$.  Similarly $E(x) = E(zx) = E(zxz)$.   Clearly $z = 1$ if and only if $f = 0$ if and only if
$E$ is faithful.   If $x \in \M_+$ then $E(x) = 0$ if and only if $x = xf$, and 
if and only if $x = fxf$.  Indeed $E(x) = 0$ if and only if 
$x^{\frac{1}{2}} \in L_E$, and if and only if $x^{\frac{1}{2}} = x^{\frac{1}{2}} f$, which implies
$x = xf$ and $x = fxf$.    Conversely, if $x = fxf$ then $E(x) = 0$ by a fact a few lines back.
Thus $E(x) = 0$ if and only if $zxz = 0$ (since the latter implies that $x^{\frac{1}{2}} z = 0$ and
$x^{\frac{1}{2}} f = x^{\frac{1}{2}}$.  

Suppose that $\N \subset \M$, $E \circ E = E$, and $E$ is contractive. Let $\M^1 = \M \oplus \Cdb$. 
Then it is an exercise (or alternatively a consequence of \cite[Theorem 2.3 (2)]{BNjp}) that 
$$E'(x+  \lambda 1,\lambda) = (E(x) + \lambda 1 , \lambda), \qquad x \in \M , \lambda \in \Cdb,$$ 
defines a unital normal idempotent  positive projection on $\M^1$ extending $E$.  Since the projections in $\M^1$ are
the projections in $\M$ together with their orthocomplements in $\M^1$, it is easy to see that 
the support projection of $E'$ equals $z$ above.  Thus $zE(x) = z E'(x) = E'(x) z = E(x) z$ for $x \in \M$
by \cite[Lemma 1.2 (2)]{ES}.
\end{proof} 

\begin{theorem}  \label{wcentnonf} Let  $\mathcal{D} \subset \M$ be an inclusion of  von Neumann algebras. 
Let $\omega$ be a normal state on $\M$, and let $z$ be the  support projection of $\omega_{|\D}$ in $\D$. 
If we have that $\omega(ax)=\omega(xa)$ for every $a\in \mathcal{D}, x\in \M$, then there is 
 an $\omega$-preserving normal idempotent contractive 
completely positive $\D$-module map $\mathbb{E}_\omega$ from 
$\M$ onto  an ideal in  $\mathcal{D}$, with ${\rm supp} (\mathbb{E}_\omega) \leq z$. We have  $\mathbb{E}_\omega(\M) = z \D$. 
Indeed there is a unique $\omega$-preserving normal idempotent contractive $\D$-module map 
from $\M$ onto a $*$-subalgebra of $\mathcal{D}$ with ${\rm supp}  (E) \leq z$.  
\end{theorem} 

\begin{proof}  Let $z$ be the support projection in $\D$ for $\omega_{|\D}$. The left kernel of $\omega_{|\D}$ in $\D$ is an ideal 
in this case, so $z$ is central in $\D$.  Then $z \mathcal{D} \subset z \M z$ is an inclusion of  von Neumann algebras 
as in Theorem  \ref{wcent}, and by that result there is a unique   normal conditional expectation $F$ from $z \M z$ onto 
$z \mathcal{D}$ which preserves  $\omega_{|z \M z}$.    

Composing with the normal idempotent contractive $\D$-module  map $x \mapsto zxz$ we get  
a  normal idempotent contractive  map $\mathbb{E}_\omega$ from $\M$ onto $z \mathcal{D}$. \ Clearly $\mathbb{E}_\omega$  is a 
$\D$-module map since e.g.\ 
$$\mathbb{E}_\omega(d x) = F(zdxz) = zd F(zxz) = d F(zxz) = d \mathbb{E}_\omega(x), \qquad d \in \D, x \in \M .$$ 
We have 
$$\omega(\mathbb{E}_\omega(x)) = \omega_{|z \M z} (\mathbb{E}_\omega(zxz)) =\omega_{|z \M z} (zxz) = \omega(x) , \qquad x \in \M.$$ 
The last equality holds since $z$ is in the multiplicative domain of $\omega$ (see e.g.\ \cite[Proposition 1.3.11]{BLM}).  
Of course $\mathbb{E}_\omega(1-z) = 0$, so that ${\rm supp}  (\mathbb{E}_\omega) \leq z$.     

Suppose that $E$ is an $\omega$-preserving normal idempotent contractive $\D$-module map 
from $\M$ onto an ideal $R$ of $\mathcal{D}$ with ${\rm supp} \, (E) \leq z$.  Now $E(z) = z E(1) z \leq z$,
and $\omega(z- E(z)) = 0$, so that $z = E(z)$.
Also $E(1-z) = 0$  and so $E(1) =  E(z) = z$.
The restriction of $E$ to $z \M z$  is a conditional expectation onto $z \D$ which preserves $\omega_{|z \M z}$. 

By the uniqueness assertion in the first paragraph the restriction of $E$ to $z \M z$ 
is uniquely determined, and indeed  $E(x) = E(zxz) = F(zxz)  = \mathbb{E}_\omega(x)$ for $x \in \M_+$. 
\end{proof}  

\begin{remark} 
The condition in the last result that  ${\rm supp} \, (E) \leq z$ is 
necessary, as is clear from the following example. Let  $\M = M_3$ with $\omega(a) = a_{33}$, and let 
$\D = D_3$, the diagonal matrices.  Then $\mathbb{E}_\omega(a) = \omega(a) \, z$. 
However if $F$ is any idempotent contractive $D_2$-module  map from $M_2$ onto a $*$-subalgebra of $D_2$
then  $E(a) = F(a) \oplus \omega(a)$ is an $\omega$-preserving normal idempotent contractive $\D$-module map 
onto a $*$-subalgebra of $\D$.  Since there are many such maps $F$ in general,
there may be many such $E$.  Hence the `uniqueness' 
statement at the end of the theorem  is violated.
\end{remark}

We may call $\mathbb{E}_\omega$ in the last result {\em the 
weight-preserving conditional expectation} from $\M$ to $\D$ associated with the 
normal state 
$\omega$ on $\M$.

 \begin{corollary}  \label{unfco}  Let  $\mathcal{D} \subset \M$ be an inclusion of  von Neumann algebras.  Let $\omega$ be a normal state on $\M$
 such that $\omega(ax)=\omega(xa)$ for every $a\in \mathcal{D}, x\in \M$,
and let $z$ be the support in $\D$ of $\omega_{|\D}$. 
 There is a bijective correspondence between surjective normal  idempotent contractive $\D$-module maps $E : \M \to \D z$ with ${\rm supp}  (E) \leq z$, and  extensions of $\omega_{|\D}$ to normal states $\rho$ on all of $\M$ which have $\D$ in the centralizer.    
Under this correspondence, the support projections of $E$ and $\rho$ in $\M$ are the same, and 
$\rho = \rho \circ E = \omega \circ E$.    Also, such maps $E$ are completely positive.
 \end{corollary} 

\begin{proof} 
If $E : \M \to \D z$ is a surjective normal idempotent contractive $\D$-module map  with ${\rm supp}  (E) \leq z$ then
$\omega \circ E$ is a normal contractive functional with $\D$ in the centralizer, and $$\omega(E(d)) = \omega(E(dz)) = \omega(dz) = \omega(d),
\qquad d \in \D.$$ In particular $\omega(E(1)) = 1$ so that $\omega \circ E$ is a state.

Conversely, given an extension of $\omega_{|\D}$ to a normal state $\rho$ on  $\M$ which has $\D$ in the centralizer,
then by Theorem \ref{wcentnonf} there is a unique surjective $\rho$-preserving normal idempotent contractive 
completely positive $\D$-module map 
from $\M$ onto  $\mathcal{D} z$ with ${\rm supp}  (E) \leq z$.  

Note that $E(x) = 0$ if and only if $\omega(E(x)) = 0$,
for $x \in \M_+$.   Thus the support projections of $E$ and $\omega \circ E$ are the same.    
\end{proof}

\begin{remark}
1)\ The `conditional expectations' $E$ in Corollary \ref{unfco} need not be nicely related 
to the conditional expectation $\mathbb{E}_\omega$ in 
Theorem \ref{wcentnonf}, unlike the case mentioned early in the introduction where we discussed `weight functions' $h$,
 and e.g.\ in Remark  \ref{afterg} (1).
  A good example to see some of the issues that can arise in this `nonfaithful' case
is as follows.   Let $\M = M_3$ and $\D = \Cdb I_3 + \Cdb E_{33}$, and let $\omega(x) = x_{11}$.   The `conditional expectations' $E$ in Corollary \ref{unfco}
correspond to the density matrices in $M_2$ (i.e.\ positive trace 1 matrices in $M_2$).   Only some of these 
density matrices are related to 
$\mathbb{E}_\omega$ or to `Radon-Nikodym derivatives' of $\omega$, in contrast to the faithful case (see e.g.\ Corollary \ref{Hcongen}
or Corollary \ref{Hcon},  and the remark between these results).

If however the support projections of $\omega$ and $\omega_{| \D}$ agree, 
then there is much more that one can say.   Indeed we may cut down to $z \M z$ in this case, on which algebra 
$\omega$ and  $\omega_{| \D}$ are faithful, 
and on $z \M z$ simply apply the theory of von Neumann algebraic conditional expectations, and the later result 
Corollary \ref{Hcongen}.       

\smallskip

2)\  The proof above shows that every surjective normal  idempotent contractive $\D$-module map $E : \M \to \D z$ with ${\rm supp}  (E) \leq z$
is the 
`weight-preserving conditional expectation' on $\M$ associated with a state on $\M$ (namely the state $\omega \circ E$). 
\end{remark}

The following simple result  is similar to  \cite[Corollary 3.1]{BLv}. 

\begin{corollary} \label{HRf}  Suppose that  $\D$ is a von Neumann subalgebra of  a von Neumann algebra $\M$,
and that  $\M$ has a  normal weight $\omega$   such that $\omega$ is 
faithful on $\D$, and such that there exists an $\omega$-preserving conditional expectation 
$\mathbb{E}_\omega : \M \to \D$.
Let $\Phi : \A \to \mathcal{D}$ be an $\omega$-preserving unital  linear 
map on a unital 
subalgebra $\A$ of $\M$  containing $\D$, which is a $\mathcal{D}$-bimodule map.
 Then there is a unique  $\omega$-preserving normal conditional expectation $\M \to \mathcal{D}$ extending 
$\Phi$. 
\end{corollary} 

\begin{proof} The desired expectation is 
the map $\mathbb{E}_\omega$.
To see this let $d\in\mathcal{D}$ and $a\in \A$ be given. It follows from the hypothesis 
that 
$$\omega(d\Phi(a))= \omega(\Phi(da))=\omega(da)=\omega(\mathbb{E}_\omega(da))=\omega(d\mathbb{E}_\omega(a)).$$
Since $\omega$ is faithful on $\mathcal{D}$ and $d\in\mathcal{D}$ arbitrary, this equality suffices to show that $\Phi(a)=\mathbb{E}_\omega(a)$ as claimed. 
The uniqueness is similar, and 
follows as in the proof of Theorem \ref{wcent}.  
\end{proof}

\section{Weight functions and characterizations of normal conditional expectations} \label{Gcex}

Let $\D$ be a von Neumann subalgebra of a von Neumann algebra  $\M$, and let $\nu$ be a faithful normal 
semifinite weight on $\M$. 
We again recall that as in \cite{tak2}, given a positive operator $h$ affiliated to $\M_\nu$, we write $\nu_h$ for the weight $\lim_{\epsilon\searrow 0} \, \nu(h_\epsilon^{1/2} x h_\epsilon^{1/2})$.

We say that a normal conditional expectation $E : \M \to \D$ {\em commutes with} $\nu$ if 
$\nu \circ E$ commutes with $\nu$ in the sense of {\rm \cite{PT}} (so e.g.\ $\nu \circ E \circ \sigma^\nu_t = \nu \circ E$ for $t\in\mathbb{R}$).

\begin{lemma}  \label{Ecomm}  
Let $\D$ be a von Neumann subalgebra of von Neumann algebra $\M$, and let $\nu$ be a faithful normal 
semifinite weight on $\M$.  Then a normal conditional expectation $E : \M \to \D$ commutes with $\nu$ if and only if 
$E \circ \sigma^\nu_t = \sigma^\nu_t  \circ  E$ for all $t\in\mathbb{R}$. \end{lemma}

\begin{proof} 
Since  $\nu \circ \sigma^\nu_t = \nu$ the one direction is easy.  
Conversely, suppose that $\nu \circ E \circ \sigma^\nu_t = \nu \circ E$ for $t\in\mathbb{R}$. 
If $a\in \mathfrak{n}_{\nu\circ E}$ then 
$\nu\circ E(|\sigma_t^\nu(a)|^2)=\nu\circ E(|a|^2)$. So for any $t$, $\sigma_t^\nu(a)$ will then again belong to $\mathfrak{n}_{\nu\circ E}$. Similarly if $a\in \mathfrak{n}_{\nu\circ E}$, we may use the Kadison-Schwarz inequality to see that $\nu\circ E(|E(a)|^2)\leq \nu\circ E(E(|a|^2)) =\nu\circ E(|a|^2)$. Thus $E(a)$ will then again belong to $\mathfrak{n}_{\nu\circ E}$. For any $a\in \mathfrak{n}_{\nu\circ E}$ and any $t\in\mathbb{R}$ 
each of the terms in the expansion of $\nu(|E(\sigma_t^\nu(a)) -\sigma_t^\nu(E(a))|^2)$
 as $$\nu(|E(\sigma_t^\nu(a))|^2) - \nu(E(\sigma_t^\nu(a))^*\sigma_t^\nu(E(a))) - \nu(\sigma_t^\nu(E(a))^*E(\sigma_t^\nu(a)))+\nu(\sigma_t^\nu(|E(a)|^2)),$$ is  finite. Repeated use of the facts that $\nu \circ \sigma^\nu_t = \nu$ and $\nu \circ E \circ \sigma^\nu_t = \nu\circ E$
then shows that
\begin{equation}\label{eq-commexp}
\nu(|E(\sigma_t^\nu(a)) -\sigma_t^\nu(E(a))|^2) = \nu(|E(\sigma_t^\nu(a))|^2) - \nu(|E(a)|^2) .
\end{equation}
We therefore have  $\nu(|E(\sigma_t^\nu(a))|^2) - \nu(|E(a)|^2)$ nonnegative for each $t\in\mathbb{R}$. But since 
$\sigma_t^\nu(a)\in \mathfrak{n}_{\nu\circ E}$ for any $t$, we must by symmetry also have that $$0\leq\nu(|E(\sigma_{-t}^\nu(\sigma_t^\nu(a)))|^2) - \nu(|E(\sigma_t^\nu(a))|^2)=\nu(|E(a)|^2) - \nu(|E(\sigma_t^\nu(a))|^2)$$ for each $t$. Thus $\nu(|E(\sigma_t^\nu(a))|^2) - \nu(|E(a)|^2) = 0$ for each $t\in\mathbb{R}$, in which case we will  have by equation (\ref{eq-commexp})
that $\nu(|E(\sigma_t^\nu(a)) -\sigma_t^\nu(E(a))|^2)=0$ for each $t$. The faithfulness of $\nu$ then ensures that $E(\sigma_t^\nu(a)) = \sigma_t^\nu(E(a))$ for all $a\in \mathfrak{n}_{\nu\circ E}$. By normality this equality holds on all of $\M$. Thus $E \circ \sigma^\nu_t = \sigma^\nu_t  \circ  E$.
\end{proof}

Thus  $E$  commutes with $\nu$ if  and only if $E$  commutes with $(\sigma^\nu_t)$.   In particular, 
by 
Remark \ref{crossprod2} (or \cite[Proposition 4.9]{haag-OV1}), $E_{\D}$ commutes with 
$\nu$, where $E_{\D}$ is  the unique weight-preserving conditional expectation of $\M$ onto $\D$ associated with 
the weight $\nu$ (from \cite[Theorem IX.4.2]{tak2}). 

We recall that two selfadjoint unbounded operators $S, T$ on $H$ {\em commute strongly} if  all of their Borel spectral projections commute.
This is equivalent to saying that there is a commutative von Neumann algebra $\N$ on $H$ which both $S$ and $T$ are affiliated with (see e.g.\ 
\cite[Theorem 5.6.15]{KR1}).  For the one direction of this equivalence take the von Neumann algebra $\N$ generated by these two commuting families
of spectral projections (we have $S \, \eta \, W^*(S) \subset \N$ so that $S \, \eta \, \N$, and similarly for $T$).  For the other direction recall that 
$W^*(S)$ is the smallest von Neumann algebra on $H$ with which  $S$ is  affiliated, and similarly for $T$.  So all the spectral projections of $S$ and $T$ are in $\N$, 
and hence they commute.

Before proving the main theorem of this section, we need some insight into how one may describe the weights of the form described in equation (\ref{PTRN}) in Haagerup $L^p$-space terms. This is the topic of the next result.

\begin{proposition} \label{commhk} Let $\nu$ be a faithful normal semifinite weight on $\M$. A possibly unbounded positive operator  $h$ is affiliated with $\M_\nu^+$ if and only if $h \,  \eta \, \M$ and 
 $h$ strongly commutes with $k_\nu$, where $k_\nu$ denotes the density $\frac{d\widetilde{\nu}}{d\tau}$ of the dual weight $\widetilde{\nu}$ on $\M\rtimes_\nu\mathbb{R}$.   In particular, 
if  $\D$ is a von Neumann subalgebra of $\M$ then $\D$ is contained in the centralizer of $\nu$ if and only if 
$k_\nu$ commutes strongly with $\D$.   
Moreover for any positive operator $h$ affiliated with $\M_\nu^+$, the dual weight of $\nu_h$ is given by the formula $\lim_{\epsilon\searrow 0} \widetilde\nu(h_\epsilon^{1/2}\cdot h_\epsilon^{1/2})$, with the strong product $k=k_\nu\cdot h$ corresponding to the density $\frac{d\widetilde{\nu_h}}{d\tau}$ of the dual weight $\widetilde{\nu_h}$ on $\M\rtimes_\nu\mathbb{R}$.
\end{proposition}

\begin{proof}  Indeed if $h \,  \eta \, \M$  then affiliation of $h$ with $\M_\nu$ is equivalent to the condition that $h\chi_{[0,n]}(h)\in \M_\nu$ for each $n\in \mathbb{N}$, which in turn is equivalent to requiring that $\sigma^\nu_t(h\chi_{[0,n]}(h))=h\chi_{[0,n]}(h)$ for all $t\in \mathbb{R}$ and all $n$. On regarding $\M$ as a subalgebra of the crossed product $\M\rtimes_\nu\mathbb{R}$, the modular automorphism group will then be implemented by $k_\nu^{it}$. The above equality 
may be reformulated as the claim that $h\chi_{[0,n]}(h)=k_\nu^{it}h\chi_{[0,n]}(h)k_\nu^{-it}$ for all $t\in \mathbb{R}$ and all $n\in \mathbb{N}$. 
 This clearly shows that $h\chi_{[0,n]}(h)$ and $k_\nu^{it}$ will commute for any $n\in\mathbb{N}$ and any 
$t\in \mathbb{R}$, which in turn is equivalent to the fact that the spectral projections of $h$ and $k_\nu$ all commute. 

In particular, $k_\nu$ commutes strongly with $\D$ if and only if for each $d\in \D$ and each $t\in \mathbb{R}$ we have that $k_\nu^{it}dk_\nu^{-it}=d$.  This holds if and only if $\sigma_t^\nu(d)=d$, that is,  if and only if $\D \subset \M_\nu$.

We next show that $\frac{d\widetilde{\nu_h}}{d\widetilde{\nu}}$ exists and equals $h$. Let $T$ denote the operator valued weight from the 
extended positive part of $\M\rtimes_\nu\mathbb{R}$ to that of $\M$. If we combine \cite[Proposition 4.9]{haag-OV1} with the fact that 
$\nu_h$ commutes with $\nu$, we have that $\widetilde{\nu_h}\circ\sigma_t^{\widetilde{\nu}}=\nu_h\circ T\circ \sigma_t^{\widetilde{\nu}} = \nu_h\circ \sigma_t^{\nu} \circ T = \nu_h\circ T= \widetilde{\nu_h}$ for all $t\in\mathbb{R}$. Thus $\widetilde{\nu_h}$ must then 
commute with $\widetilde{\nu}$, which ensures that $\widetilde{h}=\frac{d\widetilde{\nu_h}}{d\widetilde{\nu}}$ exists as an operator 
affiliated to $(\M\rtimes_\nu\mathbb{R})_{\widetilde{\nu}}$. We go on to prove that $\widetilde{h}=h$. In the case where $\nu_h$ is 
faithful this easily follows from the fact that $$h^{it}=(D\nu_h:D\nu)_t=(D\widetilde{\nu_h}:D\widetilde{\nu})_t=\widetilde{h}^{it}\mbox{ for all } t\in \mathbb{R}.$$
(The first and last equality follows from \cite[Theorem VIII.2.11]{tak2} and the definition of the cocycle derivative; 
see the proof of \cite[Corollary VIII.3.6]{tak2}. The second equality follows from \cite[Theorem 4.7]{haag-OV1}.) In the case where 
$\nu_h$ is not faithful (i.e.\ $e_0= \mathrm{supp}(\nu_h)$ is strictly smaller than $\I$), we may pass to the weight 
$(\nu_h)_0 =\nu_h+(\I-e_0).\nu.(\I-e_0)$. It is an exercise to see that this is a faithful normal semifinite weight on $\M$ 
with $\frac{d(\nu_h)_0}{d\nu}=h+(\I-e_0)$. The dual weight will then be 
$\widetilde{(\nu_h)_0} = \widetilde{\nu_h}+(\I-e_0).\widetilde{\nu}.(\I-e_0)$  
using one of \cite[Lemma II.1]{terp} or \cite[Theorem 6.55]{GLnotes}. 
We also have  $h+(\I-e_0)=\frac{d\widetilde{(\nu_h)_0}}{d\widetilde{\nu}}$, by the `faithful case' above. But clearly 
$\frac{d(\I-e_0).\widetilde{\nu}.(\I-e_0)}{d\widetilde{\nu}}=(\I-e_0)$. We also know from 
either \cite[Lemmas II.1 and II.2, and Proposition II.4]{terp} or \cite[Theorem 3.24, Theorem 6.55 and Proposition 6.67]{GLnotes}, that $$\mathrm{supp}(\widetilde{\nu_h})=\mathrm{supp}(\nu_h)=\mathrm{supp}(h) = e_0.$$
Therefore $e_0.\widetilde{(\nu_h)_0}.e_0= \widetilde{\nu_h}$. Using all 
these facts it is now possible to verify that $\frac{d\widetilde{\nu_h}}{d\widetilde{\nu}}=e_0(h+(\I-e_0))e_0=h$. 
This fact then ensures that $\widetilde{\nu_h}=\widetilde{\nu}_h$, that is that the dual weight of $\nu_h$ is given by the formula $\lim_{\epsilon\searrow 0} \widetilde\nu(h_\epsilon^{1/2}\cdot h_\epsilon^{1/2})$.

Since the centralizer $(\M\rtimes_\nu\mathbb{R})_\tau$ is all of $\M\rtimes_\nu\mathbb{R}$, we may now apply the chain rule described in \cite[Proposition 4.3]{PT} to the pair $k_\nu$ and $h$, to see that $hk_\nu=\frac{d\widetilde{\nu_h}}{d\widetilde{\nu}}\frac{d\widetilde{\nu}}{d\tau}=\frac{d\widetilde{\nu_h}}{d\tau}=k$ as required.
\end{proof}

\begin{remark}
Let $\nu$ be a faithful normal semifinite weight on $\M$ and let $h$ be a positive operator affiliated to $\M_\nu$. 
For $x \in \M_+$ it is tempting to assign a meaning to the product  $h^{\frac{1}{2}} x h^{\frac{1}{2}}$, perhaps as an 
element $b$ of the extended positive part 
 $\widehat{\M}_+$ of  $\M$, and to then interpret $\nu_h(x)$ as the canonical 
extension of $\nu$ to $\widehat{\M}_+$ applied to $b$. 

In this regard the conclusions regarding $h$, $k_\nu$ and $k=k_\nu\cdot h$ noted in the preceding Proposition, may be 
interpreted as a realisation of exactly this objective. To see why this is so, we shall for the sake of simplicity pass to 
the case where $\nu$ is a normal state, and $\nu(h)<\infty$. (Note that the weight $\nu$ has a well-defined action on $h$ 
since it canonically extends to $\widehat{\M}_+$.) Since $\nu_h(\I)=\lim_{\epsilon\searrow}\nu(h_\epsilon)=\nu(h)$, this 
assumption ensures that $\nu_h$ is a normal functional. The theory of Haagerup $L^p$-spaces then ensures that for any 
$x\in \M_+$, we will have that $\nu(x)=tr(k_\nu^{\frac{1}{2}} xk_\nu^{\frac{1}{2}})$ and that $\nu_h(x)= 
tr(k^{\frac{1}{2}} xk^{\frac{1}{2}})$. When these equalities are considered alongside the fact that $k^{\frac{1}{2}}= k_\nu^{\frac{1}{2}}\cdot h^{\frac{1}{2}}$, it is then clear that the equality $\nu_h(x)=tr(k^{\frac{1}{2}} xk^{\frac{1}{2}})$ 
may be interpreted as a `rigorization' of the formal equality $\nu_h(x) = \nu(h^{\frac{1}{2}}xh^{\frac{1}{2}})$.

In the case where $\nu$ is a trace, we can achieve this objective more directly, by using basic facts about noncommutative 
integration with respect to a trace from e.g.\ Section IX.2 in \cite{tak2} 
(see also the fact in \cite[Remark 4.15]{GLnotes}, which informs us the element $b$ above 
 will be in $L^1(\M,\nu)$). Note that if we view $h$ as an element of $\widehat{\M}_+$ then the product $c = x^{\frac{1}{2}} \cdot h \cdot x^{\frac{1}{2}}$, 
defined as in \cite[Definition IX.4.6]{tak2}, is in $\widehat{\M}_+$. The trace $\nu$ allows for an extension to 
$\widehat{\M}_+$, and for this extension normality of the trace ensures that
\begin{equation}\label{PTremark}\nu_h(x) = \lim_{\epsilon \searrow 0}
\, \nu(h_\epsilon^{\frac{1}{2}} x h_\epsilon^{\frac{1}{2}}) = 
\lim_{\epsilon \searrow 0}
\, \nu(x^{\frac{1}{2}} h_\epsilon x^{\frac{1}{2}})=\nu(c). 
\end{equation}
If this quantity is finite, that is if $x\in \mathfrak{p}_{\nu_h}$, then we know from either \cite[Remark 1.2]{LabCo} or \cite[Remark 4.15]{GLnotes} that 
the 
 product $c=x^{\frac{1}{2}} \cdot h \cdot x^{\frac{1}{2}}$ above is in $L^1(\M,\nu)$. But if $x^{\frac{1}{2}}\cdot h \cdot x^{\frac{1}{2}}$ is a densely defined operator, then using the ideas in the proof of either \cite[Lemma 7.39]{GLnotes} 
or \cite[Lemma 2.1]{GL2}, shows that $h^{\frac{1}{2}} x^{\frac{1}{2}}$ is densely defined with $c=|h^{\frac{1}{2}} x^{\frac{1}{2}}|^2$. Thus if $\nu_h(x)$ is finite, then $h^{\frac{1}{2}} x^{\frac{1}{2}}\in L^2(\M,\nu)$. We may then 
use the tracial property of $\nu$ to conclude that 
$$\nu_h(x)=\nu(c)= \nu((h^{\frac{1}{2}} x^{\frac{1}{2}})^*(h^{\frac{1}{2}} x^{\frac{1}{2}})) =\nu((h^{\frac{1}{2}} x^{\frac{1}{2}})(h^{\frac{1}{2}} x^{\frac{1}{2}})^*),$$ 
giving what was claimed at the start of this remark with $b = (h^{\frac{1}{2}} x^{\frac{1}{2}})(h^{\frac{1}{2}} x^{\frac{1}{2}})^*$. 
If conversely $h^{\frac{1}{2}} x^{\frac{1}{2}}$ is densely defined and in $L^2(\M,\nu)$, then one may then use the 
tracial property alongside equation (\ref{PTremark}) to see that 
$\nu_h(x) < \infty$.

Finally we point out that the argument in this remark can be considerably simplified, 
 and slightly strengthened, 
if $h$ is also $\nu$-measurable: in this case for any $x\in \M_+$, the operators $x^{\frac{1}{2}}h_\epsilon x^{\frac{1}{2}} = 
 |h_\epsilon^{\frac{1}{2}}x^{\frac{1}{2}}|^2$ are known to increase to the operator $x^{\frac{1}{2}}h x^{\frac{1}{2}} = |h^{\frac{1}{2}}x^{\frac{1}{2}}|^2$ inside the algebra 
of $\nu$-measurable operators as $\epsilon \searrow 0$ \cite[Proposition 2.63]{GLnotes}. Normality and the tracial property of the extension of $\nu$ to this algebra \cite[Proposition 4.17]{GLnotes} then ensure that  $$\nu_h(x) = \lim_{\epsilon \searrow 0}
\, \nu(h_\epsilon^{\frac{1}{2}} x h_\epsilon^{\frac{1}{2}})=\nu(x^{\frac{1}{2}}h x^{\frac{1}{2}})=\nu(h^{\frac{1}{2}}x h^{\frac{1}{2}}) ,$$
for {\em all} $x \in \M_+$.
\end{remark}

\begin{theorem} \label{gencor} Let $\D$ be a von Neumann subalgebra of $\M$, and $\nu$ a faithful normal weight on $\M$, which is semifinite on $\D$, and satisfies $\sigma^\nu_t(\D) \subset \D$ for all $t \in \Rdb$. There is a bijective correspondence between the following objects:
\begin{enumerate}
\item[(a)] Normal conditional expectations $E$ onto $\D$ 
which commute with $(\sigma^\nu_t)$ in the sense above,
\item[(b)] densities $h \,  \eta \, \M_\nu^+$ 
which commute with all elements of $\D$, and 
which satisfy $\nu_h = \nu$ on $\D_+$, 
\item[(b$^{\prime}$)] densities $h  \, \eta \,  \M_\nu^+$ which 
commute with all elements of $\D$, and which 
satisfy $E_{\D}(h) = 1$
 (where the action on $h$ is by the normal extension of $E_{\D}$ to the extended positive part of $\M$ in e.g.\
 {\rm \cite[Proposition 3.1]{Gol}}), 
\item[(b$^{\prime \prime}$)] densities $k \in {}^\eta L^1_+(\M)$ which commute strongly with 
the element $k_\nu$ in the last result, for which $h=k.k_\nu^{-1}$ commutes with all elements of $\D$, and which satisfy $E_{\D}(k) = E_{\D}(k_\nu)$
(where the action on $k_\nu$ and $k$ is by the normal extension of $E_{\D}$ to the extended positive part of $\M\rtimes_\nu\mathbb{R}$), 
\item[(c)] extensions of $\nu_{|\D}$ to normal semifinite weights $\rho$ on all of $\M$ which commute with 
$\nu$ in the sense of  the Introduction, whose support projection $z$ commutes with 
$\D$ and satisfies $\sigma^{\rho}_t(\D z) \subset \D z$ for all $t \in \Rdb$.  
\end{enumerate} 
With respect to the correspondences above we also have the formulae  $$\rho = \nu \circ E = \nu_h, \; \; \; \; \; \; \;  \; \; \; \; h = \frac{d \rho}{d\nu}=k.k_\nu^{-1},$$ and 
$$E(x) =\lim_{\epsilon\searrow 0} \; E_{\D}(h_\epsilon^{\frac{1}{2}}  
\, x \, h_\epsilon^{\frac{1}{2}} ), \qquad x \in \M_+,$$ where $h_\epsilon=h(\I+\epsilon h)^{-1}$ 
and  with the last limit being a ($\sigma$-strong*) limit of an increasing net in $\M$. 
\end{theorem}

\begin{proof}   By \cite[Theorem IX.4.2]{tak2} there is a unique $\nu$-preserving normal conditional expectation 
$E_{\D} : \M \to \D$.  
We noted immediately after Lemma \ref{Ecomm} that $E_{\D}$ commutes with 
$(\sigma_t^{\nu})$. 

(a)  $\Rightarrow$ (c)  \ 
Given a normal conditional expectation $E$ onto $\D$ 
which commutes with $(\sigma^\nu_t)$, let $\rho = \nu \circ E$.  Then $\rho \circ E = \rho$.   
This is a normal weight on $\M$ extending $\nu_{| \D}$, for which we have that $\rho\circ\sigma^\nu_t=\rho$ since $E$ commutes with $\sigma^\nu_t$. 
Thus $\rho$ commutes with $\nu$  in the sense of {\rm \cite{PT}} (see also  \cite[Corollary VIII.3.6]{tak2}, although 
here all weights are faithful). 
We recall the support projection $z = {\rm supp} (E)$ from Lemma \ref{suE}.
Let $z$ be the support projection of $E$.   It is also the support projection 
of $\rho$, since $\nu$ is faithful (so that  for 
 $x \in \M_+$ we have $E(x) = 0$ if and only if $\nu(E(x)) = 0$). 
 
The fact that $\rho$ is semifinite follows from the facts that
$\nu_{|\D}=\rho_{|\D}$ is semifinite and $\rho\circ E=\rho$. We sketch the proof of this claim: 
  We firstly note that the left ideal $\mathfrak{n}(\D)_\nu$ admits a right approximate identity $( f_\lambda )$ of positive contractive elements (see for example 
\cite[Proposition 2.2.18]{BR}) which by \cite[Lemma 2.4.19]{BR} must converge $\sigma$-strongly to some contractive element $p$ of $\M$. Since $\nu_{|\D}$ is faithful on $\D$, this element in fact turns out to be the support projection of $\nu_{|\D}$
(in the sense of \cite{tak2}, see p.\ 57 there), which in this case is $\I$ since $\nu_{|\D}$ is semifinite 
on $\D$. Although we shall not need this fact, we pause to note that this approximate identity may in fact be chosen to consist entirely of entire analytic elements of $\mathfrak{n}(\D)_\nu^+$ by \cite[Lemma 9]{Terp2}. One may now use this approximate identity to show that 
$\mathfrak{m}_\rho^+$ is weak* dense in $\M^+$, which in turn ensures that $\rho$ is semifinite. In particular given any $a\in\M^+$, 
the $\sigma$-strong convergence of $\{f_\lambda\}$ to $\I$, ensures that $( f_\lambda af_\lambda )$ converges $\sigma$-weakly (weak*) to $a$. For any $\lambda$, we moreover have that 
$$\rho(f_\lambda af_\lambda)=\nu(E(f_\lambda af_\lambda))=\nu(f_\lambda E(a)f_\lambda)\leq \|a\| \nu(f_\lambda^2)<\infty .$$

By Lemma \ref{suE}, $z$ commutes with $\D$.  
The map $\D \to \D z$ of multiplication by $z$ is a surjective normal $*$-homomorphism with right inverse $E$.
Hence $\D \cong \D z$. 
Then $\rho$ restricts to a faithful normal weight on $z \M z$, which is semifinite on $z \M z$ and $\D z$
since  if $(e_t)$ is an increasing
$\nu$-finite net in $\D$ with limit $1$ then $(z e_t)$  is an increasing
$\rho$-finite and $\nu$-finite net in $\D z$ with limit $z$.  
 Define $E'(zxz) = E(x) z = E(zxz) z$.  
Then $E' : z \M z \to \D z$ is a $\rho$-preserving normal conditional expectation.  
 By \cite[Theorem VIII.2.6]{tak2} 
we have $\sigma^\rho_t(\D z) \subset \D z$ for all $t \in \Rdb$.   

(c) $\Rightarrow$ (b)\ Now suppose that $\rho$ is a normal semifinite weight on $\M$ which extends $\nu_{|\D}$, 
and commutes with 
$\nu$. The commutation with $\nu$, ensures that $h=\frac{d\rho}{d\nu}$ exists, and that $h \, \eta \, \M_\nu$ 
\cite{PT}.   We have $\nu_h(d) = \rho(d) = \nu(d)$ for $d \in \D_+$, so that $\nu_h = \nu$ on $\D_+$.

 Let  $z_\rho=\mathrm{supp}(\rho)=\mathrm{supp}(h)$, which we sometimes write as $z$.  Note that as in a proof two paragraphs above,  
$z \M z$ and $\D z$ are von Neumann subalgebras of $z \M z$ on which $\rho$ and $\nu$ are semifinite  faithful normal weights.
Since  $h=\frac{d\rho}{d\nu}$ is affiliated to $\M_\nu$,   by the Borel functional calculus  we have  $z_\rho=\chi_{(0,\infty)}(h)\in \M_\nu$. It  remains to show that  $h$ commutes with $\D$. 
Let $h_n=h\chi_{[0,n]}(h)$, $f = z_\rho^\perp$, $\rho' = \rho +\nu_{f}$ and $h_0 = h + f$. It is clear that $(h_n)$ increases to $h$. So 
by \cite[Propositions 4.1 \& 4.2]{PT}, we have 
$$\rho' = \nu_h+ \nu_{f}=\sup_{n\geq 1} \, \nu_{h_n}+\nu_{f} =\sup_{n\geq 1} \, [\nu_{h_n}+\nu_{f}]= \sup_{n\geq 1} \, \nu_{h_n+ f} = \nu_{h_0}.$$ Since $\rho$ is faithful on 
$z_\rho\M z_\rho$ and $\nu_{f}$ is faithful on $f \M f$, the sum  $\rho'$
is faithful on all of $\M$. By \cite[Theorem  4.6]{PT} 
the modular automorphism group of $\rho'$ is $h_0^{it} \, \sigma_t^\nu(\cdot) \, h_0^{-it}$. Now $h_0^{it}=h^{it}+f^{it}$ since  
$hf =0$. So the modular group  of $\rho'$  is 
$(h^{it}+f^{it}) \, \sigma_t^\nu(\cdot) \,  (h^{-it}+f^{-it})$ for all $t$.

Theorem 4.6 in \cite{PT} ensures that the modular automorphism group of $\nu_f$ is $f^{it} \, \sigma_t^\nu (\cdot) \,  f^{-it}$. Since $\sigma_t^\nu$ preserves $\D$ and $z = z_\rho$ is a fixed point of $\sigma_t^\nu$, this automorphism group clearly preserves $\D f$. So there exists a faithful normal conditional expectation $E_{f}$ from $f \M f$ onto $\D f$
such that $\nu_f \circ E_{f} =\nu_f$. Again by \cite[Theorem 4.6]{PT}, the modular group of $\rho$ is $h^{it}
\, \sigma_t^\nu(\cdot) \,  h^{-it}$, which by assumption preserves $\D z$. Since $\rho$ is faithful on $z \M z$, this then ensures that there exists a unique faithful normal conditional expectation $E_z$ from 
$z \M z$ onto $\D z$ for which we have that $\rho_{|z M_z} = \rho_{|z \M z} \, \circ \, E_z$. It is now an exercise to see that the map $E_1:\M\to \D$
defined by $E_1(a) = E_z(zaz)+E_{f}(faf)$ is a faithful normal conditional expectation onto $\D$ for which we have that $\rho' \, \circ \,
E_1= \rho'$. By \cite[Corollary IX.4.22]{tak2}, the modular automorphism group of the restriction of $\rho'$ to $\D$ is just the restriction of $(h^{it}+f^{it})\, \sigma_t^\nu (\cdot) \, (h^{-it}+f^{-it})$ to $\D$. Since $z_\rho$ commutes with $\D$ and $\sigma_t^\nu$ preserves $\D$, this group simplifies to $h^{it}\sigma_t^\nu(\cdot)h^{-it}+f^{it}\sigma_t^\nu(\cdot) f^{-it}$ on $\D$. But on $\D$, $\rho' = \rho+\nu_f$ agrees with $\nu_g$ where
$g = 2\I-z_\rho$. So this group must agree with the modular group of $(\nu_g)_{| \D}$. We proceed to describe that group.

We again use \cite[Theorem 4.6]{PT} to see that the modular automorphism group of $\nu_{g}$ is $g^{it}
\, \sigma_t^\nu(\cdot) \, g^{-it}$. Recall that $\sigma_t^\nu$ preserves $\D$ and that $z_\rho$ commutes with $\D$ and that $z_\rho$ is a fixed point of $\sigma_t^\nu$. So for any $d\in \D$ we have that  $g^{it}
\, \sigma_t^\nu(d) \, g^{-it}=\sigma_t^\nu(d)$. So this modular group preserves $\D$. Thus by \cite[Theorem IX.4.2 \& Corollary
 IX.4.22]{tak2}, the modular group of $\nu_{g}|_{\D}$, is just the restriction of $g^{it}
\, \sigma_t^\nu(\cdot) \, g^{-it}$ to $\D$, which as we have just seen is just $\sigma_t^\nu$. When considering this fact alongside what we showed in the previous paragraph, it is clear that $h^{it}\sigma_t^\nu(\cdot)h^{-it}+ f^{it}
\, \sigma_t^\nu(\cdot) \, f^{-it}$ agrees with $\sigma_t^\nu$ on $\D$. So on multiplying with $z_\rho$, we have that $h^{it}\sigma_t^\nu(\cdot)h^{-it}= \sigma_t^\nu(\cdot) z_\rho$ on $\D$. Taking into account that $\sigma_t^\nu$ preserves $\D$, we have that 
$$h^{it}dh^{-it}=h^{it}\sigma_t^\nu(\sigma_{-t}^\nu(d))h^{-it}=\sigma_t^\nu(\sigma_{-t}^\nu(d))z_\rho=dz_\rho, \qquad d\in \D, t \in \Rdb,$$
 or equivalently that $h^{it}d=dh^{it}$ for all $t$. Thus $h$ commutes strongly with each $d\in \D$.
 
(b) $\Leftrightarrow$ (b$^{\prime}$)\ 
Since $E_{\D}$ commutes with the modular group $\sigma_t^\nu$ and each $x \in \M_\nu$ is a fixed point of this modular group, it now clearly follows that $\sigma_t^\nu (E_{\D}(x)) = E_{\D}(x)$ for each $x \in \M_\nu$.  Thus $E_{\D}(\M_\nu) \subset \D \cap
\M_\nu \subset \D_\nu$.   Let $E'$ be $E_{\D}$ viewed as a conditional expectation from $\M_\nu$ onto $\D_\nu$.
The restriction of $\nu$ to $\M_\nu$ is a  a faithful normal  weight.  It is also a trace (see III.4.6.2 in \cite{Bla}).   

Most of the following part of the proof consists of a review of some aspects of Haagerup's theory 
of the extended positive part of 
a von Neumann algebra \cite{haag-OV1}, and 
establishing a general consequence of \cite[Proposition IX.4.11]{tak2} (which may be known to some 
experts).
Write $\hat{E}_{\D}$ for the canonical extension of $E_{\D}$ 
to the extended positive part $\hat{\M}_+$ of $\M$.  
We may regard the extended positive part $\widehat{\M_\nu}_+$ as a subspace of 
$\hat{\M}_+$, and may regard the extended positive part $\widehat{\D_\nu}_+$ as a subspace of 
$\hat{\D}_+$.   It is easy to see that the restriction of $\hat{E}_{\D}$ to the extended positive part of $\M_\nu$
equals the canonical extension $\widehat{E'}$ of $E'$ to $\widehat{\M_\nu}_+$.    
Let ${\mathcal S}_\nu$ denote  the set of selfadjoint positive operators $a$
affiliated with $\M_\nu$.  These correspond to certain weights $\nu_a$ on $\M$ by the Pedersen-Takesaki Radon-Nikodym 
correspondence \cite{PT}.  

By \cite[Proposition IX.4.11]{tak2} we may view an element $m$ in $\widehat{\M_\nu}_+$ (resp.\ in $\widehat{\D_\nu}_+$) as a  
normal weight $\nu_m$ on $\M$  (resp.\  normal weight $(\nu_{| \D})_m$ on $\D$).   Then  $\widehat{E'}$  induces 
a `normal' order preserving map $j$ from ${\mathcal S}_\nu$ into the set of normal weights  on $\D$.    
Indeed $j(a) = (\nu_{| \D})_m$  with $m = \widehat{E'}(a)$.  
Let $i$ be the map from ${\mathcal S}_\nu$ into the weights  on $\D$ defined by $i(a) = (\nu_a)_{| \D_+}$.
 We claim that $j = i$.  
 If  $x$ and $x^{\frac{1}{2}}$ are in $(\M_\nu)_+$ then we have 
 $$i(x)(d) = (\nu_x)(d) = \nu(x^{\frac{1}{2}} dx^{\frac{1}{2}}) = \nu(d^{\frac{1}{2}} x d^{\frac{1}{2}}) = \nu(d^{\frac{1}{2}}  E_{\D}(x) d^{\frac{1}{2}} ), \qquad d \in (\mathfrak{m}(\D)_{\nu_{|\D }})_+ . $$
This equals $\nu_{E_{\D}(x)}(d)$ since $E_{\D}(x) \in \widehat{\D_\nu}_+$.  Since normal weights on a semifinite algebra
 are determined by their action on $\mathfrak{m}_+$, we 
  have $j(x)= (\nu_{| \D})_{E_{\D}(x)} = i(x)$.    So $i = j$ on $(\M_\nu)_+$.   
If $a_1 \leq a_2$ (in the sense of unbounded operators as on p.\ 62 in \cite{PT}) 
 in ${\mathcal S}_\nu$ then 
$\nu_{a_1} \leq \nu_{a_2}$ by  \cite{PT}.  Hence $i(a_t) \leq  i(a)$.  So $i$ is order preserving.  
If $a_t \nearrow a$ (in the sense of unbounded operators  as on p.\ 62  \cite{PT}) in ${\mathcal S}_\nu$ 
then $\nu_{a_t} \nearrow \nu_{a}$ by  \cite{PT}.  Hence $i(a_t) \nearrow  i(a)$.  So $i$ is normal.
Since every element of ${\mathcal S}_\nu$ is an increasing limit of a sequence in $(\M_\nu)_+$, we see that  $i = j$. 

Next suppose we are given $h \, \eta  \, \M_\nu^+$
with $\nu_h = \nu$ on $\D$.  Then by hypothesis $i(h) (d) = \nu_h(d) = \nu(d)$ for $d \in \D_+$.  So in the notation 
of the last paragraph,  $i(h) = (\nu_{| \D})_m$ where $m = 1$.
On the other hand $j(h) = (\nu_{| \D})_n$ where  $n = \widehat{E'}(h)$.  Taking the Radon-Nikodym derivative with respect to 
$\nu_{| \D}$, we obtain $\widehat{E'}(h) = 1$.  Thus $\hat{E}_{\D}(h) = 1$.

At this point it is easy to see that (b) is equivalent to (b$^{\prime}$).   Indeed if $E_{\D}(h) = 1$ and $h$ commutes with elements of $\D_+$, then $$\nu_h(d) = \lim_{\epsilon\searrow 0}\nu(h_\epsilon^{\frac{1}{2}}  d 
h_\epsilon^{\frac{1}{2}} ) = \lim_{\epsilon\searrow 0}\nu(d^{1/2}h_\epsilon d^{\frac{1}{2}} ) = \nu(d^{1/2}hd^{1/2})$$ 
$$= \nu(E_{\D}(d^{1/2}hd^{1/2}) ) = \nu(d^{1/2}E_{\D}(h)d^{1/2})=\nu(d), \qquad d \in \D_+.$$

(b$^{\prime}$) $\Leftrightarrow$ (b$^{\prime \prime}$)\ Given $h  \, \eta \,  \M_\nu^+$ it is clear from Proposition \ref{commhk} that $h$ commutes strongly with $k_\nu$. For any $n\in \mathbb{N}$ we may then use the Borel functional calculus to see that 
$$\theta_s(h\chi_{[0,n]}(h).k_\nu\chi_{[0,n]}(k_\nu))=\theta_s(h\chi_{[0,n]}(h)).\theta_s(k_\nu\chi_{[0,n]}(k_\nu)),$$
which equals $$h\chi_{[0,n]}(h).e^{-s}k_\nu\chi_{[0,n]}(e^{-s}k_\nu)=e^{-s}h\chi_{[0,n]}(h).k_\nu\chi_{[0,e^sn]}(k_\nu).$$ Letting $n\to \infty$ now yields the conclusion that $\theta_s(hk_\nu)=e^{-s}$ for all $s$. That is $hk_\nu\in {}^\eta L^1_+(\M)$. 

Conversely given some $k\in {}^\eta L^1_+(\M)$ which commutes with $k_\nu$, a similar argument shows that for $h=k.k_\nu^{-1}$ we then have that $\theta_s(h)=h$ for every $s$. For each $n\in \mathbb{N}$ we will then have that $\theta_s(\chi_{[0,n]}(h))=\chi_{[0,n]}(\theta_s(h))=\chi_{[0,n]}(h)$ for all $s$. Thus all the spectral projections of $h$ are in $\M$, which ensures that $h\,\eta\,\M$. Since $h=k.k_\nu^{-1}$ clearly commutes with $k_\nu$, we in fact have that $h\,\eta\,\M_\nu$.

There is therefore a bijection between densities $k\in {}^\eta L^1_+(\M)$ which commute strongly with $k_\nu$ and for which $k.k_\nu^{-1}$ commutes with $\D$, and densities $h  \, \eta \,  \M_\nu^+$ which commute with $\D$. This bijection is given by $h\to hk_\nu$. 

If we are able to show that for densities $h  \, \eta \,  \M_\nu^+$ which strongly commute with $\D$ and $k_\nu$ we have that $E_{\D}(h)=\I$ 
  and only if $\overline{E_{\D}}(hk_\nu)=\overline{E_{\D}}(k_\nu)$, then the class of densities described in (b$^{\prime \prime}$) will  clearly be in 
  bijective correspondence with the class described in (b$^{\prime}$).  

Suppose first that $E_{\D}(k_\nu h) = k_\nu$. Note that the assumptions on $h$ ensure that in the crossed product it commutes strongly with $\D$ and with each $\lambda_t=k_\nu^{it}$. Thus $h$ commutes strongly with the von Neumann subalgebra generated by these operators, namely $\D\rtimes_\nu\mathbb{R}$.

Recall that we noted in Remark \ref{crossprod2} that $\tau\circ \overline{E_{\D}} =\tau$. Since by Proposition \ref{commhk} the strong product $k_\nu h$ is the density $k_{nu_h}$ of the weight $\widetilde{\nu}_h$, we may now use this fact to conclude that 
$$\widetilde{\nu}_h(d)= \tau((hk_\nu)^{\frac{1}{2}}  d (hk_\nu)^{\frac{1}{2}}) = \tau(d^{\frac{1}{2}}  \, hk_\nu d^{\frac{1}{2}}) = 
\tau(\overline{E_{\D}}(d^{\frac{1}{2}}  \, hk_\nu d^{\frac{1}{2}}))$$ $$= \tau(d^{\frac{1}{2}} \overline{E_{\D}}(hk_\nu) d^{\frac{1}{2}})=\tau(d^{\frac{1}{2}}  \, k_\nu d^{\frac{1}{2}}))=\widetilde{\nu}(d)$$ for all $d \in (\D\rtimes_\nu\mathbb{R})_+$. 
Again by Remark \ref{crossprod2}, we have that  $\widetilde{\nu} \circ \overline{E_{\D}} =\widetilde{\nu}$. 

This fact together with the commutation of $h$ with $\D \rtimes_\nu \mathbb{R}$ now ensures that 
$\widetilde{\nu}(d)=\widetilde{\nu}_h(d)$ equals 
$$\lim_{\epsilon\searrow 0} \widetilde{\nu}(d^{1/2}h_\epsilon d^{1/2})= \widetilde{\nu}(d^{1/2}h d^{1/2}) = \widetilde{\nu}(\overline{E_{\D}}(d^{1/2}hd^{1/2}))= \widetilde{\nu}(d^{1/2}\overline{E_{\D}}(h)d^{1/2})$$ for all $d \in (\D\rtimes_\nu\mathbb{R})_+$. 
Since $\nu$ is faithful and semifinite on $\D$ with
$\widetilde{\nu}_{|(\D\rtimes\mathbb{R})}=\nu_{|\D}\circ (T_{\nu})_{|(\D\rtimes\mathbb{R})}$, $\widetilde{\nu}_{|(\D\rtimes\mathbb{R})}$ the 
dual weight of $\nu_{|\D}$, and therefore faithful and semifinite on $\D\rtimes_\nu\mathbb{R}$. This then ensures that 
$\overline{E_{\D}}(h)=\I$. 

Now suppose that $\overline{E_{\D}}(h)=\I$. 
It then follows from the semifinite variant of \cite[Lemma 4.8]{Gol} mentioned in Remark \ref{crossprod2}, that $\overline{E_{\D}}(k_\nu h)=\overline{E_{\D}}(k_{\nu_h})= k_{\nu_h\circ E_{\D}}$. (Here $k_{\nu_h\circ E_{\D}}$ is considered in the context of $\D\rtimes_\nu\mathbb{R}$.) Now notice that when combined with the fact that $\nu\circ E_{\D}=\nu$, the commutation of $h$ with $\D$ ensures that 
$$\nu_h\circ E_{\D}(a)=\sup_\epsilon \nu(h^{1/2}_\epsilon E_{\D}(a)h^{1/2}_\epsilon)= \sup_\epsilon \nu(E_{\D}(a)^{1/2} h_\epsilon E_{\D}(a)^{1/2}).$$
But this equals $$\nu(E_{\D}(a)^{1/2} h E_{\D}(a)^{1/2})=\nu(\widehat{E_{\D}}(E_{\D}(a)^{1/2} h E_{\D}(a)^{1/2}))=\nu(E_{\D}(a)).$$ Thus we have that $\overline{E_{\D}}(k_\nu h)= k_{\nu_h\circ E_{\D}}=k_\nu$ as required.

(b$^{\prime}$) $\Rightarrow$ (a)\ Suppose that $E_{\D}(h) = 1$.  It follows that for any $x \in \M_+$ and any $\epsilon>0$ we have $$E_{\D} (h_\epsilon^{\frac{1}{2}}  x h_\epsilon^{\frac{1}{2}} ) \leq \| x \| \, E_{\D}(h_\epsilon) \leq \| x \| \, E_{\D}(h) =  \| x \| \cdot 1$$ 
where as before $h_\epsilon=h(\I+\epsilon h)^{-1}$.  
Thus for any $x \in \M_+$, the net $(E_{\D} (h_\epsilon^{\frac{1}{2}}  x h_\epsilon^{\frac{1}{2}}))$ is a norm bounded net in $\D$. 

For any $d \in \D_+$ we know from the equivalence of (b) with (b${}^\prime$) that $\nu(d)=\nu_h(d)$. So for any $d \in \mathfrak{n}(\D)_{\nu}$ and any $x\in\M$, we will have that $d^*xd\in\mathfrak{m}_{\nu_h}$. For any 
$d \in \mathfrak{n}(\D)_{\nu}$, it is now clear that 
\begin{eqnarray*}
\nu(d^*E_{\D} (h_\epsilon^{\frac{1}{2}}  x h_\epsilon^{\frac{1}{2}})d) 
&=& \nu(d^*(h_\epsilon^{\frac{1}{2}}  x h_\epsilon^{\frac{1}{2}})d)\\
&=& \nu(h_\epsilon^{\frac{1}{2}} d^* x d h_\epsilon^{\frac{1}{2}})\\
&\nearrow& \nu_h(d^* x d)<\infty
\end{eqnarray*} 
as $\epsilon\searrow 0$. To see this recall that by \cite[Lemma VIII.2.7]{tak2} the net $(\nu(h_\epsilon^{\frac{1}{2}} d^* x d h_\epsilon^{\frac{1}{2}}))$ increases to $\nu_h(d^* x d)$. The above fact in particular also ensures that if $\epsilon_2 > \epsilon_1 >0$, we will then for any $d \in \mathfrak{n}(\D)_{\nu}$ have that 
$$\nu(d^*[E_{\D} (h_{\epsilon_2}^{\frac{1}{2}}  x h_{\epsilon_2}^{\frac{1}{2}})-E_{\D} (h_{\epsilon_1}^{\frac{1}{2}}  x h_{\epsilon_1}^{\frac{1}{2}})]d) \geq 0.$$Now recall that in the 
GNS representation for $(\nu_{|\D},\D)$ the space $\eta(\mathfrak{n}_\nu(\D))$ is a dense subspace of $H_{\nu_{|\D}}$. On 
passing to this GNS representation if need be, the above inequality corresponds to the statement that $$\langle [E_{\D} (h_{\epsilon_2}^{\frac{1}{2}}  x h_{\epsilon_2}^{\frac{1}{2}})-E_{\D} (h_{\epsilon_1}^{\frac{1}{2}}  x h_{\epsilon_1}^{\frac{1}{2}})]\xi, \xi\rangle \geq 0$$for all vectors $\xi$ in a dense subspace of $H_{\nu_{|\D}}$. By continuity we will then in fact have that $\langle [E_{\D} (h_{\epsilon_2}^{\frac{1}{2}}  x h_{\epsilon_2}^{\frac{1}{2}})-E_{\D} (h_{\epsilon_1}^{\frac{1}{2}}  x h_{\epsilon_1}^{\frac{1}{2}})]\xi, \xi\rangle \geq 0$ for all $\xi \in H_{\nu_{|\D}}$. But this ensures that $E_{\D} (h_{\epsilon_2}^{\frac{1}{2}}  x h_{\epsilon_2}^{\frac{1}{2}})\geq E_{\D} (h_{\epsilon_1}^{\frac{1}{2}}  x h_{\epsilon_1}^{\frac{1}{2}})$. Thus the net $(E_{\D} (h_\epsilon^{\frac{1}{2}}  x h_\epsilon^{\frac{1}{2}}))$ is a norm bounded increasing net. Hence for each 
$x\in \M_+$, the limit $\lim_{\epsilon\searrow 0} E_{\D} (h_\epsilon^{\frac{1}{2}}  x h_\epsilon^{\frac{1}{2}})$ exists with convergence taking place in the $\sigma$-strong* topology.

For any $x\in \M_+$, we now define the element $E(x)\in \D_+$ by $E(x)= \lim_{\epsilon\searrow 0} 
E_{\D}(h_\epsilon^{\frac{1}{2}} x h_\epsilon^{\frac{1}{2}})$. We pause to note that for this element, it is clear from the 
above computations that for any $d \in \mathfrak{n}(\D)_\nu$ and any $x\in \M_+$, we have that 
\begin{equation}\label{nuhexp} \nu(d^*E(x)d)=\lim_{\epsilon\searrow 0}\nu(d^* E_{\D} (h_\epsilon^{\frac{1}{2}}  x h_
\epsilon^{\frac{1}{2}})d) = \lim_{\epsilon\searrow 0}\nu(h_\epsilon^{\frac{1}{2}} d^*xd h_\epsilon^{\frac{1}{2}})=
\nu_h(d^*xd). \end{equation}

Notice that for $x \in \M_+$ and any $d\in \D$, the commutation of $h$ with $\D$ 
 clearly ensures that $$E(d^* x d) = 
\lim_{\epsilon\searrow 0} E_{\D} (h_\epsilon^{\frac{1}{2}}  d^* x d h_\epsilon^{\frac{1}{2}}) = 
d^* \lim_{\epsilon\searrow 0}E_{\D} (h_\epsilon^{\frac{1}{2}} x h_\epsilon^{\frac{1}{2}}) d = d^* E(x) d , \qquad d \in \D.$$
The prescription $E(x) = \lim_{\epsilon\searrow 0}E_{\D}(h_\epsilon^{\frac{1}{2}} x h_\epsilon^{\frac{1}{2}})$ 
($x \in \M_+$), therefore yields a well defined $\D$-valued operator valued weight on $\M_+$. 
Since also $E(\I) = \lim_{\epsilon\searrow 0}E_{\D}(h_\epsilon) = E_{\D} (h)  = \I$, $E$ gives rise to
a conditional expectation of $\M$ onto $\D$ that we continue to write as $E$ (see \cite[Lemma IX.4.13 (iii)]{tak2}). 

Now suppose that $x_t \nearrow x$ in $\M_+$. Since $E$ is order preserving, we clearly have that $\sup_t E(x_t)=\lim_t E(x_t) 
\leq E(x)$. But for any $d \in \mathfrak{n}(\D)_\nu$ it follows from equation (\ref{nuhexp}) and the fact that $\nu_h(d^*\cdot d)$ is 
a positive normal functional, that $$\nu(d^*(E(x)-\lim_t E(x_t))d)=\lim_t\nu_h(d^*(x-x_t)d)=0.$$ On again passing to the GNS representation for $(\nu_{|\D},\D)$ if need be, this ensures that $\langle\sup_t E(x_t)\xi,\xi\rangle = \langle E(x)\xi,\xi\rangle$ for all $\xi$ in a dense subspace of $H_{\nu_{|\D}}$. By continuity we then in fact have that $\langle\sup_t E(x_t)\xi,\xi\rangle = \langle E(x)\xi,\xi\rangle$ for all $\xi\in H_{\nu_{|\D}}$. That can of course only be if $\sup_t E(x_t)= E(x)$, which proves the normality of $E$. 

It remains to show that $\nu(E(x)) = \nu_h(x)$ for all $x\in \M_+$. To see this recall that for $x\in\M_+$ the net $(E_{\D} (h_\epsilon^{\frac{1}{2}}  x h_\epsilon^{\frac{1}{2}}))$ increases to $E(x)$. The normality of $\nu$ therefore ensures that $$\nu(E(x)) = \sup_{\epsilon > 0}\nu(E_{\D} (h_\epsilon^{\frac{1}{2}}  x h_\epsilon^{\frac{1}{2}}))=\sup_{\epsilon > 0} \nu(h_\epsilon^{\frac{1}{2}}  x h_\epsilon^{\frac{1}{2}})=\nu_h(x).$$Thus by the statement of the main correspondence in \cite{PT}, $\nu \circ E$ commutes with $\nu$.

We now prove the bijectivity of the correspondences.
The map defined above `from  (b) to (c)'  is one-to-one by the bijectivity in the Pedersen-Takesaki Radon-Nikodym theorem \cite{PT}.

Let $\rho =   \nu \circ E$, and $z = {\rm supp} \, (E)  = {\rm supp} \, (\rho)$. 
Define $E_z$ and $E_1$ as in the proof of (b) above.   Then as we saw in the proof of (a) above, 
$E(zxz) z$ defines a normal conditional expectation 
from $z \M z$ onto $\D z$ with $\rho(E(zxz) z) = \rho(zxz)$, for $x \in \M$.  
By the uniqueness assertion where we defined $E_z$ above we have
$$E_z(zxz) = E_z(zxz) z = E_1(zxz) z = E(zxz) z = E(x) z, \qquad x \in \M .$$   
Since   $\D \cong \D z$ via the $*$-isomorphism of right multiplication by $z$ we see from the last formula
that $E$ may be retrieved from $E_z$ and hence from $\rho$.   (See also \cite[Lemma 4.8]{haag-OV1}.) 

Thus to complete the proof that the correspondences defined between (a), (b), (c) are bijective it suffices to show that 
given $h$ as in (b), if $E_h$ is the expectation obtained in the proof of 
(b$^{\prime}$)$\Rightarrow$(a), we then have $\frac{d(\nu \circ E_h)}{d \nu} = h$.  
But we saw in  the proof of (b$^{\prime}$) $\Rightarrow$(a) that 
$\nu_h = \nu \circ E_h$.  Hence $\frac{d(\nu \circ E_h)}{d \nu} = h$ as required.
    \end{proof} 
 
\begin{remark} 
 \label{afterg} 
1)\ Let $\D$ be a von Neumann subalgebra of $\M$, and let $\nu$ 
 be a faithful normal weight on $\M$ which is semifinite on $\D$ and  
 satisfies $\sigma_t^\nu(\D) = \D$ for real $t$.   Let $E_{\D}$ be the $\nu$-preserving 
conditional expectation of $\M$ onto $\D$ (see the definition at the start of Section \ref{ncec}). 
  Let $E$ be a faithful normal conditional expectation onto $\D$ which commutes with $(\sigma^\nu_t)$.   Then $E$ is the {\em 
weight-preserving  
conditional expectation} associated with the faithful normal semifinite weight $\rho = \nu \circ E$.  
 This weight satisfies $\sigma_t^\rho(\D) = \D$ for real $t$ since by \cite[Theorem 4.7 (1)]{haag-OV1} we have 
$\sigma_t^{\nu \circ E} =  \sigma_t^{\nu}$.  
So we are in the setting of the theorem but with $\omega = \nu \circ E$ playing the role of $\nu$, and then 
the associated $\omega$-preserving 
conditional expectation $E_{\D}$ (guaranteed by \cite[Theorem IX.4.2]{tak2}  as at the start of Section \ref{ncec}), is exactly $E$.  

\smallskip

2)\  An explicit example showing the necessity of $\nu$ being semifinite on $\D$ is the case that 
$D$ is the copy of $L^\infty([0,1])$  in $B(L^2([0,1]))$ in Example \ref{abave}, with 
 $\nu$ the trace on $B(L^2([0,1]))_+$.     Here $\nu$ is not semifinite on $\D \cong L^\infty([0,1])$, although 
 it is semifinite on $B(L^2([0,1]))_+$ and $\sigma^\nu_t(\D) = \D$.  

\smallskip

3)\ Let $\nu$ be a faithful normal semifinite 
{\em trace} on $\M$ and let $h$ be a positive operator affiliated to $\M_\nu$. 
Following on from the remark after Proposition \ref{commhk} we show that we may in the specialized 
setting of Theorem \ref{gencor}, similarly rewrite the limit 
$E(x) = \lim_{\epsilon \searrow 0} \, E_{\D}(h_\epsilon^{\frac{1}{2}} x h_\epsilon^{\frac{1}{2}})$ as the claim that 
$E_{\D}(h^{\frac{1}{2}} x h^{\frac{1}{2}})=E(x)$, where by $h^{\frac{1}{2}} x h^{\frac{1}{2}}$ we mean the operator 
$b=(h^{\frac{1}{2}}x^{\frac{1}{2}})(h^{\frac{1}{2}}x^{\frac{1}{2}})^*\in L^1(\M)$.  

Recall that the theorem ensures that $\nu(d)=\nu_h(d)$ for any $d\in \D_+$. Thus 
${\mathfrak n}(\D)_\nu = {\mathfrak n}(\D)_{\nu_h}$. So for any $d \in {\mathfrak n}(\D)_\nu$ and with $x$ as before, 
we will have that $d^*xd \in \mathfrak{p}_{\nu_h}$. The results of the remark following Proposition \ref{commhk}, therefore 
ensures that $h^{\frac{1}{2}} d^* x^{\frac{1}{2}}\in L^2(\M,\nu)$. When combined with the fact that $E_{\D}$ is 
$\nu$-preserving, the strong commutation of $h$ with $\D$ further ensures that
$$\nu(d^* E_{\D}(b) d) = \nu(d^* bd) = \nu( (h^{\frac{1}{2}} d^*x^{\frac{1}{2}})(h^{\frac{1}{2}}d^*x^{\frac{1}{2}})^*)= 
\nu((h^{\frac{1}{2}} d^*x^{\frac{1}{2}})^* (h^{\frac{1}{2}} d^*x^{\frac{1}{2}})),$$
while by similar considerations 
$$\nu(d^* E(x) d) = \lim_{\epsilon \searrow 0} \, \nu(E_{\D}(d^* h_\epsilon^{\frac{1}{2}} x h_\epsilon^{\frac{1}{2}} d) )
= \lim_{\epsilon \searrow 0} \, \nu(h_\epsilon^{\frac{1}{2}} d^* x d h_\epsilon^{\frac{1}{2}} ).$$
In equation (\ref{PTremark}) we may by a tiny variant of that argument replace $x^{\frac{1}{2}}$ with $d^* x^{\frac{1}{2}}$, 
upon which it is then clear that the 
quantity in the last equation equals 
$$\nu((x^{\frac{1}{2}} d) h (x^{\frac{1}{2}} d)^*) = \nu(|h^{\frac{1}{2}} d^*x^{\frac{1}{2}} |^2) .$$
Hence $\nu(d^* E(x) d) = \nu(d^* E_{\D}(b) d)$. We know from Remark \ref{crossprod2} that $E_{\D}(b)\in L^1(\M,\nu)_+$. 
Since $\nu(E(x))=\nu_h(x)<\infty$, we also have that $E(x)\in L^1(\M,\nu)$. It follows that $E_{\D}(b)=E(x)=  
\lim_{\epsilon \searrow 0} \, E_{\D}(h_\epsilon^{\frac{1}{2}} x h_\epsilon^{\frac{1}{2}})$ as desired.
\end{remark}

We now turn to the special case that semifinite $\D$ is contained in the centralizer $\M_\nu$.   This is equivalent by \cite[Theorem VIII.2.6]{tak2}  to the modular automorphism group
$(\sigma^\nu_t)$ restricting to the identity
map on $\D$ for each real $t$. 
In this case there again exists the 
$\nu$-preserving  conditional expectation $E_{\D} : \M \to \D$ by \cite[Theorem IX.4.2]{tak2}.

\begin{corollary} \label{Hcongen} Let $\D$ be a von Neumann subalgebra of a  von Neumann algebra $\M$, and let $\nu$ be a faithful normal weight on $\M$, which is semifinite on $\D$ and is also $\D$-central 
(that is, $\D \subset \M_\omega$; or equivalently, by e.g.\ Lemma  {\rm \ref{tuse}}, $\nu$ is tracial on $\D$ with in addition 
$\nu(xd)=\nu(dx)$ for any $x\in \mathfrak{m}(\M)_\nu$ and any $d\in \mathfrak{m}(\D)_\nu$). 
 There is a bijective correspondence between the following objects:
\begin{enumerate}
\item[(a)] Normal conditional expectations $E$ onto $\D$ 
which commute with $(\sigma^\nu_t)$ ($t\in\mathbb{R}$), 
\item[(b)] densities $h \eta \M_\nu^+$ which commute with all elements of $\D$, and which satisfy $\nu_h = \nu$ on $\D_+$,
\item[(b$^{\prime}$)] densities $h \eta \M_\nu^+$ which commute with all elements of $\D$, and which satisfy $E_{\D}(h) = 1$ 
(where the action on $h$ is by the normal extension of $E_{\D}$ to the extended positive part of $\M$ in e.g.\
 {\rm \cite[Proposition 3.1]{Gol}}), 
\item[(b$^{\prime \prime}$)] densities $k \in {}^\eta L^1_+(\M)$ which commute strongly with both $k_\nu$ and $\D$, and which satisfy $E_{\D}(k) = E_{\D}(k_\nu)$
(where the action on $k_\nu$ and $k$ is by the normal extension of $E_{\D}$ to the extended positive part of $\M\rtimes_\nu\mathbb{R}$),
\item[(c)] extensions of $\nu_{|\D}$ to normal semifinite weights $\rho$ on all of $\M$ which commute with 
$\nu$, and are $\D$-central in the sense that $\rho(xd)=\rho(dx)$ for any $x\in\mathfrak{m}(\M)_\rho$ and any $d\in
 \mathfrak{m}(\D)_\nu$.
\end{enumerate} 
With respect to the correspondences above we also have the formulae  $$\rho = \nu \circ E = \nu_h, \; \; \; \; \; \; \;  \; \; \; \; h = \frac{d \rho}{d\nu} = k.k_\nu^{-1},$$ and 
$$E(x) =\lim_{\epsilon\searrow 0} \; E_{\D}(h_\epsilon^{\frac{1}{2}}  x h_\epsilon^{\frac{1}{2}} ), \qquad x \in \M_+,$$ where $h_\epsilon=h(\I+\epsilon h)^{-1}$ 
and  with the last limit being a ($\sigma$-strong*) limit of an increasing net in $\M$. 
 Indeed the relation between $E$ and $h$ is uniquely determined by the equation $\nu(E(exe)) = \nu_h(exe)$ 
 for  $x \in M_+$ and  projections $e\in \D$ with $\nu(e) < \infty$.  
\end{corollary}

\begin{proof}  We saw earlier that $\D \subset \M_\omega$ 
is equivalent to the requirement that $k_\nu$ commutes strongly with $\D$. 
Under this a priori assumption the equivalence of (a), (b), (b$^{\prime}$) and (b$^{\prime \prime}$) now follow directly from Theorem \ref{gencor}. The only equivalence that bears some investigation is (c). 

(a) $\Rightarrow$ (c)\  As in the proof in Theorem \ref{gencor} 
if $E: \M \to \D$ is a normal conditional expectation onto $\D$ which commutes with $(\sigma^\nu_t)$
then  $\rho  = \nu \circ E$ is a normal semifinite  weight on $\M$ extending $\nu_{| \D}$.
Again  $\rho\circ\sigma^\nu_t=\rho$ and   $\rho$ commutes with $\nu$  in the sense of {\rm \cite{PT}} (see also  \cite[Corollary VIII.3.6]{tak2}, although 
here all weights are faithful).

Since $\rho$ and $\nu$ agree on $\D$, it is clear that $\mathfrak{m}(\D)_\nu^+=\mathfrak{m}(\D)_\rho^+$, and hence by linearity 
that $\mathfrak{m}(\D)_\nu=\mathfrak{m}(\D)_\rho$.
For $d\in\mathfrak{m}(\D)_\rho$, we then clearly have that 
$d\mathfrak{m}_\rho \subset \mathfrak{m}_\rho$ and $\mathfrak{m}_\rho d \subset \mathfrak{m}_\rho$. Now let $a\in\mathfrak{m}_\rho^+$ 
be given. Then $E(a)\in \mathfrak{m}_\nu^+$ since by the Kadison-Schwarz inequality 
 $$\nu(E(a)^* E(a)) \leq \nu(E(a^* a) = \rho(a^* a) < \infty .$$ It is therefore clear that 
$$\rho(ad)=\nu(E(ad))=\nu(E(a)d)=\nu(dE(a))=\nu(E(da))=\rho(da)$$ 
and hence that 
$\mathfrak{m}(\D)_\rho$ is in the centralizer of $\rho$.

(c) $\Rightarrow$ (a)\ Take $\rho$ commuting with $\nu$ as in (c).      Let $E$ be the $\rho$-preserving conditional expectation  in Theorem  \ref{excodsfex} 
(1).   We have  $\rho(E(\sigma_t^\nu(x)))  = \rho(\sigma_t^\nu(x))  = \rho(x)$ if $x \in \M_+$. So $E \circ \sigma_t^\nu$ is a  $\rho$-preserving conditional expectation, hence  $E \circ \sigma_t^\nu = E$ by 
 the uniqueness in Theorem \ref{excodsfex}.  The fact that $E \circ \sigma_t^\nu$ is the identity on $\D$ follows because it is normal and is the identity on ${\mathfrak m}(\D)$.

For the final claim, it obviously follows from equation (\ref{nuhexp}) that $\nu(E(exe)) =\nu_h(exe)$, for $x \in \M_+$ and  projections $e\in \D$ with $\nu(e) < \infty$.

Conversely, if this equation holds then $\nu_h$ and $\nu \circ E$ are two normal semifinite weights on $\M_+$ that agree on 
$e \M e$ for each projection $e\in \D$ with $\nu(e) < \infty$.   Note that such $e$ are in the 
centralizers of these weights (for $\rho = \nu_h$ it follows from e.g.\   \cite[Proposition 4.6]{PT}
  
 that $\sigma^\rho_t(e) = h^{it} e h^{-it} = e$).  Since such $e$ form an increasing net with limit $1$, we may deduce from 
\cite[Proposition 4.2]{PT} (or \cite[Lemma VIII.2.7]{tak2}) that $\nu_h = \nu \circ E$.  
However we have seen above that $\nu_h = \nu \circ E$ completely determines $h$ in terms of $E$ and vice versa. 
\end{proof}

\begin{remark} 
1)\ The 
following parallels  Remark 1)\ after Theorem \ref{gencor}.  Let $\D$ be a von Neumann subalgebra of $\M$, and $\nu$ a 
faithful normal semifinite weight on $\M$, which is semifinite on $\D$ and  $\D$-central in the sense of the corollary.  
Let $E$ be a faithful normal conditional expectation onto $\D$ which commutes with $(\sigma^\nu_t)$.   
As in   Remark 1)\ after Theorem \ref{gencor}, $E$ is the {\em 
weight-preserving 
conditional expectation} associated with the faithful normal semifinite weight $\rho = \nu \circ E$.  
We check that this weight satisfies the appropriate centralizer condition.  By \cite[Theorem 4.7 (1)]{haag-OV1} we have 
$\sigma_t^{\nu \circ E} =  \sigma_t^{\nu}$ on $\D$ for all $t$, so that $\sigma_t^{\nu \circ E}(d) = d$ for all $d \in \mathfrak{m}(\D)_\nu$, 
and hence  $$\nu(E(xd))= \nu(E(x)d)= \nu(d E(x)) = \nu(E(dx)), \qquad x \in \mathfrak{m}(\M)_\nu, \;    
d\in \mathfrak{m}(\D)_\nu.$$
So we are in the setting of the corollary but with $\rho = \nu \circ E$ playing the role of $\nu$. Then 
the associated $\rho$-preserving 
 conditional expectation $E_{\D}$ (guaranteed by 
\cite[Theorem IX.4.2]{tak2}  as at the start of Section \ref{ncec}), is exactly $E$. 

\smallskip

2)\  For a nonzero projection $p \in \D$ with $\nu(p) < \infty$, a multiple of $\nu$ is a faithful normal state $\nu_p$ on $p \M p$.   
As in Theorem  \ref{wcentw}  and its proof we may  view the conditional expectation $E$ as being approximated by an increasing 
net of conditional expectations $E_p(x) = E(pxp) = p E(x) p$ on the  von Neumann algebras $p \M p$. 
Pursuing this line suggests that 
there is no doubt a variant of (b) (or (b$^{\prime}$)) in 
the last result phrased in terms of the $\sigma$-finite von Neumann algebras $p \M p$, for each such $p$, and `local densities' $ehe$.

\smallskip

3)\ Let $\D$ be a von Neumann subalgebra of $\M$, and $\nu$ a faithful normal  tracial weight on $\M$, 
which is semifinite on $\D$.  It follows that we then have a bijective correspondence between 
(a)\ normal conditional expectations $E$ onto $\D$, 
(c)\  extensions of $\nu_{|\D}$ to normal semifinite weights $\rho$ on all of $\M$ which commute with 
$\nu$ in the sense in the Introduction, and  items (b) and (b$^{\prime}$) in the corollary as stated.  
We may also use the remark after Proposition \ref{commhk} and Remark 
3 after Theorem \ref{gencor} to write 
$\nu_h$ and $E$ in terms of $h$ without using limits, as described there.
\end{remark}

The following is a special case of Corollary  \ref{Hcongen} and 3)  of the last remark.   The explicit formulation as a result characterizing $\D$-valued normal conditional expectations will however prove to be extremely useful.  

\begin{corollary} \label{Hcon} Let $\D$ be a von Neumann subalgebra of $\M$, and $\tau$ a faithful normal tracial state on $\M$. There is a bijective correspondence between the following objects:
\begin{enumerate}
\item[(a)] Normal conditional expectations onto $\D$,
\item[(b)] densities $h \in L^1(\M)_+$ which commute with $\D$ and satisfy 
$E_{\D}(h)=\I$,
\item[(c)] extensions of $\tau_{|\D}$ to normal states $\rho$ on all of $\M$ which have $\D$ in their centralizer.
\end{enumerate}
The  normal conditional expectation $E$ in {\rm   (a)} and  $h$ in  {\rm   (b)}  are related by the formula
$$E(x) = E_{\D}(h^{\frac{1}{2}}  x h^{\frac{1}{2}}) , \qquad x \in \M_+.$$ It also may be described by the relation 
$\tau(E(x) ) = \tau(hx)$ for all $x \in \M e$.  Also $h$ in {\rm  (b)} is the 
Radon-Nikodym derivative $\frac{d \rho}{d \tau}$ for $\rho$ as in {\rm (c)}. 
\end{corollary}

\begin{proof} Since $\tau$ is a tracial state the commuting condition in Corollary \ref{Hcongen} (a) and (c) is automatic, and $\mathfrak{m}(\M)_\nu = \M$ and $\mathfrak{n}(\D)_\nu
= \D$.    With these in mind, the rest is an exercise.  \end{proof} 

\section{The noncommutative Hoffman-Rossi theorem--preliminary considerations} \label{HR1}

\begin{proposition} \label{HRf3g}  
Consider the inclusions $\D \subset \A \subset \M,$ where $\M$ is a von Neumann algebra with 
normal  state $\nu$ which is faithful on von Neumann subalgebra $\D$, with  $\D$ in the centralizer of $\nu$, and $\A$ is a
 $\D$-submodule of $\M$. 
Let $\Phi : \A \to \mathcal{D}$ be a  
map with $\Phi(1) = 1$, which is a $\mathcal{D}$-bimodule map (if $A$ is an algebra
and $\Phi$ is a homomorphism
 then this
is  equivalent to $\Phi$ being the identity map on $\mathcal{D}$). 
\begin{itemize}
\item [(1)]  Assume that  $\nu \circ \Phi$ extends to a normal state 
$\rho$ on $\M$ with $\D$ in its centralizer (that is, $\rho(xd)=\rho(dx)$ for every $d \in \mathcal{D}, x\in \M$).
Then there is a unique  $\rho$-preserving normal conditional expectation $\M \to \mathcal{D}$, and this expectation  extends $\Phi$. 
\item [(2)] Assume that  $\nu$ is faithful and 
$\nu \circ \Phi$ extends to a normal state  on $\M$.    
 Then $\nu \circ \Phi$ extends to a normal state 
$\rho$ on $\M$ with $\D$ in its centralizer,
and there is a  normal $\rho$-preserving conditional expectation $\M \to \mathcal{D}$ extending $\Phi$.  Moreover such a
normal $\rho$-preserving conditional expectation is unique.  
\end{itemize}
\end{proposition}

\begin{proof}    (1)\    By Corollary \ref{unfco}  there is a  $\rho$-preserving normal conditional 
expectation $E : \M \to \mathcal{D}$.  
As usual  the fact that it is $\rho$-preserving makes it unique, and 
by the same argument $E$ extends $\Phi$.   

(2)\ Suppose that  
$\nu \circ \Phi$ extends to a normal state  $\omega$ on $\M$.     
Let  $K_\omega$ and $\rho \in K_\omega$ be as  in Lemma \ref{aves2}, so that
$\rho(xd)=\rho(dx)$ for every $d \in \mathcal{D}, x\in \M$.
  We have a sequence $\psi_n = \sum_k \, t_k u_k^* \omega u_k$
of convex combinations approximating $\rho$.  For  $a \in \A$ we have  $$\psi_n(a) = \sum_k \, t_k \, \omega( u_k a u_k^* ) 
= \sum_k \, t_k \, \nu(\Phi( u_k a u_k^* )) =  \sum_k \, t_k \,  \nu(\Phi(a) ) = \nu(\Phi(a) ).$$
Thus $\rho(a) = \nu(\Phi(a) )$, so that $\rho$ is a normal state extending $\nu \circ \Phi$.  Now apply (1) to obtain 
a  normal $\rho$-preserving conditional expectation $\M \to \mathcal{D}$ extending $\Phi$.  \end{proof}

\begin{remark} 
Indeed we may take $\rho$  in (2) to be $\omega \circ E_{\D' \cap \M}$ 
where $E_{\D' \cap \M}$ is the unique $\omega$-preserving 
conditional expectation onto 
$\D' \cap \M$.
\end{remark}

\begin{corollary}  \label{HRf3}  
Consider the inclusions $\D \subset \A \subset \M,$ where $\M$ is a von Neumann algebra
with faithful normal tracial state $\tau$, $\A$ is a
 subalgebra of $\M$, and $\D$ is a von Neumann subalgebra containing the unit of $\M$.
Let $\Phi : \A \to \mathcal{D}$ be a  unital 
map, which is a $\mathcal{D}$-bimodule map. 
\begin{itemize}
\item [(1)]  Assume that  $\tau \circ \Phi$ extends to a normal state 
$\rho$ on $\M$ whose density with respect to $\tau$ commutes with $\D$.   Then there is a unique  $\rho$-preserving normal conditional expectation $\M \to \mathcal{D}$ extending $\Phi$. 
\item [(2)] Assume that  $\tau \circ \Phi$ extends to a normal state  on $\M$.    
 Then there is a  normal conditional expectation $\M \to \mathcal{D}$ extending $\Phi$. 
\end{itemize}
\end{corollary}

To illustrate the efficacy of the technology developed above, we show how it may be 
used to recover a commutative 
generalization of the classical Hoffman-Rossi theorem conditioned to the present context. The proof strategy of this corollary will in the next section serve as the template for the proof of the general noncommutative Hoffman-Rossi theorem.

\begin{corollary} \label{HRc}  Consider the inclusions $\D \subset \A \subset \M,$ where $\M$ is a commutative von Neumann algebra, $\A$ is a
weak* closed subalgebra of $\M$, and $\D$ is a von Neumann subalgebra containing the unit of $\M$. 
Let $\sigma$ be a faithful  normal semifinite weight on $\D$.
Let $\Phi : \A \to \mathcal{D}$ be a weak* continuous  $\D$-character.   
   Then   $\sigma \circ \Phi$  extends to a  normal semifinite weight $\rho$ on $\M$, such that
there is a unique  $\rho$-preserving normal conditional expectation $\M \to \mathcal{D}$ extending $\Phi$.
Indeed for any  normal semifinite trace $\rho$ extending $\sigma \circ \Phi$,
there is a unique  $\rho$-preserving normal conditional expectation $\M \to \mathcal{D}$ extending $\Phi$.
 \end{corollary}

\begin{proof}   First assume that $\D$  possesses a
faithful normal state $\sigma$.  Let $\tau$ be a
faithful normal semifinite trace on $\M$ (integration with respect to $\mu$ if $\M = L^\infty(\mu)$). 
  By Banach space duality $\sigma \circ \Phi$  extends to a weak* continuous 
functional on $\M$, so there exists $r = ab^* \in L^1(\M),$ with $a, b \in L^2(\M)$, such that $\tau(r x) = \sigma(\Phi(x))$ for all $x \in \A$. 
   In particular $\tau(r d) = \tau(da b^*) =  \sigma(d)$ for 
$d \in \D$, 
and $\tau(r j) = \tau(ja b^*) = 0$ for $j \in J = {\rm Ker}(\Phi)$, so that $b \perp F$
 where $F = [Ja]_2$.    We apply an idea that appears to go back to Sarason and others, as discussed
in  \cite{BFZ}.  Let $c$ be the projection $P_E(b)$ of $b$ onto $E = [Aa]_2$.
Then $c \perp F$ since $F \subset E$.   Note that
$$\tau(j c c^*) = \tau(j c b^*)  = 0, \qquad j \in J,$$ since $jc \in J [Aa]_2 \subset F \subset E$.   For $d \in \D$ 
 we have 
 $$ \sigma(d^*d) =  |\langle d^*d a , b \rangle |  = |\langle d^* da , c \rangle |   = |\langle a , d^*dc \rangle |\leq 
\| a \|_2 \| d^* d c \|_2 . $$
 Thus 
$$\sigma(d^*d)^2 \leq  \| a \|_2^2  \, \tau( (d^* d)^2 c c^*) =  \| a \|_2^2  \, \tau( (d^* d)^2 E_{\D}(c c^*)), \qquad  d \in \D .$$

  Thus $\rho(x) = \frac{1}{\| c \|_2} \, \tau(x c c^*)$ is a
normal state on $\M$ restricting to a faithful normal tracial state $\omega$ on $\D$, and $\rho$ annihilates $J$.  

By e.g.\ Theorem \ref{wcent}  (or by Corollary
3.2 in \cite{BLv}), 
there is a unique  $\rho$-preserving normal conditional expectation $\Psi : \M \to \mathcal{D}$.  Fix $d \in \D$. 
Now $\omega(\Psi(a) d) = \rho(ad)$.   This equals $\omega(\Phi(a) d)$ for all $a$ in $J$ and in $\D$, since both are zero on $J$
and equal $\omega(ad)$ for $a \in \D$.   Since $\A = J \oplus \D$ we deduce that $\Psi_{|\A} = \Phi$.

Now that we know that there exists a normal conditional expectation $\Psi : \M \to \mathcal{D}$ extending $\Phi$,
redefining $\rho = \sigma \circ \Psi$ gives a normal state  on $\M$ extending $\sigma \circ \Phi$.
Also $\Psi$ is   $\rho$-preserving.     The last assertion follows from 
Theorem \ref{wcent}: for any  normal state  $\rho$   on $\M$ extending $\sigma \circ \Phi$,
there exists a  $\rho$-preserving  normal conditional expectation $\Psi : \M \to \mathcal{D}$.   By the argument at the 
end of the last paragraph $\Psi_{|\A} = \Phi$. 

In the general case suppose that $(p_i)$ are projections in $\D$ adding to $1$  with $\sigma(p_i) < \infty$, as in the proof of
Theorem \ref{wcentw}.   Let 
$\D_i = p_i \D$, which has a faithful normal state $\sigma_i = \frac{1}{\sigma(p_i)} \, \sigma(p_i \, \cdot)$.
We have $\Phi(p_i) = p_i$ and $\Phi(p_i \A) = \D_i$. 
So by the first part there exists a unique $\sigma_i$-preserving normal conditional expectation $E_i : \M_i = \M p_i  \to \mathcal{D}_i$ 
extending $\Phi(p_i \, \cdot)$.   Define $\Psi : \M \to \D$ by $\Psi((p_i x)) = (E_i(p_ix))$.   
As in the proof of 
Theorem \ref{wcentw} this is a  normal conditional expectation $\M \to \mathcal{D}$.
 Let $\rho = \sigma \circ \Psi$ on $\M_+$, a normal weight extending $\sigma$.  Note that $\rho$ is a semifinite weight since $x e_t \to x$ for all $x \in \M$, 
 and $e_t$ and hence $x e_t$ are in ${\mathfrak n}_\rho$.
  We have $$\sigma(\Psi((p_i x))) = \sigma(\sum_i \, E_i(p_ix)) =\sum_i \, \sigma(E_i(p_ix)) =  \sum_i \, 
\rho(p_ix) = \rho(x) , \qquad x \in \M_+.$$

If $\rho$ is any normal weight on $\M$ extending $\sigma \circ \Psi$, then e.g.\ 
Theorem \ref{excodsfex} gives a unique normal conditional expectation $\Psi$ from $\M$ onto $\mathcal{D}$ which is 
 $\rho$-preserving. Fix $d \in \D$, and $p$ a projection in $\D$ of finite trace. 
Now $\omega(\Psi(a) dp) = \rho(adp)$.   This equals $\omega(\Phi(a) dp)$ for all $a$ in $J$ and in $\D$, since both are zero on $J$
and equal $\omega(adp)$ for $a \in \D$.   Since $\A = J \oplus \D$ we deduce that $\Psi(a)p= \Phi(a)p$.  Taking a weak*  limit along net of such projections
$p$ increasing to $1$, we obtain $\Psi(a) = \Phi(a)$.

The uniqueness is  as in the proof of
Theorem \ref{wcentw}: suppose that $x \in \M_+$ and $\sigma(d^* \Psi(x) d) = \sigma(d^* E(x) d)$ for all $d \in \D$.
Choosing $d$ a positive element in $\D_i$ shows that $\sigma((\Psi(x) - E(x))p_i d^2) = 0$.  
Since such squares span $\D_i$, we see that $\Psi(x) p_i = E(x) p_i$ for all $i$ so that $\Psi(x)  = E(x)$.
\end{proof}

The last proof shows that for any normal state 
$\omega$ on $\D$, $\omega \circ \Phi$  extends to a  normal state $\rho$ on $\M$, and 
there is a  $\rho$-preserving normal conditional expectation $\M \to \mathcal{D}$ extending $\Phi$.   The latter is unique
if $\omega$ is faithful on $\D$.

One problem with the proof of the last result in the  noncommutative case is that not every  faithful normal 
tracial state on $\D$ extends to a faithful normal tracial state on $\M$ (consider the  diagonal algebra  
in $M_2$). Indeed it need not even extend to a faithful normal state on $\M$ 
with $\D$ in the centralizer (see the example above Lemma \ref{aves2}).  We will overcome this difficulty in 
the form of Theorem \ref{HRc2} and Theorem \ref{HRsfcent}. In the case that $\M$ is semifinite and noncommutative 
but $\D$ is finite the result cannot be true without extra conditions (consider $\M = B(L^2([0,1]) )$, and  
$\A = \D$, the latter being the copy of $L^\infty([0,1])$ in $\M$ as in in Example \ref{abave}).     

\begin{corollary} \label{HRf2} Consider the inclusions $\D \subset \A \subset \M,$ where $\M$ is a von Neumann algebra, $\A$ is a
 subalgebra of $\M$, and $\D$ is a von Neumann subalgebra containing the unit of $\M$.
 Let $\omega$ be a faithful normal tracial state on $\mathcal{D}$. Let $\Phi : \A \to \mathcal{D}$ be a unital
 $\mathcal{D}$-bimodule map (or equivalently is the identity map on $\mathcal{D}$).   Assume also that  $\A + \A^*$ is weak* dense in $\M$. 
  Then there is a normal conditional expectation $\Psi : \M \to \mathcal{D}$ extending $\Phi$
if and only if
 $\omega \circ \Phi$ extends to a normal state $\psi$ on $\M$.   In this case $\Psi$ is $\psi$-preserving.  \end{corollary}

\begin{proof}  One direction is obvious.
For the other, if  $\omega \circ \Phi$ extends to a normal state $\psi$ on $\M$, then
we have  $\omega(\Phi(ad)) =  \omega(\Phi(a)d) =  \omega(d \Phi(a))  = \omega(\Phi(da))$ for $d \in \mathcal{D}, a\in \A$.   Thus
 $\psi(xd) = \psi(dx)$ for $x \in \A$, hence for $x \in \M$.
  Now appeal to Corollary \ref{HRf}: there is a normal $\psi$-preserving
conditional expectation $\Psi : \M \to \mathcal{D}$.
For any normal conditional expectation $\Psi : \M \to \mathcal{D}$ extending $\Phi$
we have $$\psi(\Psi(a)) = \psi(\Phi(a)) = \omega(\Phi(a)) = \psi(a), \qquad a \in \A 
, $$
so that $\Psi$ is $\psi$-preserving.
 \end{proof}

\begin{remark}
The condition that $\A +\A^*$ is weak* dense may not simply be dropped in the last result.
The example mentioned before the corollary shows this (with
$\D = \A = L^\infty(\Tdb)$).

 In several results in the present paper, a condition on $\M$ (and perhaps also on $\D$) may be replaced 
by a  condition on an intermediate subalgebra $\A$ between $\M$ and $\D$, if  $\A + \A^*$ is weak* dense in $\M$ (a condition 
which is crucial for example  for Arveson's subdiagonal algebras \cite{Arv}, or for much of the abstract analytic function theory 
from the 1960's and 70's \cite{Gam,BLsurv}).     This follows e.g.\ by tricks seen in the last proof.  
\end{remark}

\section{The noncommutative Hoffman-Rossi theorem for sigma-finite algebras} \label{HR2} 

\begin{lemma}  \label{mth}  Let $(K,  \mu)$ be a 
measure 
 space and $g \in L^1(K,\mu)_+$.
Then $$(\int_K \, f \, d \mu)^2 \leq \int_K \, f^2 g \, d \mu, \qquad
f \in L^\infty(K,\mu)_+,$$ if and only if $g$ is
$\mu$-a.e.\ nonzero and $\int_K \, g^{-1} \, d \mu
\leq 1$. 
\end{lemma}

\begin{proof}  Suppose that $\int_K \, g^{-1} \, d \mu
\leq 1$.   Then by Cauchy-Schwarz, 
$$(\int_K \, f \, g^{\frac{1}{2}} \,
g^{-\frac{1}{2}} \, d \mu)^2   
\leq 
(\int_K \, f^2 g \, d \mu) \, (\int_K \, g^{-1} \, d \mu )
\leq \int_K \, f^2 g \, d \mu .$$

Conversely, suppose that the inequality holds. 
Letting $E = \{ x \in K : g(x) = 0 \}$, and $f = \chi_E$ we see that $\mu(E)^2 = 0$.   So we may assume that $g$ is never
$0$, so that $g^{-1}$ is well defined and finite.  
 Letting $E_n = \{ x \in K : g(x)^{-1} \leq n \}$, and  $f_n = g^{-1}
\, \chi_{E_n}$, we have 
$$(\int_{E_n} \, g^{-1} \, d \mu)^2 \leq 
\int_{E_n} \, (g^{-1})^2 \, g \, d \mu = 
\int_{E_n} \, g^{-1} \, d \mu .$$  
Thus $\int_{E_n} \, g^{-1} \, d \mu
\leq 1$.   Letting $n \to \infty$ we have $\int_{K} \, g^{-1} \, d \mu \leq 1$. \end{proof}

\begin{remark} 
We thank Vaughn Climenhaga for discussions around the last result.  
\end{remark}

The following is the generalized  Hoffman-Rossi theorem for subalgebras of `finite' von Neumann algebras.  See \cite{BFZ} for the 
case when $\D$ is atomic.

\begin{theorem} \label{HRc2}   Consider the inclusions $\D \subset \A \subset \M,$ where $\M$ is a von Neumann algebra
with faithful normal tracial state $\tau$, $\A$ is a
weak* closed subalgebra of $\M$, and $\D$ is a von Neumann subalgebra containing the unit of $\M$.
Let $\Phi : \A \to \mathcal{D}$ be a weak* continuous  $\D$-character.   
Then there exists a normal conditional expectation $\Psi : \M \to \mathcal{D}$ extending $\Phi$. 
Indeed $\tau \circ \Phi$  extends to a  normal state $\rho$ on $\M$ such that 
there exists a unique  $\rho$-preserving normal conditional expectation $\Psi : \M \to \mathcal{D}$ extending $\Phi$. 
\end{theorem}

\begin{proof}  Proceeding as in the proof of Corollary \ref{HRc} we obtain $r = ab^* \in L^1(\M)$ and $c \in [\A a]_2$, 
such that   $x \mapsto \tau(x c c^*)$  annihilates $J$, and 
 $$\tau(d^*d)^2 \leq  \| a \|_2^2  \, \tau( (d^* d)^2 c c^*) =  \| a \|_2^2  \, \tau( (d^* d)^2 E_{\D}(c c^*)), \qquad  d \in \D .$$
Letting $g = E_{\D}(c c^*) \in L^1(\D)_+$ and $f = d^* d$ we have $$\tau(f)^2 \leq    \| a \|_2^2  \, \tau( f^2g), \qquad  f \in \D_+ .$$
Let $\M_0$ be the von Neumann algebra generated by $g$ (see e.g.\ \cite[p.\ 349]{KR1}), a commutative subalgebra of $\D$.
For $f \in (\M_0)_+$ we have $\tau(f)^2 \leq    \| a \|_2^2  \, \tau( f^2g)$.   This implies  that 
$g^{-1} \in L^1(\M_0)_+$  by Lemma \ref{mth}.   Hence as a closed densely defined positive operator affiliated with $\D$, 
we have $g^{-1} \in L^1(\D)_+$.  

Let $c_1 = g^{-\frac{1}{2}} c \in L^1(\M)$, and let $h = g^{-\frac{1}{2}} c c^* g^{-\frac{1}{2}}$ (the so called `strong product' \cite[p.\ 174]{tak2}).   Let $\hat{E}$ be Haagerup's extension of $E_{\D}$ to the 
extended positive part $\hat{\M}_+$.
 Since $\hat{E}(c c^*) = E(c c^*) = g \in L^1(\M)$, by \cite[Proposition 3.3]{Gol} we have 
$\hat{E}(h) = g^{-\frac{1}{2}} \cdot g \cdot g^{-\frac{1}{2}}$, where the latter is the
strong product.  But this is $1$.  
Thus $$\hat{\tau}(h) = \tau(g \bullet g^{-\frac{1}{2}}) = \hat{\tau}(g^{-\frac{1}{2}} \cdot g \cdot g^{-\frac{1}{2}} ) 
= 1 ,$$
by  Propositions  2.6 (used twice) and  3.2
in \cite{Gol}.  Thus $h \in L^1(\M)$ with $\| h \|_1 = 1$
(see Remark 1.2 in \cite{LabCo}), and $c_1 \in L^2(\M)$ 
with $\| c_1 \|_2 = 1$.  
Then $\omega = \tau( \cdot \, h)$ is a normal state on $\M$
which extends $\tau_{|\D}$.  Indeed $\tau(dh) = \tau(E_{\D}(dh))= \tau(dE_{\D}(h))
= \tau(d)$ for $d \in \D$.

 We now check that $\omega$ 
  annihilates $J$.  Let $d_n$ be $g^{-\frac{1}{2}}$ times the spectral projection of $[0,n]$ for $g^{-\frac{1}{2}}$.
So $d_n \nearrow g^{-\frac{1}{2}}$, and this convergence is also in $L^2$-norm (as one can see by spectral theory).
Then, similarly to the above, 
$$\tau((g^{-\frac{1}{2}} -d_n) cc^* (g^{-\frac{1}{2}} - d_n)) = 
\tau ( cc^* \bullet (g^{-\frac{1}{2}} -d_n)) = \tau(g  \bullet (g^{-\frac{1}{2}} -d_n)) =
\tau((1 - g^{\frac{1}{2}} d_n)^2).$$
This equals $\tau(p_n),$
where $p_n$ is the  spectral projection of $[n,\infty)$ for $g^{-\frac{1}{2}}$.
Clearly  $1-p_n \nearrow 1$ weak*. 
We deduce that $c^*d_n \to c^* g^{-\frac{1}{2}}$
and $d_n c \to g^{-\frac{1}{2}} c$  in $2$-norm.

It follows that $g^{-\frac{1}{2}} c \in  E = [Aa]_2$.  We have $j  g^{-\frac{1}{2}} c \in F = [J a]_2 \subset E$.
Thus $$\tau(j  g^{-\frac{1}{2}} c c^*) = \langle  j g^{-\frac{1}{2}} c , c \rangle = 0, $$
since $b \perp F$ and $c \perp F$ as before.  Hence 
$$0 = \lim_n \, \tau(j g^{-\frac{1}{2}} c c^* d_n 
)  =  \tau(j  g^{-\frac{1}{2}} c c^*  g^{-\frac{1}{2}} )
=  \tau(j h ) .$$ 
So $\omega$  annihilates $J$.

Since $\omega$ extends $\tau_{|D}$, we see that $\omega$ extends $\tau \circ \Phi$.   
Next appeal to  Corollary \ref{HRf3} (2) to obtain  a  normal conditional expectation $\Psi
: \M \to \mathcal{D}$ extending $\Phi$.   If $\rho = \tau \circ \Psi$ then $\rho$ 
extends $\tau \circ \Phi$, and $\Psi$ is $\rho$-preserving.   The uniqueness is as before.  
\end{proof}

\begin{remark} 
If  $\Phi : \A \to \D$ is a 
homomorphism and projection 
onto a $C^*$-subalgebra $\D$ of $\A$,
and if $\Phi$ is faithful on $\A \cap \A^*$,
 then  it follows by an argument of Arveson \cite{Arv} that $\D = \A \cap
\A^*$:  Certainly $\D \subset \A \cap \A^*$.
If $x \in (\A \cap \A^*)_{\rm sa}$ then
$\Phi((x - \Phi(x))^2) = (\Phi(x - \Phi(x)))^2 = 0$.  So    $(x - \Phi(x))^2
= x - \Phi(x) =0$, hence  $x = \Phi(x) \in \D$.    Thus $\D = \A \cap \A^*$.
This holds in particular  if $\Phi$ is preserved by a faithful state.  
\end{remark}

The following is the generalized  Hoffman-Rossi theorem for certain subalgebras of $\sigma$-finite von Neumann algebras--in this case $\D$ is still 
finite (that is, has a faithful normal tracial state $\tau_{\D}$), but $\M$ need not be finite.     

\begin{theorem} \label{HRsfcent}  Consider inclusions $\D \subset \A \subset \M,$ where $\M$ is a von Neumann algebra 
with faithful normal state $\omega$  with $\D$ in its centralizer, $\A$ is a 
weak* closed subalgebra of $\M$, and $\D$ is a von Neumann subalgebra  containing the unit of $\M$. 
Let $\Phi : \A \to \mathcal{D}$ be a weak* continuous  $\D$-character.    
Then there exists a normal conditional expectation $\Psi : \M \to \mathcal{D}$ extending $\Phi$. 
Indeed $\omega \circ \Phi$  extends to a  normal state $\rho$ on $\M$ such that 
there exists a unique  $\rho$-preserving normal conditional expectation $\Psi : \M \to \mathcal{D}$ extending $\Phi$. 
\end{theorem}

\begin{proof}  Note that the fact that $\D\subset \M_\omega$, ensures that $\D$ is the image of an $\omega$-preserving faithful normal conditional expectation $E_{\D} : \M \to \D$. We then know from Remark \ref{crossprod2} that $\D\rtimes_\omega\mathbb{R}\subseteq\M\rtimes_\omega\mathbb{R}$, whence $\widetilde{\D\rtimes_\omega\mathbb{R}}\subseteq\widetilde{\M\rtimes_\omega\mathbb{R}}$. Thus for each $p>0$, 
\begin{align*} L^p(\D) & = \{a\in \widetilde{\D\rtimes_\omega\mathbb{R}}\colon \theta_s(a)=e^{-s/p}a\mbox{ for all} \; s\} \\ 
 & \subseteq\{a\in \widetilde{\M\rtimes_\omega\mathbb{R}}\colon \theta_s(a)=e^{-s/p}a\mbox{ for all} \; s\}=L^p(\M). \end{align*}
It is clear from Remark \ref{crossprod} (and perhaps the fact in Remark \ref{crossprod2} that  $E_\D \circ\sigma_t^\omega=\sigma_t^\omega \circ E_\D$) 
 that the dual action  $\theta_s$ corresponding to $\M$  agrees with the one corresponding to $\D$.
 However the fact that $\D\subset \M_\omega$, also ensures that $\omega_{|\D}$ is a trace on $\D$. So by for example pages 62-63 of \cite{terp} or the introductory discussion in Section 2 of \cite{LabCo}  (or 
Theorem 6.74 in \cite{GLnotes} and Section 1 in \cite{HJX}), $\D\rtimes_\omega\mathbb{R}$, the density $k_\omega=\frac{d\widetilde{\omega_{|\D}}}{d\tau}$ 
(which we observed in Remark \ref{crossprod2} equals the density of the dual weight $\tilde{\omega}$ on $\M \rtimes_\omega\mathbb{R}$), and $L^p(\D)$ will respectively up to Fourier transform canonically correspond to $\D\overline{\otimes}L^\infty(\mathbb{R})$, $\I\otimes\exp$ and $\{a\otimes(\exp)^{1/p}:a\in L^p(\D,\omega)\}$, where $L^p(\D,\omega)$ denotes the tracial $L^p$-space constructed using the trace $\omega_{|\D}$. The density $k_\omega\in L^1(\D)$ then clearly commutes with $\D$ 
(see also Proposition \ref{commhk} 
for another way to see this).

We momentarily follow the argument for Corollary \ref{HRc}:
By Banach space duality $\omega \circ \Phi$  extends to a weak* continuous functional on $\M$, so that there exists $r = ab^* \in L^1(\M),$ with $a, b \in L^2(\M)$, such that $tr(r x) = \omega(\Phi(x))$ for all $x \in \A$. In particular $$tr(rd) = tr(da b^*) =  \omega(d) \,  , \; \; \; 
\; tr(r j) = tr(ja b^*) = 0, \qquad d \in \D , j \in J = {\rm Ker}(\Phi),$$ 
so that $b \perp F$ where $F = [Ja]_2$. We apply 
 an idea that appears to go back to Sarason and others, as discussed in  \cite{BFZ}.  Let $c\in L^2(\M)$ be the projection $P_E(b)$ of $b$ onto $E = [\A a]_2$.
Then $c \perp F$ since $F \subset E$.  Note that $$tr(j c c^*) = tr(j c b^*)  = 0, \qquad j \in J,$$ since $jc \in J [\A a]_2 \subset F \subset E$.  
 For $d \in \D$ we then have $$\omega(d^*d) = tr(r d^*d) = tr(b^* d^*da) 
=  |\langle d^*d a , b \rangle |  = |\langle d^* da , c \rangle |   = |\langle a , d^*dc \rangle |,$$
which is dominated by 
$\| a \|_2 \| d^* d c \|_2.$
 Thus for $d \in \D$ we have 
$$\omega(d^*d)^2 \leq  \| a \|_2^2 \| d^* d c \|_2^2=\| a \|_2^2  \, tr( (d^* d)^2 c c^*) =  \| a \|_2^2  \, tr( (d^* d)^2 E_{\D}(c c^*)) .$$
Letting $g = E_{\D}(c c^*) \in L^1(\D)_+$ and $f = d^* d$ we have 
$$\omega(f)^2 \leq    \| a \|_2^2  \, tr( f^2g), \qquad  f \in \D_+ .$$ For the sake of clarity we will in the rest of this proof 
dispense with convention, and distinguish between $\D$ and the copy thereof, namely $\pi(\D)$, inside the crossed product. 
Similarly $\tilde{\D}$ may be identified with a subset of $\widetilde{\D\rtimes_\omega\mathbb{R}}$, which we 
will also write as $\pi(\tilde{\D})$.  
 From the 
above discussion it is clear that up to Fourier transform, the action of $\pi$ is to map each $d\in \D$ onto $d\otimes 1$. The above 
inequality should then properly be written as $\omega(f)^2 \leq    \| a \|_2^2  \, tr( \pi(f)^2g)$ for all $f \in \D_+$. We now 
reformulate this inequality in terms of the $L^p(\D,\omega)$ spaces. By the discussion at the start of this proof, $g$ is up to Fourier 
transform of the form $g_0\otimes \exp=(g_0\otimes 1).(\I\otimes\exp)$ for a unique $g_0\in L^1_+(\D,\omega)$. So on taking the inverse 
transform, it then follows that $g=\pi(g_0)k_\omega\equiv (g_0\otimes 1)(\I\otimes \exp)$ for a unique $g_0\in L^1_+(\D,\omega)$. Thus 
the preceding inequality may be written in the form $$\omega(f)^2 \leq    \| a \|_2^2  \, tr( \pi(f)^2\pi(g_0)k_\omega) = \| a \|_2^2  \, \omega( f^2g_0), \qquad  f \in \D_+ .$$ 
The last $\omega$ here is the natural extension of this trace on $\D$ to $\tilde{D}_+$, hence to $L^1(\D,\omega)_+$.

Let $\D_0$ be the von Neumann subalgebra of $\D$ generated by $g_0$ (see e.g.\ \cite[p.\ 349]{KR1}), a commutative subalgebra of $\D$.
For $f \in (\D_0)_+$ we have $\omega(f)^2 \leq    \| a \|_2^2  \, \omega( f^2g_0)$. By Lemma \ref{mth} this implies  that 
$g_0^{-1} \in L^1(\D_0,\omega)_+\subset L^1(\D,\omega)_+$. Thus as a closed densely defined positive operator affiliated with $\D$, 
we have $g_0^{-1} \in L^1(\D,\omega)_+$.  
Writing $\widetilde{g}$ for $k_\omega \pi(g_0^{-1})$, this then ensures that $\widetilde{g}^{1/2}=k^{1/2}_\omega \pi(g_0^{-\frac{1}{2}})\in L^2(\D)$. (Here we used the definition of Haagerup's $L^2(\D)$, and the fact that $k_\omega$ commutes strongly with $\pi(\D)$.)

Let $d_n$ be $g_0^{-\frac{1}{2}}e_n$ where $e_n$ is the spectral projection of $g_0^{-\frac{1}{2}}$ for $[0,n]$. So $e_n  \nearrow 1$ and 
$d_n \nearrow g_0^{-\frac{1}{2}}$ as $n\to \infty$. Notice that for $m\geq n$ we may use the fact that  $e_ne_m=e_n$ to see that  
\begin{align*} & tr(|c^*\pi(d_n)-c^*\pi(d_m)|^2) \\
& =tr(\pi(d_n)cc^*\pi(d_n)-\pi(d_m)cc^*\pi(d_n)-\pi(d_n)cc^*\pi(d_m)+\pi(d_m)cc^*\pi(d_m)) \\ 
& =tr(\pi(d_m)cc^*\pi(d_m)-\pi(d_n)cc^*\pi(d_n)) =tr(E_{\D}(\pi(d_m)cc^*\pi(d_m)-\pi(d_n)cc^*\pi(d_n))) \\
& =tr(\pi(d_m)g\pi(d_m)-\pi(d_n)g\pi(d_n))=tr(\pi(e_m)k_\omega - \pi(e_n)k_\omega)=\omega(e_m-e_n). \end{align*} 
(In the second equality we used the fact that since $e_ne_m=e_n$, we for example have that $$tr(\pi(d_m)cc^*\pi(d_n))=tr(\pi(d_n)\pi(d_m)cc^*)=tr(\pi(d_n)^2cc^*)=tr(\pi(d_n)cc^*\pi(d_n)).$$ In the second last equality we 
used the fact that $k_\omega$ commutes with $\pi(\D)$, and that $g=\pi(g_0)k_\omega$.) The projections $e_n$ increase to $\I$ and so converge $\sigma$-strongly to $\I$. It now clearly follows from the above sequence of 
equalities that $$\lim_{n,m\to\infty}tr(|c^*\pi(d_n)-c^*\pi(d_m)|^2)=\lim_{n,m\to\infty}|\omega(e_m-e_n)|=\omega(\I-\I)=0.$$ So $(c^*\pi(d_n))$ must converge in $L^2$-norm to some $h_1^*\in 
L^2(\M)$, with $(\pi(d_n)c)$ obviously converging to $h_1$. Now let $h=h_1^*h_1$. This element of $L^1(\M)$ in some sense corresponds to 
$\pi(g_0^{-1/2}) c c^* \pi(g_0^{-1/2})\in L^1(\M)$, and must by construction be the limit in $L^1$-norm of $(\pi(d_n)cc^*\pi(d_n))$. To see this observe that
$$\|h_1^*h_1-d_ncc^*d_n\|_1= \|h_1^*h_1-h_1^*c^*d_n+h_1^*c^*d_n -d_ncc^*d_n\|_1$$ $$\leq \|h_1\|_2\|h_1-c^*d_n\|_2 +\|c^*d_n\|_2\|h_1^*-d_ncd_n\|_2.$$For each $x\in \D$, 
$(e_nxe_n)$ is $\sigma$-weakly convergent to $x$. So for such an $x$, we may use $L^p$-duality to see that 
$$tr(E_{\D}(h)\pi(x))=\lim_{n\to\infty}tr(E_{\D}(\pi(d_n)cc^*\pi(d_n))\pi(x))=\lim_{n\to\infty}tr(\pi(d_n)g\pi(d_n)\pi(x))$$ $$=\lim_{n\to\infty}tr(\pi(e_n)k_\omega \pi(e_nx)) =\lim_{n\to\infty}tr(k_\omega \pi(e_nxe_n))=tr(k_\omega \pi(x)).$$But then clearly $E_{\D}(h)=\lim_{n\to\infty}E_{\D}(\pi(d_n)cc^*\pi(d_n))=k_\omega$. 
Thus $$tr(h) = tr(E_{\D}(h)) = tr(k_\omega) =\omega(\I) = 1,$$ thereby ensuring that $h \in L^1_+(\M)$ with $\| h \|_1 = 1$, and 
$h_1 \in L^2(\M)$ with $\| h_1 \|_2 = 1$. Then $\vartheta = tr( \cdot \, h)$ is a normal state on $\M$. Since  $$\vartheta(d)=tr(\pi(d)h)=tr(E_{\D}(\pi(d)h))=tr(\pi(d)k_\omega)=\omega(d), \qquad d\in \D,$$ we have 
that $\vartheta$  agrees with $\omega$ on $\D$. 

We now check that $\vartheta$ annihilates $J$. With $(d_n)$ as before we clearly have that $(\pi(d_n) c)\subset E = [\pi(A)a]_2$, and 
hence that $h_1 \in  E = [\pi(A)a]_2$. But then $\pi(j)  h_1 \in F = [\pi(J) a]_2 \subset E$. Thus 
$tr(\pi(j)  h_1 c^*) = \langle  \pi(j) h_1 , c \rangle = 0$, since as before $b \perp F$ and $c \perp F$.  Hence 
$$0 = \lim_n \, tr(\pi(j) h_1(c^*\pi(d_n))) = tr(\pi(j)h_1h_1^*) = tr(\pi(j)h).$$
So $\vartheta$ annihilates $J$ as required.

Since $\vartheta$ extends $\omega_{|\D}$ and  annihilates $J$, it is a normal state on $\M$ agreeing with $\omega \circ \Phi$ on $\A = D + J$. 
Next appeal to Proposition \ref{HRf3g} (2) to obtain  a  normal conditional expectation $\Psi: \M \to \mathcal{D}$ extending $\Phi$. 
If $\rho = \omega \circ \Psi$ then $\rho$ extends $\omega \circ \Phi$, and $\Psi$ is $\rho$-preserving. The uniqueness is as before.  
\end{proof}

Again it is easy to see the converse.   Namely, if $\Phi$ is the restriction to $\A$ of  
a normal conditional expectation of $\M$ onto $\D$, then  
for any normal state 
$\omega$ on $\D$, $\omega \circ \Phi$  extends to a  normal state $\rho$ on 
$\M$, and 
there is a  $\rho$-preserving normal conditional expectation $\M \to \mathcal{D}$ extending $\Phi$.   The latter is unique 
if $\omega$ is faithful on $\D$.

\section{Representing measures and the Hoffman-Rossi theorem for general von Neumann algebras} \label{repm} 

Our first result gives a sufficient condition under which the main result of the previous section holds for a general von Neumann algebra $\M$.
This result also generalizes Corollary \ref{HRc} to a large class of noncommutative situations.

\begin{theorem} \label{HRc2-sfcomm} Consider the inclusions $\D \subset \A \subset \M,$ where $\M$ is a von Neumann algebra
with faithful normal semifinite weight 
$\omega$, $\A$ is a weak* closed subalgebra of $\M$, and $\D$ is a von Neumann subalgebra  of $\M_\omega$.
Let $\Phi : \A \to \mathcal{D}$ be a weak* continuous  $\D$-character.  
Suppose that $\I$ is the sum of a collection $\{e_t\}$ of mutually orthogonal $\omega$-finite projections in $\D$, which are central in 
$\D$. Then there exists a normal conditional expectation $\Psi : \M \to \mathcal{D}$ extending $\Phi$. Indeed $\tau \circ \Phi$ extends to a normal semifinite weight $\rho$ on $\M$ for which there exists a unique 
$\rho$-preserving normal conditional expectation $\Psi : \M \to \mathcal{D}$ extending $\Phi$. 
\end{theorem}

Note that the existence of the collection of projections $\{e_t\}$ above, ensures that the restriction of $\omega$ to $\D$ is still semifinite and hence
by \cite[Theorem IX.4.2]{tak2}  there exists a normal conditional expectation $E_{\D}$ onto $\D$ preserving $\tau$, since $\D \subset \M_\omega$.

\begin{proof}  By Lemma \ref{istrD} $\omega$ is tracial on $\D$.
We may apply Theorem \ref{HRsfcent} to each of the compressions $e_t \M e_t$.
Since $\omega (e_t d e_t x e_t) =  \omega (e_t d e_t x e_t)$ for all $x \in \M, d \in \D$, we  obtain a normal conditional expectation 
$\Psi_e : e \M e \to e\mathcal{D}e$ extending $\Phi_t=\Phi_{\vert e_t\A e_t}$. (Note that since $\Phi(e_tae_t)=e_t\Phi(a)e_t$ for all 
$a\in \A$, $\Phi_{\vert e\M e}$ maps $e_t\A e_t$ onto $e_t \D e_t$.)   We define the map $\Psi:\M\to \D$ by 
$\Psi(x)= \oplus_t \Psi_t(e_txe_t)$. Each $\Psi_t$ is normal, and so the sum $\Psi$ will also be. We first show that this 
map extends $\Psi$. To see this note that $\Psi$ clearly extends $\oplus_t \Phi_t$. It is an easy exercise to see that  for any 
$a\in A$ we have that $\oplus_t \Phi_t(a)=\oplus_t e_t\Phi(a)e_t= \Phi(\oplus_t e_tae_t)$. The $\D$-centrality of the $e_t$'s ensure that if 
$t\neq s$, then $\Phi(e_tae_s)=e_t\Phi(a)e_s=e_se_t\Phi(a)=0$. So $\Phi(a)=\Phi((\oplus_t e_t)a(\oplus_se_s)) = \oplus_t \Phi_t(a)$ for any $a\in A$. Thus $\Psi$ extends $\Phi$.

It remains to show that $\Psi$ is a conditional expectation. To see this, note that for any $x\in \M$ we will have that 
$$\Psi(\Psi(x)) = \oplus_t \Psi_t(e_t(\oplus_r \Psi_r(e_rxe_r))e_t)=\oplus_t \Psi_t(\Psi_t(e_txe_t))=\oplus_t \Psi_t(e_txe_t)=\Psi(x).$$

The proof of the final claim now proceeds as  in Theorem \ref{HRc2}.
\end{proof}

We pass to proving a noncommutative Hoffman-Rossi theorem for general von Neumann algebras. We shall use the Haagerup reduction theorem to extract this result from the one for the finite algebras. In so doing we shall first prove a result for $\sigma$-finite algebras (because of the relative simplicity of the reduction theorem in this setting) before indicating how the proof technique may be adaped to yield a theorem for general von Neumann algebras. Though we eventually do get a theorem for a more 
general class of von Neumann algebras, it is at the cost of normality of the conditional expectation extending the given $\D$-character.

 \begin{theorem} Consider the inclusions $\D \subset \A \subset \M,$ where $\M$ is a $\sigma$-finite von Neumann algebra
with faithful normal state $\varphi$, $\A$ is a weak* closed subalgebra of $\M$, and $\D$ is a von Neumann subalgebra 
containing the unit of $\M$. Suppose that $\sigma^\varphi_t(\A)=\A$ and $\sigma^\varphi_t(\D)=\D$ for each $t\in \mathbb{R}$. Let 
$\Phi : \A \to \mathcal{D}$ be a weak* continuous completely bounded $\D$-character which commutes with $(\sigma^\varphi_t)$. Then there 
exists a possibly non-normal conditional expectation $\Psi$ from $\M$ onto $\D$ extending $\Phi$. 
\end{theorem}

\begin{proof} We will closely follow the notation of \cite{HJX}. (The proof was originally written down in May 1978 
by Haagerup and circulated as a hand-written copy. An outline of the proof appeared in print in 2005 (see \cite{Xu}). 
However it was not until the appearance of \cite{HJX} that a full proof appeared in print.) Write $\R$ for 
$\M\rtimes\mathbb{Q}_D$ where $\mathbb{Q}_D$ are the diadic rationals. 
We shall identify $\M$ with the *-isomorphic copy thereof inside $\M\rtimes\mathbb{Q}_D$. It is known that $\R$ is the weak* closure of 
$\mathrm{span}\{\lambda_t x: x\in \M, t\in \mathbb{Q}_D\}$. We shall write $\widehat{\A}$ and $\widehat{\D}$ for the weak* closures of 
$\mathrm{span}\{\lambda_t a: a\in \A, t\in \mathbb{Q}_D\}$ and $\mathrm{span}\{\lambda_t d: d\in \D, t\in \mathbb{Q}_D\}$ respectively. 
Note that $\mathrm{span}\{\lambda_t a: a\in \M, t\in \mathbb{Q}_D\}$ is in fact a dense *-algebra, since $$\lambda_t a_0\lambda_s a_1 = \lambda_{t+s}\lambda_{-s}a_0\lambda_s a_1 = \lambda_{t-s}\sigma^\varphi_{-s}(a_0)a_1.$$  The prescription 
$\widehat{\Phi}(\lambda_ta)=\lambda_t\Phi(a)$ defines a map on a dense subalgebra of $\widehat{\A}$, which by the technique of 
\cite[Theorem 4.1]{HJX} extends to a bounded normal map from $\widehat{\A}$ to $\widehat{\D}$. Another way to see these facts is to note 
that the complete boundedness of $\Phi$ allows us to extend it to a map $\Phi\otimes\mathrm{Id}$ on $\A\otimes B(L^2(\mathbb{Q}_D))$ 
with preservation of norm. We may now without loss of generality assume that $\M$ is in standard form, and hence that the modular group 
is implemented. Consulting  \cite[Part I, Proposition 2.8]{vD} now gives a very precise description of how $\M\rtimes\mathbb{Q}_D$ may be realised 
as a subspace of $\M\otimes B(L^2(\mathbb{Q}_D))$. What remains now is to check that $\widehat{\A}$ similarly lives inside 
$\A\otimes B(L^2(\mathbb{Q}_D))\subset\M\otimes B(L^2(\mathbb{Q}_D))$, that $\Phi\otimes\mathrm{Id}$ is still normal and that the demand 
that $\Phi$ commutes with $\sigma^\varphi_t$, ensures that $\widehat{\A}$ is an invariant subspace of the action of 
$\Phi\otimes\mathrm{Id}$ on $\A\otimes B(L^2(\mathbb{Q}_D))$. The restriction of $\Phi\otimes\mathrm{Id}$ to this invariant subspace is 
the map $\widehat{\Phi}$ we seek. Note that by definition (and normality) $\widehat{\Phi}$ will map $\widehat{\A}$ onto $\widehat{\D}$. 
The map $\widehat{\Phi}$ is moreover a homomorphism. The easiest way to see this is to note the earlier computations may be modified to 
show that for any $a_0,a_1\in \A$ and $t,s\in \mathbb{Q}_D$ we have that $$\widehat{\Phi}(\lambda_t a_0\lambda_s a_1) = \widehat{\Phi}(\lambda_{t-s}\sigma^\varphi_{-s}(a_0)a_1)=\lambda_{t-s}\Phi(\sigma^\varphi_{-s}(a_0)a_1)$$ $$= \lambda_{t-s}\sigma^\varphi_{-s}(\Phi(a_0))\Phi(a_1)=\lambda_t \Phi(a_0)\lambda_s \Phi(a_1)= \widehat{\Phi}(\lambda_ta_0)\widehat{\Phi}(\lambda_s a_1).$$  Thus it is a $\widehat{\D}$-character.

There exists an increasing sequence $(\R_n)$ of finite von Neumann subalgebras of $\R$ for which $\cup\R_n$ is weak*-dense in $\R$. 
Moreover with $\widehat{\varphi}$ denoting the dual weight (it is a state) on $\R$, there exist normal conditional expectations 
$$\Edb_\M:\R\to \M \mbox{ and }\Edb_n:\R\to\R_n$$commuting with the automorphism group $\sigma^{\widehat{\varphi}}_t$ 
\cite[Lemma 2.4]{HJX}. We will 
write $\A_n$ and $\D_n$ for $\Edb_n(\widehat{\A})$ and $\Edb_n(\widehat{\D})$ respectively. We claim that $\A_n\subset \widehat{\A}$ and 
$\D_n\subset\widehat{\D}$ for each $n$. To see this note that $a_n$ as defined prior to \cite[Lemma 2.4]{HJX} belongs to the von Neumann 
subalgebra generated by the $\lambda_t$'s and hence to $\widehat{\A}$. The fact that $\sigma^{\widehat{\varphi}}_t$ is implemented by 
the unitary group $t\to\lambda_t$, ensures that $\sigma^{\widehat{\varphi}}_t(\widehat{\A})=\widehat{\A}$ for every $t$. Hence by 
\cite[Equation(2.4)]{HJX} $\sigma^{\varphi_n}_t(a)\in\widehat{\A}$ for every $t$. So the definition of $\Edb_n$ on 
page 2133 of \cite{HJX} ensures that $\A_n\subset \widehat{\A}$. Similarly $\D_n\subset\widehat{\D}$. In fact on modifying the ideas in 
\cite[Lemmata 2.6 \& 2.7]{HJX}, one can show that the $\A_n$'s and $\D_n$'s are increasing with $\cup \A_n$ and $\cup \D_n$ are 
respectively dense in $\widehat{\A}$ and $\widehat{\D}$. To see the claim about being increasing, note that the last 5 lines of page 
2133 of \cite{HJX} shows that if $x\in \R_n$, then $\sigma^{\varphi_n}_t(x)=x$ for any $t$. This suffices to show that if say $a\in \A_n$ 
it follows from the definition of $\Edb_{n+1}$ (again on p 2133 of \cite{HJX}) that $\Edb_{n+1}(a)=a$. So clearly $\A_n\subset \A_{n+1}$ 
for every $n$. Similarly $\D_n\subset \D_{n+1}$. 

The fact that $\Phi$ commutes with $\sigma^\varphi_t$, ensures that $\widehat{\Phi}$ commutes with $\sigma^{\widehat{\varphi}}_t$ (this 
can also be verified with the technique of \cite[Theorem 4.1]{HJX}). A version of  \cite[Theorem 4.1(i)]{HJX} similarly holds for 
$\widehat{\Phi}$. When these facts are considered alongside \cite[Equation(2.4)]{HJX}, it follows that 
$\sigma^{\varphi_n}_t\circ\widehat{\Phi}=\widehat{\Phi}\circ\sigma^{\varphi_n}_t$ for all $t$. Thus again by the definition of $\Edb_n$, 
it follows that $\widehat{\Phi}\circ\Edb_n= \Edb_n\circ\widehat{\Phi}$, and hence that $\widehat{\Phi}$ maps $\A_n$ into $\D_n$. In fact 
for the same reason it also follows that $\Edb_n\circ\widehat{\Phi}=\widehat{\Phi}\circ \Edb_n$. We will 
write $\Phi_n$ for the induced map from $\A_n$ to $\D_n$. The inclusions $\A_n\subset \widehat{\A}$ and $\D_n\subset\widehat{\D}$ ensure 
that each $\Phi_n$ is a $\D_n$-character. So by the result for finite von Neumann algebras, each $\Phi_n$ extends to a normal conditional 
expectation $\Psi_n:\A_n\to \D_n$.

Now consider the maps $\Theta_n=\Psi_n\circ\Edb_n$. Each $\Theta_n$ is a completely bounded normal conditional expectation which agrees 
with $\widehat{\Phi}$ on $\A_n$. Thus they all belong to $CB(\R,\widehat{\D})$. Take a weak* limit point $\widehat{\Psi}$ of the 
$\Theta_n$'s. The definition of $\Phi_n$ combined with the fact that $\Edb_n\circ\widehat{\Phi}=\widehat{\Phi}\circ \Edb_n$, ensures 
that in fact each $\Theta_n$ is a normal conditional expectation from $\R$ to $\D_n$ which agrees with $\Edb_n\circ\widehat{\Phi}$ on 
all of $\widehat{\A}$, and hence with $\Edb_n\circ\Phi$ on $\A$. But for each $a\in \A$, $\Edb_n\circ\Phi(a)\to \Phi(a)$ 
$\sigma$-strongly as $n\to \infty$. (This follows from for example \cite[Lemma 2.7]{HJX}.) It follows that $\widehat{\Psi}$ is a 
possibly non-normal conditional expectation from $\R$ to a subspace of $\widehat{\D}$, which agrees with $\Phi$ on $\A$. So 
$\widehat{\Psi}(\M)$ must include $\D=\Phi(\A)$. By construction $\widehat{\Psi}$ also maps $\cup_n\R_n$ onto $\cup_nD_n$. Thus if 
indeed $\widehat{\Psi}$ was normal, we would be able to conclude that $\widehat{\Psi}$ maps $\R$ onto $\widehat{\D}$. However for now 
the most we can say is that $\D \subset \widehat{\Psi}(\M)\subset\widehat{\D}$. The conditional expectation $\Edb_\M$ maps 
$\widehat{\D}$ (and hence also $\widehat{\Psi}(\M)$) onto $\D$. So setting $\Psi=\Edb_\M\circ\widehat{\Psi}_{|\M}$ will yield a possibly 
non-normal conditional conditional expectation from $\M$ onto $\D$ which agrees with $\Phi$ on $\A$.
\end{proof} 

We now present a version of the above theorem for general von Neumann algebras. Here we require the canonical weight on $\M$ to be strictly semifinite on $\D$. We point out that this restriction is a consequence of the criteria required for the application of the reduction theorem and not of the proof technique (see 
\cite[Remark 2.8]{HJX}). If this restriction can be lifted from the reduction theorem, it can therefore also be lifted here. We also note that for tracial weights the concepts of semifinite and strictly semifinite agree. Hence in the tracial setting this theorem is probably as general as it can be.

\begin{theorem} \label{mostgen} Consider the inclusions $\D \subset \A \subset \M,$ where $\M$ is an arbitrary von Neumann algebra, $\A$ is a weak* 
closed subalgebra of $\M$, and $\D$ is a von Neumann subalgebra containing the unit of $\M$. Suppose that $\M$ is equipped with a fns 
weight $\varphi$ for which the restriction to $\D$ is strictly semifinite, and that $\sigma^\varphi_t(\A)=\A$ and 
$\sigma^\varphi_t(\D)=\D$ for each $t\in \mathbb{R}$. Let $\Phi : \A \to \mathcal{D}$ be a weak* continuous completely bounded 
$\D$-character which commutes with $(\sigma^\varphi_t)$. Then there exists a possibly non-normal conditional expectation $\Psi$ from $\M$ onto $\D$ 
extending $\Phi$.
\end{theorem}

\begin{proof} By \cite[Exercise VIII.2(1)]{tak2}, the restriction of $\varphi$ to $\D$ is semifinite on the centralizer $\D_\varphi$. 
The restriction of the weight $\varphi$ to $\D_\varphi$ is again a trace, and hence there exists a net of projections 
$(e_\alpha)\subset\D_\varphi$ all with finite weight, increasing to $\I$. Amongst other facts, the inclusion 
$(e_\alpha)\subset\D_\varphi$ ensures that each $e_\alpha$ is a fixed point of $\sigma^\varphi_t$. When in standard form, the fact of 
$e_\alpha$ being a fixed point of the modular group, translates to the claim that $e_\alpha$ commutes with the modular operator 
$\Delta$. Each compression $e_\alpha\M e_\alpha$ is a $\sigma$-finite von Neumann algebra with respect to the restriction of $\varphi$. 
The commutation of $e_\alpha$ with $\Delta$ further ensures that the modular operator for the pair $(e_\alpha\M e_\alpha, \varphi_{|e_\alpha\M e_\alpha})$ may be identified with $\Delta_{|e_\alpha H}$, and that the modular group of this pair may be identified with the restriction of $\sigma_t^{\varphi_\alpha}=e_\alpha\sigma_t^\varphi(\cdot)e_\alpha$ to $e_\alpha\M e_\alpha$. (The reader may find these facts in the  proof of \cite[Theorem VIII.2.11]{tak2}.) All of this ensures that 
$\sigma_t^{\varphi_\alpha}(e_\alpha\A e_\alpha)=e_\alpha\A e_\alpha$ and that $\sigma_t^{\varphi_\alpha}(e_\alpha\D e_\alpha)= e_\alpha\D e_\alpha$. Moreover $\Phi_{|e_\alpha\A e_\alpha}$ is easily seen to be a $e_\alpha\D e_\alpha$ character on 
$e_\alpha\A e_\alpha$. We may therefore apply the preceding theorem to the compression to see that $\Phi{|e_\alpha\A e_\alpha}$ extends 
to a possibly non-normal conditional expectation $\Psi_\alpha$ from $e_\alpha\M e_\alpha$ to $e_\alpha\D e_\alpha$ which agrees with 
$\Phi_{|e_\alpha\A e_\alpha}$ on $e_\alpha\A e_\alpha$. Now let $\Theta_\alpha$ be given by $\Phi_\alpha\circ C_\alpha$ where $C_\alpha$ 
is the compression $\M\to e_\alpha\M e_\alpha$. Each $\Theta_\alpha$ belongs to $CB(\M, \D)$. Let $\Psi$ be a weak* limit. First notice 
that for any $a\in \A$ we will have that $\Theta_\alpha(e_\alpha a e_\alpha) =\Phi(e_\alpha a e_\alpha) = e_\alpha \Phi(a) e_\alpha\to \Phi(a)$ (since $\{e_\alpha\}$ converges $\sigma$-strong* to $\I$). So $\Psi$ agrees with $\Phi$ on $\A$. Since each $\Theta_\alpha$ is 
a compressed expectation, and since (as we have just seen)  $\Theta_\alpha(e_\alpha d e_\alpha)\to d$ for every $d\in \D$, it 
follows that $\Psi$ is a possibly non-normal conditional expectation from $\M$ to $\D$.   
\end{proof}

\section{Jensen measures} \label{Jens}  

Consider the inclusions $\D \subset \A \subset \mathcal{C},$ where $\mathcal{C}$ is a unital $C^*$-algebra with subalgebra $\A$ and 
$C^*$-subalgebra $\D$, containing $1_{\mathcal{C}}$.
Let $\Psi$ be a ``noncommutative representing measure'' of a $\D$-character $\Phi$,
that is  $\Psi$ is a $\D$-valued normal conditional expectation extending $\Phi$.
Let $\omega$ be a state of $\mathcal{C}$ preserved by $\Psi$ (that is   $\omega\circ\Psi=\omega$).  We say that  $\omega$, or the 
pair $(\Psi, \omega)$,  is a \emph{noncommutative Jensen measure} for $\Phi$, if it satisfies the Jensen-like inequality 
$$\omega(\log(|\Phi(a)|) ) \leq \omega( \log |a|), \quad a \in \A.$$ 
We say that  $\omega$ is an 
\emph{noncommutative Arens-Singer measure} for $\Phi$, if it satisfies this inequality for all invertible $a\in \A$.   These 
inequalities may be rewritten in terms of  the $\omega$-{\em geometric mean} $\Delta_\omega(a) = \exp(\omega( \log |a|))$
(see  the introduction to \cite{BLv}).  For example  noncommutive Arens-Singer measures
satisfy  $$\Delta_\omega(\Phi(a)) \leq  \Delta_\omega(a) , \quad a \in \A^{-1}.$$ 
As we will see in the proof below this inequality is actually an equality.  

In the classical setting we have that $\D=\mathbb{C}\I$ and $\Psi(b)=\omega(b)\I$ for all $b\in \M$, where $\omega(b)=\int b\,d\mu_\omega$ 
and $\mu_\omega$ is the probability 
measure associated with $\Psi$  by the Riesz representation theorem.  It is a straightforward exercise to see that under these identifications the above condition yields exactly the classical definition of Jensen measures
and Arens-Singer measures 
(see e.g.\ \cite{Gam,Stout}, e.g.\ the argument on p.\ 108 in the latter text).

In this section we will prove a theorem concerning the inequalities above which is very closely related to the main 
result of \cite{LL1}.  We will need  two facts noted as 
lemmata in \cite{LL1}: Let $a, b$ be positive invertible
elements of a unital $C^*$-algebra $\mathcal{C}$, with $[a, b] = 0$. Then $b + b^{-1}a \geq
2a^{1/2}$. If for such a positive invertible element $a \in \mathcal{C}$, we inductively define $$x_1 = a, \quad x_{n+1} =
\frac{1}{2}(x_n + ax_n^{-1}),$$then $(x_{n+1})$ decreases
monotonically to $a^{1/2}$ with convergence taking place in the
norm topology.

\begin{lemma}\label{lemma 3} Let $\omega$ be a state on a $C^*$-algebra $\mathcal{C}$. For any $a\in \mathcal{C}$, the sequence $( \omega(|a|^{2^{-n}})^{2^n} )$ is decreasing. If $\omega$ is in fact tracial, then 
\begin{enumerate}
\item $\omega(|a|^p)=\omega(|a^*|^p)$ for any $0<p<\infty$ and any $a\in \mathcal{C}$;
\item given $0<p,q,r\leq \infty$ with $\frac{1}{p}=\frac{1}{q}+\frac{1}{r}$, we will for any $a,b\in \mathcal{C}$ have that 
$\omega(|ab|^p)^{1/p}\leq\omega(|a|^q)^{1/q} \, \omega(|b|^r)^{1/r}$.
\end{enumerate}
\end{lemma}

\begin{proof}
Regarding the first claim note that for any $a\in\mathcal{C}$, we may use the Cauchy-Schwarz inequality to see that $\omega(|a|^{2^{-n}})\leq \omega(|a|^{2^{-n+1}})^{2}\omega(\I)=\omega(|a|^{2^{-n+1}})^{2}$. 

Now suppose that $\omega$ is tracial.  Then the normal extension $\tilde{\omega}$ to $\M = \mathcal{C}^{**}$ is still a trace.  
As in the lines above Lemma 1.1 in \cite{BLv}, there is a central projection $z \in \M$ such that 
$\tilde{\omega}(z \cdot)$ is a faithful normal tracial state on $zM$, and $\tilde{\omega}(z x) = \tilde{\omega}(x)$ for all
$x \in \M$.   So $z \M$ is a  finite von Neumann algebra.   Since the result which we are interested in is true for 
 faithful normal tracial states, we have
$$\omega(|a|^p)=  \tilde{\omega}(z |a|^p) = \tilde{\omega}( |za|^p) =  \tilde{\omega}( |a^* z|^p),$$
and similarly $\omega(|a^*|^p)=  \tilde{\omega}(|a^* z|^p)$.
  We have also used the fact from Lemma 1.2 in \cite{BLv} that $|x|^p z = (|x| z)^p
= |xz|^p$.   In that lemma $x$ was invertible, however this is not used in the proof.  
 An alternative proof: By the Stone-Weierstrass theorem we may select a sequence of polynomials $\{p_n\}$ in one real 
variable which converge uniformly to the function $t\to t^{p/2}$ on $[0,\|a\|]$. It is an easy exercise to use the tracial property to 
see that $\omega((a^*a)^m)=\omega((aa^*)^m)$ for all $m\in\{0,1,2,\dots\}$. (In the case $m=0$ this boils down to the well known fact 
that $\omega(s_r(a))=\omega(s_l(a))$, where $s_r(a)$ and $s_l(a)$ denote the left and tight supports of $a$. But then 
$\omega(p_n(a^*a))=\omega(p_n(aa^*))$ for all $n$, whence 
$\omega(|a|^p)=\lim_{n\to\infty}\omega(p_n(a^*a))=\lim_{n\to\infty}\omega(p_n(aa^*))=\omega(|a^*|^p)$.

The final claim follows from \cite[Theorem 4.9]{FK} and a similar argument to the argument starting the second paragraph of the proof
(reduction to the finite von Neumann algebra case). 
Thus for example,
$$\omega(|ab|^p)^{1/p} = \tilde{\omega}(z |ab|^p)^{1/p} = 
\tilde{\omega}( |zazb|^p)^{1/p} = \| (za) (zb) \|_p \leq \| za \|_q \| zb \|_r,$$
while $\| za \|_q = \tilde{\omega}( |za|^q)^{1/q} = 
\omega(|a|^p)^{1/p}$, and so on.  
\end{proof}

\begin{remark} 
Item (2) improves the H\"older inequality from \cite{GPVT} to include  the case of $p < 1$.
\end{remark}

We recall 
from \cite{BLI} that a unital norm closed subalgebra $\A$ of a unital $C^*$-algebra $\mathcal{C}$ is {\em  logmodular}
 if every element $b \in \mathcal{C}$ such that $b \geq \epsilon 1$  for some $\epsilon > 0,$ is  a uniform limit of terms of the form $a^* a$ where 
 $a$ is an invertible element in $\A$.   See e.g.\ \cite{BK} for a very recent survey and paper 
 on logmodular algebras.  If $\A$ is  logmodular then $\mathcal{C}$ is the $C^*$-envelope of $\A$ by \cite[Proposition 4.3]{BLI}.
 In addition, by \cite[Theorem 4.4]{BLI} for any unital $C^*$-subalgebra  $\mathcal{D}$  and   
  $\mathcal{D}$-character $\Phi$ on $\A$, if there exists a positive map    
 $\Psi$ from $C$ onto $\D$  
extending $\Phi$ (which is necessarily a conditional  expectation), then it is unique.     
Of course the 
earlier Hoffman-Rossi  theorems can assist the existence of such an extension in some important cases.

The latter  is a generalization of the important {\em unique normal
state extension property} in the theory of noncommutative $H^\infty$ (a la Arveson's subdiagonal algebras)--see e.g.\ \cite{BLsurv,Arv}.
Indeed if $\mathcal{D} \subset B(H)$ then  there is a unique UCP map $\Psi : C \to B(H)$ extending $\Phi$,
and of course $\Psi(\A  + \A^*) \subset \mathcal{D}$.   So if further $\mathcal{C}$ is a von Neumann algebra with $\A + \A^*$ weak* dense in $\M$
(which is the case in Arveson's subdiagonal setting), and if 
$\Psi$ is weak* continuous, then $\Psi$ maps into $\D$, and by the above it is the unique positive extension of  $\Phi$ to $\M$. 
 Without the  logmodular assumption certainly there may exist no conditional  expectation 
 from $C$ onto $\mathcal{D}$, i.e.\ no `noncommutative representing measure' in the sense of this paper.  

\begin{theorem} Let $\A$ be a unital norm closed logmodular subalgebra of a unital $C^*$-algebra $\mathcal{C}$, and $\omega$ a state on $\mathcal{C}$ which is tracial on a $C^*$-subalgebra $\mathcal{D}$ of $\A$.  Let $\Phi$ be a $\omega$-preserving
$\mathcal{D}$-character on $\A$, which extends to an $\omega$-preserving conditional expectation
from  $\mathcal{C}$ onto $\mathcal{D}$.  For any $a\in \A^{-1}$, we  have 
$$\lim_{n\to\infty} \omega(|a|^{2^{-n}})^{2^n} \geq \lim_{n\to\infty} \omega(|\Phi(a)|^{2^{-n}})^{2^n}.$$
These limits are decreasing, and equal $\exp ( \omega( \log |a|))$ and $\exp ( \omega( \log |\Phi(a)|))$ respectively.  Thus we have
Jensen's equality 
$$\Delta_\omega(a) = \Delta_\omega(\Phi(a)), \; \qquad a\in \A^{-1},$$
and $\omega$ is a noncommutative Arens-Singer measure for $\Phi$. 
\end{theorem}

\begin{proof}
The proof uses induction. For any $a\in \A^{-1}\cup(\A^*)^{-1}$, we clearly have by the Kadison-Schwarz inequality that
$$\omega(|a|^2)=\omega(\Psi(|a|^2))\geq \omega(|\Phi(a)|^2)).$$ Now suppose that for any $a\in \A^{-1}\cup(\A^*)^{-1}$ we have that 
\begin{equation}\label{eq:a}
\omega(|a|^{2^{1-n}})^{2^{n-1}} \geq \omega(|\Phi(a)|^{2^{1-n}})^{2^{n-1}} \mbox{ for
all } 0\leq n\leq k.
\end{equation}
Here we interpret $\Phi(a^*) = \Phi(a)^*$ if $a \in A$.
We  show that this inequality also holds for $n=k+1$. So let $a\in \A^{-1}$ be given, and inductively define the
sequence $( x_m )  \subset \mathcal{C}_+$ by $$x_1 = |a|^{2^{1-k}}, \quad
x_{m+1} = \frac{1}{2}(x_m + |a|^{2^{1-k}}x_m^{-1}).$$Then $(x_{m+1})$ decreases monotonically in norm to 
$|a|^{2^{-k}}$ by the facts noted prior to Lemma \ref{lemma 3}. Now let $u_a$ be the partial isometry in the polar decomposition $a=u_a|a|$. Since $a\in \A^{-1}$, we have that $u_a$ is a unitary in $\mathcal{C}^{**}$, and $u_a^* a = |a|$. 
Several times we will use silently the identity $|x^{-1}| = |x^*|^{-1}$, which is true since 
$$|x^{-1}|^2 =(x x^*)^{-1} = |x^*|^{-2}.$$  We will also use the operator theoretic fact 
(perhaps originally due to Macaev and Palant) that $\| x_n^{\frac{1}{2}} - y^{\frac{1}{2}} \| \leq 
2 \| x_n - y \|^{\frac{1}{2}}$ for positive operators $x_n, y$.   Hence  if $x_n \to x$ then $x_n^{\frac{1}{2}} \to y^{\frac{1}{2}}$.
If $z_n \to z$ in $B(H)$ then $z_n^* z_n  \to z^* z$, and hence $w^* z_n^* z_n  w \to w^*  z^* z w$ for any operator $w$ on $H$.
Thus by the last fact,  $|z_n w | \to |zw|$.

Because $\A$ is logmodular, we may for each fixed $m$, select a
sequence $( z_l^{(m)} ) \subset \A^{-1}$ with $$\lim_{l\to\infty} |(z_l^{(m)})^{-1}| = |x_m^{-2^{k-1}}u_a^*|.$$
Note that $x_m^{-2^{k-1}}u_a^* = (u_a x_m^{-2^{k-1}})^* \in C$, since $x_m^{-2^{k-1}} \in C^*(|a|)$.    
For this sequence we have that 
$$\lim_{l\to\infty}|(z_l^{(m)})^{-1}a|=\lim_{l\to\infty}|\,|(z_l^{(m)})^{-1}|a|=|\,|x_m^{-2^{k-1}}u_a^*|a|=|x_m^{-2^{k-1}}u_a^*a|=x_m^{-2^{k-1}}|a|.$$
Observe that we then also have (since all of our operators are bounded below) that
$$\lim_{l\to\infty}|(z_l^{(m)})^*|^2=\lim_{l\to\infty}|(z_l^{(m)})^{-1}|^{-2}=|x_m^{-2^{k-1}}u_a^*|^{-2}=|u_ax_m^{2^{k-1}}|^2=(x_m^{2^{k-1}})^2,$$ and hence that $\lim_{l\to\infty}|z_l^{(m)}|=x_m^{2^{k-1}}$. It is clear from the induction hypothesis and Lemma \ref{lemma 3}, that
\begin{eqnarray}\label{eqn-aa}
\frac{1}{2}\omega(|(z_l^{(m)})^*|^{2^{1-k}} + |(z_l^{(m)})^{-1}a|^{2^{1-k}})&\geq& \frac{1}{2}\omega(|\Phi(z_l^{(m)})^*|^{2^{1-k}} + |\Phi((z_l^{(m)})^{-1}a)|^{2^{1-k}})\nonumber\\
&=& \frac{1}{2}\omega(|\Phi(z_l^{(m)})|^{2^{1-k}} + |\Phi(z_l^{(m)})^{-1}\Phi(a)|^{2^{1-k}}).
\end{eqnarray}
Recall that $\omega$ is a trace on $\mathcal{D}$. Thus the $p$-seminorms on $\mathcal{D}$ defined in terms of this trace, satisfy a  
H\"{o}lder inequality. Now since $2^{k-1} + 2^{k-1} = 2^k$ we may apply this H\"{o}lder
inequality with $p =(\frac{1}{2})^{k}$ and $q = r = (\frac{1}{2})^{k-1}$, to see that
$$\omega(|\Phi(a)|^{2^{-k}})^{2^{k}} \leq \omega(|\Phi(z_l^{(m)})|^{2^{1-k}})^{2^{k-1}}
\omega(|\Phi(z_l^{(m)})^{-1}\Phi(a)|^{2^{1-k}})^{2^{k-1}}.$$This
in turn translates to 
$$\omega(|\Phi(z_l^{(m)})^{-1}\Phi(a)|^{2^{1-k}}) \geq \frac{\omega(|\Phi(a)|^{2^{-k}})^2}{\omega(|\Phi(z_l^{(m)})|^{2^{1-k}})}.$$ 
If we combine the above with inequality (\ref{eqn-aa}), it will then follow from the facts noted prior to Lemma \ref{lemma 3} that 
\begin{eqnarray*} 
\frac{1}{2}\omega(|(z_l^{(m)})^*|^{2^{1-k}} + |(z_l^{(m)})^{-1}a|^{2^{1-k}})&\geq& \frac{1}{2}
 \left[\omega(|\Phi(z_l^{(m)})|)^{2^{1-k}} + \frac{\omega(|\Phi(a)|^{2^{-k}})^2}{\omega(|\Phi(z_l^{(m)})|^{2^{1-k}})}\right]\\ 
&\geq& \omega(|\Phi(a)|^{1/2^{k}}), 
\end{eqnarray*} 
the last inequality following from the AGM inequality. 
This fact alongside the above limit formulas, now ensures that 
\begin{eqnarray*} 
\omega(x_{m+1}) &=& \frac{1}{2}\omega(x_m + x_m^{-1}|a|^{2^{1-k}})\\ 
&=& \frac{1}{2}\lim_{l\to\infty}\omega(|(z_l^{(m)})^*|^{2^{1-k}} + |(z_l^{(m)})^{-1}a|^{2^{1-k}})\\ 
&\geq& \omega(|\Phi(a)|^{1/2^{k}}). 
\end{eqnarray*} 
Since $(x_{m+1})$ decreases uniformly and monotonically to $|a|^{1/2^{k}}$, it follows that 
$$\omega(|a|^{1/2^{k}}) \geq \omega(|\Phi(a)|^{1/2^{k}})\mbox{ for all }a\in \A^{-1}.$$ 
It is an exercise to see that if $\A$ is logmodular, then so is $\A^*$ (see the remark before  \cite[Proposition 4.3]{BLI}).  
Hence the same proof with the roles of $\A$ and 
 $\A^*$ reversed, shows that we also have that 
$$\omega(|a|^{1/2^{k}}) \geq \omega(|\Phi(a)|^{1/2^{k}})\mbox{ for all }a\in (\A^*)^{-1},$$    
as desired.  

Lemma \ref{lemma 3} shows that the limits are decreasing.  In fact $\omega(|a|^{p})$ is decreasing as $p \searrow 0$, as
one may see by noting  that $\omega$ restricts to a (tracial) state on  $C^*(1,|a|)$.    We may thus reduce 
to the case of a positive invertible function $f$  in $C(K)$, with $\omega$ becoming a probability measure $\mu$ on $K$.  In this case it is
well known that $\| f \|_p \searrow \exp (\int_K \, \ln |f| \, d \mu)$ (see e.g.\ \cite{Gam}).  
By the proof of \cite[Lemma 4.3.6]{Arv} we have that $\omega(|a|^{1/2^{k}})^{2^k} \to \exp ( \omega( \log |a|))$.
The matching assertions for $\Phi(a)$ are better known, or are similar.  Thus we have shown that 
$$\Delta_\omega(a) \geq \Delta_\omega(\Phi(a)), \; \qquad a\in \A^{-1},$$
as desired.

Replacing $a$ by $a^*$ we have $\omega(|a^*|^{1/2^{k}}) \geq \omega(|\Phi(a)^*|^{1/2^{k}})$, and letting 
$k \to \infty$ yields $\exp ( \omega( \log |a^*|)) \geq \exp ( \omega( \log |\Phi(a)^*|))$.
Then replacing $a$ by $a^{-1}$  gives $\exp ( \omega( \log |a|^{-1}))  \geq \exp ( \omega( \log |\Phi(a)|^{-1})))$.
By the functional calculus $\log |a|^{-1} = - \log |a|$ since $|a|$ is bounded below. 
We deduce that $\Delta_\omega(a) \leq \Delta_\omega(\Phi(a))$, giving Jensen's equality.
\end{proof}

\begin{remark} 
If $\tau$ is a trace on $\mathcal{D}$, and $\Psi :  \mathcal{C} \to \mathcal{D}$ is a  positive extension (`noncommutative representing measure') 
of a $\mathcal{D}$-character $\Phi$   on $\A$ 
then $\omega = \tau \circ \Psi$ is a  state on $\mathcal{C}$ which is tracial on $\mathcal{D}$, and of course $\omega \circ \Psi = \omega$).    
Also, if $\rho$ is a  state on $\A$ which is tracial on $\mathcal{D}$,
and which is preserved by $\Phi$, set $\tau = \rho_{\mathcal{D}}$ and  $\omega = \tau \circ \Psi$.  Then $\Psi$ is $\omega$-preserving, and $\omega(a) = \rho(\Psi(a) ) = \rho(\Phi(a) ) = \rho(a)$ for $a \in A$, so $\omega$ extends
$\rho$.   In these cases $\omega$ (and $\Psi$) satisfy the conditions of the  theorem.  
\end{remark}

{\em Acknowledgements.}  We thank Roger Smith and Mehrdad Kalantar for  assistance with the von Neumann algebraic 
facts described in the Remark above Lemma \ref{aves2}.     We thank the referee for their careful reading, and thank David Sherman for some conversations.

\end{document}